\documentclass[11pt]{amsart}

\usepackage{amsmath,amssymb,amsthm}

\usepackage{amsmath, amssymb}
\usepackage{amsthm, amsfonts, mathrsfs}
\usepackage{fullpage}
\usepackage{amsfonts,graphicx}
\usepackage{amsmath,amssymb,graphicx,epsfig,epic,mathrsfs}

\usepackage{hyperref}
\makeatletter
\@addtoreset{equation}{section}
\makeatother

\theoremstyle{plain}
\newtheorem{thm}{Theorem}
\newtheorem{lem}{Lemma}
\newtheorem{cor}{Corollary}
\newtheorem{prop}{Proposition}

\theoremstyle{definition}
\newtheorem{defi}{Definition}

\theoremstyle{remark}
\newtheorem{Rema}{Remark}
\newtheorem*{rema*}{Remark}
\newtheorem{Remas}{Remarks}

\newcommand{\ww}{w_{\varepsilon,\varepsilon'}}
\newcommand{\NN}{\mathbb{N}}
\newcommand{\RR}{\mathbb{R}}
\newcommand{\ZZ}{\mathbb{Z}}
\newcommand{\oneover}[1]{\frac{1}{#1}}
\DeclareMathOperator{\divergence}{div}
\newcommand{\restr}[2]{\ensuremath{\left. #1 \right|_{#2}}}

\newcommand{\www}{\tilde{\Omega}}

\newcommand{\diver}[1]{ \textnormal{div}\hspace{0.07cm} #1}

\newcommand{\cepsilon}{c _\varepsilon}
\newcommand{\cepsilonq}{c _{\varepsilon,q}}

\newcommand{\cepsiloninit}{c _{0, \varepsilon} }

\newcommand{\Omegaepsilon}{\Omega _\varepsilon }

\newcommand{\vepsilon}{v _\varepsilon }
\newcommand{\vepsilonq}{v _{\varepsilon,q} }

\newcommand{\vepsiloninit}{v _{0, \varepsilon} }

\author[T. Hmidi]{Taoufik Hmidi}
\address{IRMAR, Universit\'e de Rennes 1\\ Campus de
Beaulieu\\ 35~042 Rennes cedex\\ France}
\email{thmidi@univ-rennes1.fr}
\date{}

\begin{document}
\title{ Low Mach number limit for the isentropic Euler system with   axisymmetric
   initial data}

\date{\today}

\maketitle

\begin{abstract}
  This paper is devoted to the study of  
   the low Mach number limit   for the isentropic Euler system with axisymmetric initial data without swirl. In the first part of the paper we analyze the problem corresponding to the subcritical  regularities, that is $H^s$ \mbox{with $s>\frac52$.} Taking advantage of the Strichartz estimates and using the special structure of the vorticity we show that the \mbox{lifespan $T_\varepsilon$} of the solutions is  bounded below by $\log\log\log\frac1\varepsilon$, where $\varepsilon$ denotes the  Mach number.  Moreover, we prove that the incompressible parts  converge to the solution of the incompressible Euler system, when  the parameter 
   $\varepsilon$ goes to zero. In the second part of the paper we address  the same problem but for the Besov critical regularity  $B_{2,1}^{\frac52}$. This case turns out  to be  more subtle at least due to two facts. The first one is related to  the   Beale-Kato-Majda criterion which  is not known to be valid for  rough regularities.   The second one concerns the  critical aspect of the Strichartz estimate $L^1_TL^\infty$ for the acoustic parts $(\nabla\Delta^{-1}\diver\vepsilon,\cepsilon)$: it scales in   the space variables like the space of the initial data. 
%   To get a uniform decay for high frequencies and use the interpolation method we start with the simple but very important observation that any element of a given Besov space $B_{p,r}^s, $ with $s\in\RR, p\in[1,+\infty]$ and $r\in[1,+\infty[$ is more regular than its prescribed regularity and belongs to some heterogeneous Besov space $B_{p,r}^{s,\Psi}.$  
    \end{abstract}

\tableofcontents

\section{Introduction}
\label{sec:introduction}

The object of this paper is to  study the incompressible limit problem for classical solutions  of the compressible isentropic Euler equations. The fluid is assumed to evolve  in the whole \mbox{space $\RR^3$} and it possesses  some special geometric property: the flow is invariant under the group of rotations around the vertical axis $(Oz)$.  The state of the fluid is described by the velocity field
$\vepsilon $ and the sound speed $ \cepsilon $, through a penalized quasilinear hyperbolic system, 

\begin{equation}
   \label{eqs:3}
    \left\{
     \begin{array}{l}
       \partial _t \vepsilon  + \vepsilon \cdot \nabla
       \vepsilon +\bar\gamma \cepsilon \nabla \cepsilon +
       \oneover{ \varepsilon } \nabla \cepsilon
       =0 \\[0.3ex]
       \partial _t \cepsilon + \vepsilon \cdot \nabla
       \cepsilon + \bar\gamma\cepsilon \divergence
       \vepsilon + \oneover{ \varepsilon } \divergence \vepsilon
       =0 \\[0.3ex]
       \restr{ (\vepsilon , \cepsilon) }{ t=0 }
       = (\vepsiloninit, \cepsiloninit ),
     \end{array}
   \right.
\end{equation}
with $\bar\gamma$ a strictly positive number and $\varepsilon$ a small parameter called the Mach number.
We point out that the derivation of this model can be done    from the  compressible isentropic equations after rescaling the time   and changes of variables, see for instance \cite{MR1975081,MR84d:35089,MR2002d:76095}. This model has been widely considered through the last decades and  a special attention is focused on the construction  of a  family of solutions with a non degenerate time existence. Nevertheless, the most relevant problem is to   study rigorously the convergence towards the incompressible Euler equations when the Mach number goes to zero. We recall that the incompressible Euler  system is given by,
\begin{equation}
   \label{eq:45}
   \left\{
     \begin{array}{l}
       \partial _t v + v \cdot \nabla v + \nabla p=0 \\
       \divergence v =0 \\
       \restr{ v }{ t=0 } = v _0 .
     \end{array}
   \right.
\end{equation}
 The answer to these problems depends on several factors: the domain where the fluid is assumed to evolve: the full space $\RR^d,$ the torus $\mathbb{T}^d$, bounded or unbounded domains. The second factor is the state of the initial data: whether they are well-prepared or not. In the well-prepared case \cite{MR84d:35089,klain-majda82}, we assume
that the initial data are slightly compressible, which means that  $\divergence \vepsiloninit =O(\varepsilon)$ \mbox{and
$\nabla \cepsiloninit =O(\varepsilon)$} as $\epsilon \rightarrow 0 $. However, in
the ill-prepared case \cite{MR87h:35279}, we only assume that the family $(\vepsiloninit,\cepsiloninit)_{\varepsilon} $ 
 is bounded in some Sobolev spaces $H^s$ with $ s>\frac{d}{2}+1$ and that the incompressible parts of $(\vepsiloninit)_{\varepsilon} $ tend to some field $ v _0 $.  Remark that in the well-prepared case we have a uniform bound of $(\partial_t\vepsilon)_{\varepsilon},$ and this allows to pass to the limit by using Aubin-Lions compactness lemma. Nevertheless, the ill-prepared case is more subtle because the time derivative  $\partial_t\vepsilon,$ is of size $O(\frac{1}{\varepsilon})$. To overcome this difficulty, Ukai \cite{MR87h:35279} used  the  dispersive effects generated by the acoustic waves in order to prove that the compressible part of the velocity and the acoustic term vanish when $\varepsilon$ goes to zero.   
%The regularity of the initial  data can generate additional difficulty especially when we work with rough initial data. 
In \cite{MR1975081}, we deal with a more degenerate case, we allow the initial data to be so ill-prepared that
corresponding solutions can tend to a vortex patch or even to a Yudovich
solution. Since these solutions do not belong to the Sobolev space $H
^{s}$ for any $s>2$, we  allow the initial data that are not
uniformly bounded in these spaces. 
%In this context, the dispersive
%effects are strong enough to deal with such particularly ill-prepared
%initial data. Strichartz estimates, which reflect the phenomenon of
%dispersion of acoustic waves, are indeed efficient tools to deal with
%ill-prepared initial data when the fluid is extended in the whole
%space; this was shown previously in the presence of a term of
%viscosity \cite{MR1886005,MR2000h:76144} as well as in its absence
%\cite{MR89d:76021,MR87h:35279}. 
We  point out that this problem has already been studied in numerous
papers, see for instance 
\cite{MR89d:76021,MR1886005,MR2000h:76144,MR2000j:35220,MR84d:35089,klain-majda82,MR1602773,MR99d:76084,MR85e:35077,MR2002d:76095,MR87h:35279}.

It is well-known that contrary to the incompressible Euler system, the equations \eqref{eqs:3} develop in space dimension two singularities in finite time and for some smooth initial data, see \cite{Ram}. This phenomenon holds true for higher dimension, see \cite{Sid}. On the other hand it was shown in  \cite{klain-majda82} that for the whole space or for the torus domain when the limit system \eqref{eq:45} exists for some \mbox{time interval $[0,T_0]$} then for ever $T<T_0$ there exists $\varepsilon_0>$ such that  for $\varepsilon<\varepsilon_0$ the solution to \eqref{eqs:3} lives till the time $T$.  This result was established for the well-prepared case and with sufficient smooth initial data, that is, in  $H^s$ and $s>\frac{d}{2}+2.$ We point out that this approach can not be achieved for $\frac{d}{2}+2>s>\frac{d}{2}+1$.  The proof is based on the perturbation theory by estimating in a suitable way the difference between the two systems. As an application  in space dimension two, if we denote by $T_\varepsilon$ the time lifespan of the solution $(\vepsilon,\cepsilon)$ then $\lim_{\varepsilon\to0}T_\varepsilon=+\infty.$  We observe that we can give an explicit lower bound for the lifespan $T_\varepsilon\ge C\log\log\frac1\varepsilon.$  In \cite{MR1975081}, we extend these results  in the case of Yudovich solutions. 

In this paper, we  try to accomplish the same program in dimension three for axisymmetric initial data. This is motivated by the works of \cite{Taira,Ukhovskii} where it is proven that the incompressible Euler system  is globally well-posed when the initial data is axisymmetric and belongs to $H^s, s>\frac52$. For the definition of the axisymmetry, see Definition \ref{defi1}. The proof relies on the special structure of the vorticity $\Omega$ which leads to a global bound of $\|\Omega(t)\|_{L^\infty}$ and then we  use the  Beale-Kato-Majda criterion \cite{Beale}.   Recently,  we established   in \cite{AHK} the global well-posedness for \eqref{eq:45}  with initial data lying in Borderline Besov \mbox{spaces $v_0\in B_{p,1}^{\frac3p+1}, 1\le p\le \infty$.} It is important to mention that in this context  the  Beale-Kato-Majda criterion is not known to be valid and  the geometry is crucially used in different steps of the proof and it is combined with a  dynamical interpolation method.

Our main goal here is to study the incompressible limit problem for both subcritical and critical cases with ill-prepared axisymmetric initial data.  Concerning the subcritical regularities we obtain the  following result.
\begin{thm}\label{thm1}
Let $s>\frac52$ and $\{(\vepsiloninit,\cepsiloninit)_{0<\varepsilon\le1}\}$ be a $H^s$-bounded family of axisymmetric initial data, that is
$$
\sup_{0<\varepsilon\le1}\|(\vepsiloninit,\cepsiloninit)\|_{H^s}<+\infty.
$$
Then the system \eqref{eqs:3} has a unique solution $(\vepsilon,\cepsilon)\in C([0,T_\varepsilon[; H^s),$ with
$$
T_\varepsilon\geq C\log\log\log(\frac1\varepsilon):=\tilde T_\varepsilon.
$$
The constant $C$ does not depend on $\varepsilon.$ Moreover, the compressible and acoustic parts of the solutions tend
   to zero: there exists $\sigma>0$ such that
   \begin{equation*}
   \|(\diver\vepsilon,\nabla \cepsilon)\|_{L^1_{\tilde T_\varepsilon} L^\infty}\le C_0\varepsilon^\sigma.
        \end{equation*}

Assume in addition that
the incompressible parts $(\mathcal{P}\vepsiloninit)$ converge in $L^2$ to some $v_0.$ Then 
   the incompressible parts of the solutions tend to the Kato's 
   solution $v$ of the system $(\ref{eq:45})$:
   \begin{equation*}
     \mathcal{P} \vepsilon \rightarrow v
     \quad \text{ in } L^\infty_{loc} (\RR ^{ +} ; L ^2 ).
   \end{equation*}
   \end{thm}
   \begin{Rema}
   To study the lifespan of the solutions we do not use the approach of \cite{klain-majda82} based on the stability  of the incompressible Euler system. More precisely, it seems that there is no need to use the limit system: we can only work with \eqref{eqs:3} and use the special structure of the vorticity dynamics. On the other hand, the fact of working with the vorticity and by taking advantage of its special structure we can  improve the regularity required for the stability:   we can work in the framework of Soboev spaces $H^s$ with $ s>\frac52.$ We emphasize that in \cite{klain-majda82} the Sobolev regularity must be larger than $\frac72.$
\end{Rema}

The proof of the Theorem \ref{thm1} relies on the use of Strichartz estimates for   the compressible parts $(\mathcal{Q}\vepsilon)$ and  the acoustic ones $(\cepsilon)$, see Corollary \ref{cor11}.  Thus interpolating this result with the energy estimates, see \mbox{Proposition \ref{energy1}}, we  obtain the following result described in \mbox{Proposition \ref{stt1}:} there exists $\sigma>0$ such that
\begin{equation}\label{hdd1}
\|\divergence\vepsilon\|_{L^1_TL^\infty}+\|\nabla\cepsilon\|_{L^1_TL^\infty}\le C_0\varepsilon^{\sigma}(1+T^{2}) e^{CV_\varepsilon(T)},\quad V_\varepsilon(T)=\|(\nabla\vepsilon,\nabla\cepsilon)\|_{L^1_TL^\infty}.
\end{equation} 
It is worthy pointing out that  working with the subcritical regularities is very precious to get the preceding  inequality. However, this argument fails in the critical spaces as we shall see next.
The second ingredient of the proof is the use of  the special structure of the vorticity in the axisymmetric case combined with Beale-Kato-Majda criterion. In what follows, we will briefly discuss the main feature of the axisymmetric flows and for the complete computations, see  the next sections. By the definition,  the velocity takes the form $v(x)=v^r(r,z) e_r+v^z(r,z) e_z$  in the cylindrical basis and consequently the vorticity $\Omega:=\hbox{curl }v$ is given by $\Omega=(\partial_z v^r-\partial_r v^z)e_\theta:=\Omega^\theta e_\theta.$ Therefore, the vorticity dynamics is described by
\begin{equation}\label{dynas1}
\partial_t\Omega_\varepsilon+v_\varepsilon\cdot\nabla\Omega_\varepsilon+\Omega_\varepsilon\divergence \vepsilon=\frac{v^r_\varepsilon}{r}\Omega_\varepsilon.
\end{equation}
It follows that the quantity $\frac{\Omega_\varepsilon}{r}$ satisfies the  transport equation:
$$
\big(\partial_t+v_\varepsilon\cdot\nabla+\divergence \vepsilon\big)\frac{\Omega_\varepsilon}{r}=0.
$$
We observe that this is  analogous  to the vorticity structure for the compressible Euler equations  in space dimensions two.
Performing energy estimates we obtain  an almost conservation laws:  
$$
\Big\|\frac{\Omega_\varepsilon}{r}(t)\Big\|_{L^p}\le \Big\|\frac{\Omega_{0,\varepsilon}}{r}\Big\|_{L^p}e^{(1-\frac1p)\|\divergence\vepsilon\|_{L^1_tL^\infty}},\quad \forall\,p\in[1,\infty].
$$
When the fluid is incompressible  we obtain  exact conservation laws which are sufficient to lead to the global well-posedness. 
Concerning the proof of the incompressible limit in the case of ill-prepared initial data, it is done in a straightforward manner by using  the Strichartz estimates. 
%\begin{thm}\label{thm2}
%Under the hypotheses of Theorem \ref{thm1},  we assume in addition that
%the incompressible parts $(\mathcal{P}\vepsiloninit)$ converge in $L^2$ to $v_0 (\in H^s).$ Then 
%   the incompressible parts of the solutions tend to the Kato
%   solution of the system $(\ref{eq:45})$:
%   \begin{equation*}
%     P \vepsilon \rightarrow v
%     \quad \text{ in } L^\infty_{loc} (\RR ^{ +} ; L ^2 ).
%   \end{equation*}
%   Finally, the compressible and acoustic parts of the solutions tend
%   to zero:
%   \begin{equation*}
%     ( \vepsilon -P \vepsilon, \cepsilon )
%     \rightarrow 0 \quad
%     \text{ in }L^\infty_{loc} ( \RR ^{ + } ; Lip ).
%   \end{equation*}

%
%\end{thm}

In the second part of this paper we will focus on the low Mach number limit for initial data with  critical regularities. In our context a function space $\mathcal{X}$
 is called critical if it is embedded in Lipschitz class and both have the same scaling, for example Besov space $B_{2,1}^{\frac52}$ or more generally $B_{p,1}^{\frac3p+1}$, with $p\in[1,\infty].$ These spaces are till now the largest spaces for the local well-posedness to the incompressible Euler equations or more generally for the quasilinear hyperbolic systems of order one. In \cite{AHK}, it is proven that the system \eqref{eq:45} is globally well-posed for axisymmetric initial data with critical Besov regularity.   Therefore  it is legitimate to try to accomplish the same program for the system \eqref{eqs:3} as in the subcritical case and especially to quantify a lower bound for the lifespan of the solutions. Nevertheless, to get unfiorm bounds with respect to the parameter $\varepsilon$ and remove the penalization term we need to work with critical spaces which are constructed over  the Hilbert space $L^2$ \mbox{like  $B_{2,1}^{\frac52}$} or  other modified spaces as we will see next.   Although one can prove the local well-posedness for the system \eqref{eqs:3} with a uniform time existence, the extension of the results of Theorem \ref{thm1} to the critical case seems to be much more relevant.  We distinguish  at least two principal  difficulties. The first one has a connection with Strichartz estimates  \eqref{hdd1}: in the critical framework the quantities $\|\diver\vepsilon\|_{L^\infty}$ \mbox{and $\|\vepsilon\|_{B_{2,1}^{\frac52}}$} have the same scaling and the interpolation argument used in the proof of  \mbox{Theorem \ref{thm1}} can not work without doing refined improvement.  Indeed, we have no sufficient information about the decay of the remainder series $\sum_{q\geq N}2^{\frac52 q}\|(\Delta_q\vepsilon,\Delta_q\cepsilon)(t)\|_{L^2}$ and there is no   explicit dependence of this decay  with respect to the parameters, $N, t$ and $\varepsilon:$ it seems that  in general the number $N$ can implicitly depend on the variable time and on the parameter $\varepsilon$ and this makes the task so hard. To overcome  this difficulty, we start with the important observation that  any function $f\in B_{2,1}^{\frac52}$ belongs to some heterogeneous  Besov spaces $B_{2,1}^{\frac52,\Psi}$, where $\Psi:[-1,+\infty[\to \RR_+^*$ is  a nondecreasing function depending on the profile of $f$ and satisfying $\lim_{q\to+\infty}\Psi(q)=+\infty$. This latter  space  is defined by the norm 
 $$
\|u\|_{B_{2,1}^{\frac52,\Psi}}=\sum_{q\geq-1}\Psi(q)2^{\frac52 q}\|\Delta_q u\|_{L^2}.
$$
Further details and more discussions  about these spaces will be found in the next section. We point out that 
this function $\Psi$ measures the decay of the remainder series in the space of the initial data and  we will see that the same decay will occur for the solution uniformly with respect to $\varepsilon$. Therefore one can set up the  interpolation argument in the critical framework in the similar way to  the subcritical case but without any explicit dependence of the lifespan  with respect to the initial data.
%given in Definition \ref{heter} but  at this step of the presentation we restrict ourselves to the  special case of Besov spaces  $B_{2,1}^{\frac52,\alpha}, \alpha>0$ corresponding to the choice $\Psi(x)=(2+x)^\alpha.$ These spaces  are slightly more regular but very close to  $B_{2,1}^{\frac52}$ and their norms are given by,
% $$
%\|u\|_{B_{2,1}^{\frac52,\alpha}}=\sum_{q\geq-1}(q+2)^\alpha 2^{\frac52 q}\|\Delta_q u\|_{^2}.
%$$
%We observe that for $\alpha=0$  this space reduces to the usual Besov space $B_{2,1}^{\frac52}$.

Concerning the second difficulty, it is related to the Beale-Kato-Majda criterion which is not applicable in the context of critical regularities. In this case the estimate of  $\|\Omegaepsilon(t)\|_{L^\infty}$ is not sufficient to propagate the initial regularities and though one should estimate $\|\Omegaepsilon(t)\|_{B_{\infty,1}^0}$  instead.  We will see next  that many geometric properties of the axisymmetric flows are used  and play a central role in the critical framework. Our result reads as follows.
\begin{thm}\label{thm11}
Let $\{(\vepsiloninit,\cepsiloninit)_{0<\varepsilon\le1}\}$ be  family of axisymmetric initial data such that 
\begin{equation}\label{strong11}
\sum_{q\geq-1}2^{\frac52 q}\sup_{0<\varepsilon\le1}\|(\Delta_q\vepsiloninit,\Delta_q\cepsiloninit)\|_{L^2}<+\infty.
\end{equation}
Then the system \eqref{eqs:3} has a unique solution $(\vepsilon,\cepsilon)\in C([0,T_\varepsilon[; B_{2,1}^{\frac52}),$ with
$$
\lim_{\varepsilon\to0}T_\varepsilon=+\infty.
$$
Moreover the  acoustic parts of the solutions go
   to zero: for every $T<T_\varepsilon$ we have
     \begin{equation*}
   \lim_{\varepsilon\to 0}\|(\diver\vepsilon,\nabla \cepsilon)\|_{L^1_{T} L^\infty}=0.
           \end{equation*}
Assume in addition that
the incompressible parts $(\mathcal{P}\vepsiloninit)$ converge in $L^2$ to some $v_0.$ Then 
   the incompressible parts of the solutions tend to the Kato's 
   solution $v$ of the system $(\ref{eq:45})$:
   \begin{equation*}
     \mathcal{P} \vepsilon \rightarrow v
     \quad \text{ in } L^\infty_{loc} (\RR ^{ +} ; L ^2 ).
   \end{equation*}
  \end{thm}
  We shall make some useful remarks.
\begin{Remas}{\it 
\begin{enumerate}
\item  In the preceding theorem we need the assumption \eqref{strong11} which  is much stronger than $\sup_{0<\varepsilon\le1}\|(\vepsiloninit,\cepsiloninit)\|_{B_{2,1}^{\frac52}}<+\infty.$ This is crucial to prove a uniform decay for higher frequencies, see Corollary \ref{corza1}.
%\item In the preceding Theorem the  lower bound of the lifespan  grows more slowly than the one found in Theorem \ref{thm1}. This a consequence of the fact that we need a Strichartz estimate of type $\|(\diver\vepsilon,\nabla\cepsilon)\|_{L^1_TL^\infty}$ which is critical in the sense that it scales in the space variables like the space of initial data.
\item
This  theorem is a special case of a general result that will be discussed in Theorem \ref{thm111}. More precisely, we can extend these results for initial data lying in the heterogeneous Besov spaces $B_{2,1}^{\frac52,\Psi},$ with $\Psi \in\mathcal{U}_\infty,$ see the definition \ref{heter}. When $\Psi$ has  a slow growth at infinity then the lifespan  $T_\varepsilon$ of the solutions is bounded below as follows 
$$
T_\varepsilon \geq C_0\log\log\big|\log{\Psi\big(\log\varepsilon^{-1}\big)}\big|. 
$$
This covers the results of Theorem \ref{thm1}: it suffices to choose  $\Psi(x)=2^{(s-\frac52)x}, s>\frac52.$
\end{enumerate}
}
\end{Remas}
We have  already mentioned that in the critical case  the Belae-Kato-Majda criterion is out of use,  and yet the significant quantity that one should estimate is  $\|\Omegaepsilon(t)\|_{B_{\infty,1}^0}.$ This does not seem  an  easy task due to the nonlinearities in the vorticity equation and  to the lack of the incompressibility of the velocity vector field. It is worthy pointing out that the incompressible case corresponding to the constraint  $\diver \vepsilon=0$ was studied  few years ago  in \cite{AHK} and where the following linear growth was established,
$$
\|\Omegaepsilon(t)\|_{B_{\infty,1}^0}\le C \|\Omegaepsilon(0)\|_{B_{\infty,1}^0} e^{\|\vepsilon^r/r\|_{L^1_tL^\infty}}\Big(1+\int_0^t\|\nabla \vepsilon(\tau)\|_{L^\infty}d\tau  \Big).
$$
One of the main technical parts of this paper is to extend this result for the compressible \mbox{model \eqref{dynas1}.} What we are able to prove is the following
$$
\|\Omegaepsilon(t)\|_{B_{\infty,1}^0}\le C\|\Omegaepsilon(0)\|_{B_{\infty,1}^0}e^{\|\vepsilon^r/r\|_{L^1_tL^\infty}}\Big(1+e^{C\| \vepsilon\|_{L^1_t\textnormal{Lip}}}\|\diver{ \vepsilon}\|_{ L^1_tB_{p,1}^{\frac3p}}^2\big)\Big(1+\int_0^t\|\vepsilon(\tau)\|_{\textnormal{Lip}}d\tau\Big).
$$
where $p\in[1,\infty[$ and $\|\cdot\|_{\textnormal{Lip}}$ stands for Lipschitz norm. We observe that when we \mbox{take $\diver\vepsilon:=0$} then we get the previous linear estimate. The proof of this result uses, but with important modifications,  the approach developed in \cite{AHK} for the incompressible case. This method is based on a suitable splitting of the vorticity and the use of the  dynamical interpolation technics.  The geometry of axisymmetric flows plays a crucial role in the proof,  and we use also some tools of   paradifferential calculus and harmonic analysis. 

The rest of the paper is organized as follows. In section 2, we recall some functional spaces and some of their basic properties. Section 3 is devoted to the establishment of some  energy estimates  in the heterogeneous Besov spaces $B_{2,1}^{s,\Psi}.$ In section 4, we prove some useful Strichartz estimates for the compressible and acoustic parts of the fluid. In section 5, we discuss some basic notions of axisymmetric geometry and we study some important geometric properties of the vorticity and prove some a priori estimates. In section 6, we prove Theorem \ref{thm1}. Section 7 is deserved to  the proof of Theorem \ref{thm11} and especially to the proof of the logarithmic estimate previously described. Finally, in section 8 we establish some elementary lemmas. 
\section{The functional tool box}
    In this section we review some of the basic tools of the paradifferential calculus and recall some elementary  properties of Besov and Lorentz spaces. Before going further into the details we will give some notations that we will intensively use in  this work. 
    
    {\it Notations:} 
    
    $\bullet$  Throughout this paper, $C$ stands for some real positive constant which may be different in each occurrence and $C_0$ a constant which depends on the initial data.
    
   $\bullet$ We shall sometimes alternatively use the notation $X\lesssim Y$ for an inequality of type $X\leq CY$ where $C$ is a constant independent of $X$ and $Y$.

$\bullet$ For any pair of operators $P$ and $Q$ acting in the same  Banach space $\mathcal{X}$, the commutator $[P,Q]$ is given by $PQ-QP$.

\subsection{Littlewood-Paley Theory}
Hereafter, the space dimension is fixed $d=3$ and all the results of this section are valid for any dimension despite we make if necessary the suitable  modifications. To define Littlewood-Paley operators we need to recall  the dyadic partition of the unity, for a proof see for instance \cite{MR2000a:76030}:
there exist two positive radial  functions  $\chi\in \mathcal{D}(\mathbb R^3)$ and  
$\varphi\in\mathcal{D}(\mathbb R^3\backslash{\{0\}})$ such that
\begin{itemize}
\item[\textnormal{i)}]
$\displaystyle{\chi(\xi)
+\sum_{q\geq0}\varphi(2^{-q}\xi)=1}$\; 
$\forall\xi\in\mathbb R^3,$
\item[\textnormal{ii)}]
$\displaystyle{\sum_{q\in\mathbb Z}\varphi(2^{-q}\xi)=1}$\;
if\; $\xi\neq 0,$
\item[\textnormal{iii)}]
$ \textnormal{supp }\varphi(2^{-p}\cdot)\cap
\textnormal{supp }\varphi(2^{-q}\cdot)=\varnothing,$ if  $|p-q|\geq 2$,\\

\item[\textnormal{iv)}]
$\displaystyle{q\geq1\Rightarrow \textnormal{supp}\chi\cap \textnormal{supp }\varphi(2^{-q})=\varnothing}$.
\end{itemize}
For every $u\in{\mathcal S}'(\mathbb R^3)$ we define the dyadic blocks by,
$$
\Delta_{-1}u=\chi(\hbox{D})u;\quad \forall
q\in\mathbb N,\;\Delta_qu=\varphi(2^{-q}\hbox{D})u\; \quad\hbox{and}\quad
S_qu=\sum_{-1\leq j\leq q-1}\Delta_{j}u.
$$
One can easily prove that for every tempered distribution $u,$ the following identity holds true  in the weak sense,
\begin{equation}\label{dr2}
u=\sum_{q\geq -1}\Delta_q\,u.
\end{equation}
By the same way we define the homogeneous operators:
$$
\forall q\in\mathbb Z,\,\quad\dot{\Delta}_{q}u=\varphi(2^{-q}\hbox{D})v\quad\hbox{and}\quad\dot{S}_{q}u=\sum_{j\leq q-1}\dot{\Delta}_{j}u. 
$$
We notice that these operators are of convolution type. For example for $q\in\ZZ,\,$ we have
$$\dot\Delta_{q}u=2^{3q}h(2^q\cdot)\star u,\quad\hbox{with}\quad h\in\mathcal{S},\quad \widehat{h}(\xi)=\varphi(\xi).
$$
For the homogeneous decomposition, the identity (\ref{dr2}) is not true due to the polynomials but we can write,
$$
u=\sum_{q\in\mathbb Z}\dot\Delta_qu
\quad\forall\,u\in 
{\mathcal {S}}'(\mathbb R^3)/{\mathcal{P}}[\mathbb R^3],
$$
where ${\mathcal{P}}[\mathbb R^3]$ is the set of 
polynomials.

We will  make continuous use of Bernstein inequalities (see for \mbox{example {\cite{MR2000a:76030}}}).
\begin{lem}\label{lb}\;
 There exists a constant $C$ such that for $k\in\NN, q\geq-1,$ $1\leq a\leq b$ and for  $u\in L^a(\RR^3)$, 
\begin{eqnarray*}
\sup_{|\alpha|\le k}\|\partial ^{\alpha}S_{q}u\|_{L^b}&\leq& C^k\,2^{q(k+3(\frac{1}{a}-\frac{1}{b}))}\|S_{q}u\|_{L^a},\\
\ C^{-k}2^
{qk}\|\dot{\Delta}_{q}u\|_{L^a}&\leq&\sup_{|\alpha|=k}\|\partial ^{\alpha}\dot{\Delta}_{q}u\|_{L^a}\leq C^k2^{qk}\|\dot{\Delta}_{q}u\|_{L^a}.
\end{eqnarray*}

\end{lem}

Let us now introduce Bony's decomposition  \cite{b} which is the basic tool of  the paradifferential calculus. In the product $uv$ of two distributions, which is not always well-defined, we distinguish three parts:
$$
uv=T_u v+T_v u+\mathcal{R}(u,v),
$$
where $T_uv$ is the paraproduct of $v$ by $u$ and $\mathcal{R}(u,v)$ the remainder term. They are defined as follows
\begin{eqnarray*}
T_u v=\sum_{q}S_{q-1}u\Delta_q v, \quad\quad \mathcal{R}(u,v)=
\sum_{q}\Delta_qu \widetilde \Delta_{q}v,\quad {\widetilde \Delta}_{q}=\sum_{i=-1}^{1}\Delta_{q+i}.
\end{eqnarray*}
\subsection{Usual and heterogeneous Besov spaces}
Now we will  define the nonhomogeneous  and homogeneous Besov spaces by using Littlewood-Paley operators. 
 Let $(p,r)\in[1,+\infty]^2$ and $s\in\mathbb R,$ then the nonhomogeneous  Besov 
\mbox{space $B_{p,r}^s$} is 
the set of tempered distributions $u$ such 
$$
\|u\|_{B_{p,r}^s}:=\Big( 2^{qs}
\|\Delta_q u\|_{L^{p}}\Big)_{\ell^{r}}<+\infty.
$$
We remark that the usual Sobolev space $H^s$ coincides with  $B_{2,2}^s$ for $s\in\RR$  and the H\"{o}lder space $C^s,$ coincides with $B_{\infty,\infty}^s$ when $s$ is not an integer. The homogeneous Besov space $\dot B_{p,r}^s$ is defined as the set of  $u\in\mathcal{S}'(\RR^d)$ up to polynomials such that
$$
\|u\|_{\dot B_{p,r}^s}:=\Big( 2^{qs}
\|\dot\Delta_q u\|_{L^{p}}\Big)_{\ell ^{r}(\ZZ)}<+\infty.
$$
The following embeddings, valid for both homogeneous and nonhomogeneous cases, are an easy consequence of  Bernstein inequalities, see for instance \cite{MR2000a:76030}, 
$$
B^s_{p_1,r_1}\hookrightarrow
B^{s+3({1\over p_2}-{1\over p_1})}_{p_2,r_2}, \qquad p_1\leq p_2\quad and \quad  r_1\leq r_2.
$$

{Let $T>0$} \mbox{and $\rho\geq1,$} we denote by $L^\rho_{T}B_{p,r}^s$ the space of distributions $u$ such that 
$$
\|u\|_{L^\rho_{T}B_{p,r}^s}:= \Big\|\Big( 2^{qs}
\|\Delta_q u\|_{L^p}\Big)_{\ell ^{r}}\Big\|_{L^\rho_{T}}<+\infty.$$
Now we will introduce the heterogeneous Besov spaces which are an extension of the classical Besov spaces.
\begin{defi}\label{heter}
Let $\Psi:[-1,+\infty[\to\RR_+^*$ be a given function.

(i)
 We say that $\Psi$ belongs to the class $\mathcal{U}$ if the following conditions are satisfied:
\begin{enumerate}
\item $\Psi$ is a nondecreasing function.
\item There exists  $C>0$ such that
$$
\sup_{q\in\NN\cup\{-1\}}\frac{\Psi(q+1)}{\Psi(q)}\leq C.
$$
%\item $\displaystyle{\lim_{x\to +\infty}\Psi(x)=+\infty}$.

\end{enumerate}
(ii) We define the class $\mathcal{U}_\infty$ by the set of function  $\Psi\in \mathcal{U}$ satisfying 
$\displaystyle{\lim_{x\to +\infty}\Psi(x)=+\infty}$.

(iii) Let $s\in\RR, p,r\in [1,+\infty]$ and $\Psi\in \mathcal{U}$. We define the heterogeneous Besov space $B_{p,r}^{s,\Psi}$ as follows:
$$
u\in B_{p,r}^{s,\Psi} \Longleftrightarrow \|u\|_{B_{p,r}^{s,\Psi}}=\Big( \Psi(q)2^{qs} \|\Delta_q u\|_{L^p}\Big)_{\ell^r}<+\infty.
$$

\end{defi}
\begin{Remas}
\begin{enumerate}
\item We observe that when we take $\Psi(x)=2^{\alpha x}$ with $\alpha\in\RR_+$,  the space $B_{p,r}^{s,\Psi}$ reduces to the classical Besov space $B_{p,r}^{s+\alpha}.$ 
\item The condition $(2)$ seems to be  necessary for the definition of $B_{p,r}^{s,\Psi}:$  it allows to get a definition which is independent of the choice of the dyadic partition. 

\end{enumerate}
\end{Remas}
The following lemma is important for the proof of Theorem \ref{thm11}. Roughly speaking, we will prove that  any element of a given Besov space is always more regular than the prescribed regularity. 
\begin{lem}\label{lemmm1}
Let $s\in\RR, p\in [1,+\infty], r\in[1,+\infty[$ and $f\in B_{p,r}^s$. Then there exists a function $\Psi$ belonging to $\mathcal{U}_\infty$ such that $f\in B_{p,r}^{s,\Psi}.$

\end{lem}
\begin{proof}
We observe that the proof reduces to the following statement: assume that a strictly positive sequence $(c_q)_{q\geq-1}$ satisfies $\sum_{q}c_q<+\infty,$ then there exists a nondecreasing sequence $(a_q)_{q\geq-1}$ satisfying 
\begin{equation}\label{passp1}
\lim_{q\to+\infty}a_q=+\infty,\quad \sup_{q\in\NN\cup\{-1\}} \frac{a_{q+1}}{a_q}\le C
\end{equation}
and such that
\begin{equation}\label{passp2}
\sum_{q\geq-1} a_q c_q<+\infty.
\end{equation}
Let $b_q=\big({{\sum_{n\geq q} c_n}}\big)^{-\frac12}$, then $(b_q)_{q\geq-1}$ is a nondecreasing sequence going to infinity. Moreover
\begin{equation}\label{hddsd1}
\sum_{q\geq-1} b_q c_q\le 2\big({\sum_{q\geq-1}c_q}\big)^{\frac12}.
\end{equation}
Indeed, we introduce the piecewise function $f:[-1,+\infty[\to\RR_+$ defined by
$$
f(x)=c_q, \quad\hbox{for}\quad x\in[q,q+1[ \quad\hbox{and}\quad q\in\NN\cup\{-1\}.
$$ 
Then we get by obvious computations,
\begin{eqnarray*}
\sum_{q\geq-1} b_q c_q&=&\sum_{q\geq-1} \frac{\int_q^{q+1}f(x)dx}{\big({\int_q^{+\infty}f(x)dx}\big)^{\frac12}}\\
&\le&\sum_{q\geq-1} {\int_q^{q+1}\frac{f(x)}{\big({\int_x^{+\infty}f(y)dy}\big)^{\frac12}}dx}\\
&\le&\int_{-1}^{+\infty}\frac{f(x)}{\big({\int_x^{+\infty}f(y)dy}\big)^{\frac12}}dx\\
&\le& 2\Big({\int_{-1}^{+\infty}f(x)dx}\Big)^{\frac12}=2\big({\sum_{q\geq-1}c_q}\big)^{\frac12}.
\end{eqnarray*}
Now we will construct by recursive procedure a sequence $(a_q)_{q\geq-1}$ satisfying \eqref{passp1} and \eqref{passp2}, and such that 
$$
\sum_{q\geq-1}a_qc_q\le 2\big({\sum_{q\geq-1}c_q}\big)^{\frac12}.
$$
Let $(a_q)_{q\geq-1}$ be the sequence defined by the following recursive formula
\begin{equation}
   \label{eqZZ1}
   \left\{
     \begin{array}{l}
      a_{q+1}=\frac12\big(a_q+\min(b_{q+1},2 a_q)\big) \\
       a_{-1}=b_{-1}.
     \end{array}
   \right.
\end{equation}
We will show first  that $a_q\le b_q$. This is true for $q=-1$ and since $(b_q)_{q\geq-1}$ is nondecreasing then 
$$
a_{q+1}-b_{q+1}\le \frac12(a_q-b_{q+1})\le\frac12(a_q-b_q).
$$
Thus we find by the principle of recurrence that $a_q\le b_q, \forall q\geq-1.$  From this property and \eqref{hddsd1} we get the convergence of the series $\sum_{q\geq-1} a_q c_q$ and more precisely
$$
\sum_{q\geq-1} a_q c_q\le 2\big({\sum_{q\geq-1}c_q}\big)^{\frac12}.
$$ 
Let us now prove that  the sequence $(a_q)_{q\geq-1}$ is nondecreasing. Indeed, by easy computations and from the nondecreasing property of $(b_q)_{q\geq-1}$ and the fact $a_q\le b_q$ we get
\begin{eqnarray*}
a_{q+1}-a_q&=&\frac12\big(\min(b_{q+1},2a_q)-a_q\big)\\
&=&\frac12\min(b_{q+1}-a_q,a_q)\\
&\geq& \frac12\min(b_{q}-a_q,a_q)\geq0.
\end{eqnarray*}
Now we will prove that $(a_q)_{q\geq-1}$ converges to $+\infty$, otherwise it will converge to a finite real \mbox{number $\ell>0$.} Using the relation \eqref{eqZZ1} combined with  the fact that $\lim_{q\to+\infty}b_q=+\infty$ yields necessary to $\ell=\frac32\ell,$ which contradicts the fact that $\ell\in]0,+\infty[.$ On the other hand, we have 
$$
1\le\frac{a_{q+1}}{a_q}\le \frac32\cdot
$$
This ends the proof of all the properties of the sequence $(a_q)_{q\geq-1}$.
\end{proof}
\begin{Rema}
From the proof of Lemma \ref{lemmm1} we may easily check that we can replace the value of $b_q$ by any expression $(\sum_{n\geq q}c_n)^{-\alpha},$ with $\alpha<1.$ The case $\alpha=1$ is not true, at least for   any convergent geometric series.
\end{Rema}
As a consequence we get the following result.
\begin{cor}\label{corza1}
Let $s\in\RR, p\in[1,\infty], r\in[1,\infty[$ and $(f_\varepsilon)_{0<\varepsilon\le1}$ be a family of smooth functions satisfying
$$
\Big( 2^{qs}\sup_{0<\varepsilon\le1}\|\Delta_qf_\varepsilon\|_{L^p}\Big)_{\ell^r}<+\infty
$$
Then there exists $\Psi\in \mathcal{U}_\infty$ such that 
$$
f_\varepsilon\in B_{p,r}^{s,\Psi},\quad\forall\,\, \varepsilon\in]0,1].
$$
\end{cor}
\begin{proof}
We set $c_q:= 2^{qsr}\sup_{0<\varepsilon\le1}\|\Delta_qf_\varepsilon\|_{L^p}^r$ then $\sum_{q}c_q<+\infty$ and thus we can use Lemma \ref{lemmm1}. Therefore there exists $\Psi\in \mathcal{U}_\infty$ such that
$$
\sum_{q\geq-1} \Psi^r(q)2^{qsr}\sup_{0<\varepsilon\le1}\|\Delta_qf_\varepsilon\|_{L^p}^r<+\infty.
$$
This achieves the proof.
\end{proof}
\subsection{Lorentz spaces}
Let us now introduce the Lorentz spaces that   will be used especially to analyze  the critical regularities. There are two ways to define these spaces: by rearrangement procedure or by using real interpolation theory. We will briefly give both descriptions. For any measurable function $f$ we define its nonincreasing rearrangement  by
$$
f^{\ast}(t):= \inf\Big\{s,\;\mu\big(\{ x,\; |f (x)| > s\}\big)\leq t\Big\},
$$
where $\mu$ denotes the usual Lebesgue measure.
For $(p,q)\in [1,+\infty]^2,$ the Lorentz space $L^{p,q}$ is the set of  functions $f$ 
such that $\|f\|_{L^{p,q}}<\infty,$ with
$$
\|f\|_{L^{p,q}}:=\left\lbrace
\begin{array}{l}\displaystyle
\Big(\int_0^\infty[t^{1\over p}f^{\ast}(t)]^q{dt\over t}\Big)^{1\over q}, \quad \hbox{for}\;1\leq q<\infty\\
\displaystyle\sup_{t>0}t^{1\over p}f^{\ast}(t),\quad \hbox{for}\;q=\infty.
\end{array}
\right.
$$
The second definition of Lorentz spaces which is equivalent to the first one is given   by real interpolation theory:
$$
(L^{p_0},L^{p_1})_{(\theta,q)}=L^{p,q},
$$
where
$
1\leq p_0<p<p_1\leq\infty,
$ and 
$\theta$ satisfy ${1\over p}={1-\theta\over p_0}
+{\theta\over p_1}$ and $1\leq q\leq\infty$.
These spaces inherited from Lebesgue spaces $L^p$  the stability property of the multiplication by bounded function:
\begin{equation}\label{imbed0}
\|uv\|_{L^{p,q}}
\le
C\|u\|_{L^\infty}\|v\|_{L^{p,q}}.
\end{equation}
On the other hand we have the following embeddings,
\begin{equation}
\label{imbed23}L^{p,q}\hookrightarrow L^{p,q'},\forall\, 1\leq p\leq\infty; 1\leq q\leq q'\leq \infty\quad \hbox{and}\quad L^{p,p}=L^p.
\end{equation}

\section{Energy estimates}
This section is devoted to the establishment of some  energy estimates for the system \eqref{eqs:3} in the framework of  heterogeneous  \mbox{Besov spaces $B_{2,r}^{s,\Psi}$}. This is an extension of   the classical estimates known for the usual  Besov spaces $B_{2,r}^s$ and whose proof   can be found for example \mbox{in \cite{MR1975081}.} 
 \begin{prop}\label{energy1}
Let $(\vepsilon,\cepsilon)$  be a smooth solution of \eqref{eqs:3}  and $\Psi\in \mathcal{U}$, see the Definition \ref{heter}. Then
\begin{enumerate}
\item $L^2-$estimate: there exists $C>0$ such that $\forall t\geq0$
$$
\|(\vepsilon,\cepsilon)(t)\|_{L^2}\le C\|(\vepsiloninit,\cepsiloninit)\|_{L^2}e^{C\|\divergence\vepsilon\|_{L^1_tL^\infty}}.
$$
\item Besov estimates: for $s>0,\alpha\geq0, r\in[1,+\infty],$ there exists $C>0$ such that
$$
\|(\vepsilon,\cepsilon)(t)\|_{B_{2,r}^{s,\Psi}}\le C\|(\vepsiloninit,\cepsiloninit)\|_{B_{2,r}^{s,\Psi}}\,e^{CV_\varepsilon(t)}
$$ 
with
$$V_\varepsilon(t):=\|\nabla\vepsilon\|_{L^1_tL^\infty}+\|\nabla \cepsilon\|_{L^1_tL^\infty}. 
$$
The above estimate holds true for homogeneous Besov spaces $\dot{B}_{2,r}^s.$
\end{enumerate}
\end{prop}
\begin{proof}
{\bf$(1)$} Taking the $L^2$ inner product of the first equation of  \eqref{eqs:3} with  $\vepsilon$ and integrating by parts 
$$
\frac12\frac{d}{dt}\|\vepsilon(t)\|_{L^2}^2-\frac12\int_{\RR^3}\divergence\vepsilon(|\vepsilon|^2+\bar\gamma\cepsilon^2)dx-\frac1\varepsilon\int_{\RR^3}\cepsilon\divergence\vepsilon dx=0.
$$
Multiplying the second equation of \eqref{eqs:3} by $\cepsilon$ and integrating by parts
$$
\frac12\frac{d}{dt}\|\cepsilon(t)\|_{L^2}^2+(\bar\gamma-\frac12)\int_{\RR^3}(\divergence\vepsilon)\cepsilon^2dx+\frac1\varepsilon\int_{\RR^3}\cepsilon\divergence\vepsilon dx=0.
$$
Thus summing up these identities yields to
\begin{eqnarray*}
\frac{d}{dt}\big(\|\vepsilon(t)\|_{L^2}^2+\|\vepsilon(t)\|_{L^2}^2\big)&=&\int_{\RR^3}\divergence\vepsilon\, \big(|\vepsilon|^2+(1-\bar\gamma)\cepsilon^2\big) dx\\
&\lesssim&\|\divergence\vepsilon(t)\|_{L^\infty}\big(\|\vepsilon(t)\|_{L^2}^2+\|\cepsilon(t)\|_{L^2}^2\big).
\end{eqnarray*}
The result follows easily from  Gronwall's lemma.

{\bf$(2)$} We will use  localize in frequency    the equations and use some results on the commutators. \mbox{Let $q\geq-1$} and set $f_q:=\Delta_q f$, then 
\begin{equation}
   \label{eqs:13}
    \left\{
     \begin{array}{l}
       \partial _t \vepsilonq  + \vepsilon \cdot \nabla
       \vepsilonq + \bar\gamma\cepsilon \nabla \cepsilonq +
       \oneover{ \varepsilon } \nabla \cepsilonq
       =-[\Delta_q,\vepsilon\cdot\nabla]\vepsilon-\bar\gamma[\Delta_q,\cepsilon]\nabla\cepsilon:=T_{\varepsilon,q}^1 \\[0.3ex]
       \partial _t \cepsilonq + \vepsilon \cdot \nabla
       \cepsilonq +\bar\gamma \cepsilon \divergence
       \vepsilonq + \oneover{ \varepsilon } \divergence \vepsilonq
       =-[\Delta_q,\vepsilon\cdot\nabla]\cepsilon-\bar\gamma[\Delta_q,\cepsilon]\divergence\vepsilon:=T_{\varepsilon,q}^2 \\[0.3ex]
      % \restr{ (\vepsilon , \cepsilon) }{ t=0 }
       %= (\vepsiloninit, \cepsiloninit ),
     \end{array}
   \right.
\end{equation}
Denote $w_{\varepsilon,q}(t):=|\cepsilonq(t)|^2+|\vepsilonq(t)|^2$, then taking the $L^2$ inner product in a similar way to the first part $(1)$ we get
\begin{eqnarray*}
\frac12\frac{d}{dt}\|w_{\varepsilon,q}(t)\|_{L^1}&=&\frac12\int_{\RR^3}\divergence\vepsilon\,w_{\varepsilon,q}dx-\bar\gamma\int_{\RR^3} \cepsilon\big( \nabla \cepsilonq \cdot \vepsilonq+ \cepsilonq\divergence\vepsilonq\big)dx\\
&+&\int_{\RR^3}(T_{\varepsilon,q}^1\vepsilonq+T_{\varepsilon,q}^2\cepsilonq)dx\\&=&
\frac12\int_{\RR^3}\divergence\vepsilon\,w_{\varepsilon,q}dx+\bar\gamma\int_{\RR^3} \nabla\cepsilon\cdot\big(  \cepsilonq  \vepsilonq)dx+\int_{\RR^3}(T_{\varepsilon,q}^1\vepsilonq+T_{\varepsilon,q}^2\cepsilonq)dx\\
&\lesssim&\Big(\|\divergence\vepsilon(t)\|_{L^\infty}+\|\nabla\cepsilon(t)\|_{L^\infty}\Big)\|w_{\varepsilon,q}(t)\|_{L^1}+\|(T_{\varepsilon,q}^1,T_{\varepsilon,q}^2)\|_{L^2}\|w_{\varepsilon,q}(t)\|_{L^1}^{\frac12}.
\end{eqnarray*}
We have used  in the last line the Young inequality 
$|ab|\le\frac12(a^2+b^2)$. It follows that
$$
\frac{d}{dt}\|(\vepsilonq,\cepsilonq)(t)\|_{L^2}\lesssim\Big(\|\divergence\vepsilon(t)\|_{L^\infty}+\|\nabla\cepsilon(t)\|_{L^\infty}\Big)\|(\vepsilonq,\cepsilonq)(t)\|_{L^2}+\|(T_{\varepsilon,q}^1,T_{\varepsilon,q}^2)\|_{L^2}.
$$
Multiplying by $2^{qs}\Psi(q)$ and summing up over $q$ we get
$$
\frac{d}{dt}\|(\vepsilon,\cepsilon)(t)\|_{B_{2,r}^{s,\Psi}}\le \big(\|\divergence\vepsilon(t)\|_{L^\infty}+\|\nabla\cepsilon(t)\|_{L^\infty}\big)\|(\vepsilon,\cepsilon)(t)\|_{B_{2,r}^{s,\Psi}}+\Big(2^{qs}\Psi(q)\|(T_{\varepsilon,q}^1,T_{\varepsilon,q}^2)\|_{L^2}\Big)_{\ell^r}.
$$
Now according to the Lemma \ref{proz1} we have: for $s>0, r\in[1,\infty]$ and $\Psi\in\mathcal{U}$
$$
\big(2^{qs}\Psi(q)\|[\Delta_q, v\cdot\nabla]u\|_{L^2}\big)_{\ell^r}\lesssim\|\nabla v\|_{L^\infty}\|u\|_{B_{2,r}^{s,\Psi}}+\|\nabla u\|_{L^\infty}\|v\|_{B_{2,r}^{s,\Psi}}.
$$
This yields to
$$
\frac{d}{dt}\|(\vepsilon,\cepsilon)(t)\|_{B_{2,r}^{s,\Psi}}\lesssim \big(\|\nabla\vepsilon(t)\|_{L^\infty}+\|\nabla\cepsilon(t)\|_{L^\infty}\big)\|(\vepsilon,\cepsilon)(t)\|_{B_{2,r}^{s,\Psi}}.
$$
It suffices to use Gonwall's inequality to get the desired estimate.
\end{proof}
\section{Strichartz estimates}
The main goal of this section is to establish some    Strichartz estimates  for the compressible and acoustic parts which are governed by coupling nonlinear  wave equations. As it was shown   in the pioneering   work of Ukai \cite{MR87h:35279} the  use of the dispersion is crucial to get rid of the well-prepared assumption when studying the  convergence  towards the  incompressible Euler system. More precisely,  it was shown that by averaging  in time  the compressible and acoustic parts vanish  when the  Mach number goes to zero.  In our case we will take advantage of the Strichartz estimates firstly to deal with the  ill-prepared case   and secondly to  improve the lifespan by a judicious combination with  the special structure of the vorticity for the   axisymmetric flows. The  results concerning  Strichartz estimates that we will use here are well-known in the literature and are shortly discussed. Nevertheless we will give more details about their applications for the  isentropic Euler system.

We will start with rewriting the system \eqref{eqs:3} with the aid of the  free wave propagator
\begin{equation}
   \label{eq:3}
    \left\{
     \begin{array}{l}
       \partial _t \vepsilon+ \oneover{ \varepsilon } \nabla \cepsilon=f_\varepsilon:=  - \vepsilon \cdot \nabla
       \vepsilon - \cepsilon \nabla \cepsilon  \\[0.3ex]
       \partial _t \cepsilon + \oneover{ \varepsilon } \divergence \vepsilon=g_\varepsilon:=- \vepsilon \cdot \nabla
       \cepsilon - \cepsilon \divergence
       \vepsilon.
     \end{array}
   \right.
\end{equation}
Let $\mathcal{Q}$ be the operator $\mathcal{Q}v:=\nabla\Delta^{-1}\divergence v$ which is nothing but the compressible part of the \mbox{velocity $v$.} Set $|\textnormal{D}|:=\sqrt{-\Delta}$, then by simple computations we show that the quantity
$$
\Gamma_\epsilon:=\mathcal{Q}\vepsilon-i\nabla|\textnormal{D}|^{-1}\cepsilon
$$
satisfies the wave equation
\begin{equation}\label{eq01}
\partial_t\Gamma_\epsilon+\frac{i}{\varepsilon}{|\textnormal{D}|}\Gamma_\epsilon=\mathcal{Q}f_\varepsilon-i\nabla|\hbox{D}|^{-1}g_\varepsilon.
\end{equation}
Analogously,  the quantity 
$$
\gamma_\epsilon:=|\textnormal{D}|^{-1}\divergence \vepsilon+i\, \cepsilon
$$
satisfies the wave equation
\begin{equation}\label{eq02}
\partial_t\gamma_\epsilon+\frac{i}{\varepsilon}{|\textnormal{D}|}\gamma_\epsilon=|\textnormal{D}|^{-1}\divergence f_\varepsilon+i\,g_\varepsilon.
\end{equation}
We will use the following Strichartz estimates which can be found for instance in \cite{bh,MR97a:46047}.
\begin{lem}\label{lemst}
Let $\varphi$ be a complex solution of the wave equation
$$
\partial_t\varphi+\frac{i}{\varepsilon}{|\textnormal{D}|}\varphi=F.
$$
Then for every $ s\in\RR$, $ r>2$ there exists $C$ such that for every $T>0$ we have
$$
\|\varphi\|_{L^r_T\dot{B}_{\infty,1}^{s-\frac32+\frac1r}}\le C\varepsilon^{\frac1r} \big(\|\varphi(0)\|_{\dot{B}_{2,1}^s}+\|F\|_{L^1_T\dot{B}_{2,1}^s}  \big).
$$
\end{lem}
As a consequence we get the following Strichartz estimates for both compressible and acoustic parts.
\begin{cor}\label{cor11}
Let $r>2$ and $\vepsiloninit,\cepsiloninit\in B_{2,1}^{\frac52-\frac1r}.$ Then the solutions of  the system \eqref{eqs:3} satisfy for all $T>0$ and $0<\varepsilon\le1$,
\begin{eqnarray*}
\|\mathcal{Q}\vepsilon\|_{L^r_T\dot{B}_{\infty,1}^0}+\|\cepsilon\|_{L^r_T\dot{B}_{\infty,1}^0}&\le& C_0\varepsilon^{\frac1r}(1+T)e^{CV_\varepsilon(T)}
\end{eqnarray*}
with $C_0$ a constant depending on $r$ and the norm $\|(\vepsiloninit,\cepsiloninit)\|_{B_{2,1}^{\frac52-\frac1r}}$ and 
$$
V_\varepsilon(T)=\int_0^T(\|\nabla\vepsilon(t)\|_{L^\infty}+\|\nabla\cepsilon(t)\|_{L^\infty}\big)dt.
$$
\end{cor}
\begin{proof}
Applyiny Lemma \ref{lemst} to the equation \eqref{eq01} with $r>2$ and $s=\frac32-\frac1r$ we  get
\begin{eqnarray*}
\|\Gamma_\varepsilon\|_{L^r_T\dot{B}_{\infty,1}^0}&\le& C \varepsilon^{\frac1r}\big(\|\Gamma_{\varepsilon}(0)\|_{\dot{B}_{2,1}^{\frac32-\frac1r}}+\|\mathcal{Q}f_\varepsilon-i\nabla|\hbox{D}|^{-1}g_\varepsilon\|_{L^1_T\dot{B}_{2,1}^{\frac32-\frac1r}}  \big).
%&\le&C_\eta \varepsilon^{\frac12}\log^{\frac12}(e+\varepsilon^{-1})\log^{\frac12}(e+T)\Big(\|(\vepsiloninit,\cepsiloninit)\|_{H^{1+\eta}}+\|f_\varepsilon\|_{L^1_TH^{1+\eta}}+\|g_\varepsilon\|_{L^1_TH^{1+\eta}}  \Big).
\end{eqnarray*}
Since the operators $\mathcal{Q}$ and $\nabla|\hbox{D}|^{-1}$ are homogeneous of order zero  then they are  continuous on the homogeneous Besov spaces and thus we get
\begin{eqnarray*}
\|\Gamma_\varepsilon\|_{L^r_T\dot{B}_{\infty,1}^0}&\le& C \varepsilon^{\frac1r}\big(\|(\vepsiloninit,\cepsiloninit)\|_{\dot{B}_{2,1}^{\frac32-\frac1r}}+\|(f_\varepsilon,g_\varepsilon)\|_{L^1_T\dot{B}_{2,1}^{\frac32-\frac1r}}  \big)\\
&\le& C \varepsilon^{\frac1r}\big(\|(\vepsiloninit,\cepsiloninit)\|_{{B}_{2,1}^{\frac32-\frac1r}}+\|(f_\varepsilon,g_\varepsilon)\|_{L^1_T{B}_{2,1}^{\frac32-\frac1r}}  \big).
\end{eqnarray*}
We point out that we have used above the embedding  ${B}_{2,1}^{s}\hookrightarrow \dot{B}_{2,1}^{s}$, for $s>0$.
%We combine the Strichartz estimates with the energy estimates: for $s>0$
%$$
%\|(\vepsilon,\cepsilon)(t)\|_{H^s}\le C\|(\vepsiloninit,\cepsiloninit)\|_{H^s}e^{CV_\varepsilon(t)}
%$$
To estimate the terms $f_\varepsilon$ and $g_\varepsilon$ we will use the following law product that can be easily proven by using Bony's decomposition: for $s>0$
$$
\|f\partial_j g\|_{{B}_{2,1}^{s}}\lesssim \|f\|_{L^\infty}\|g\|_{{B}_{2,1}^{s+1}}+\|g\|_{L^\infty}\|f\|_{{B}_{2,1}^{s+1}}.
$$ 
Therefore we find
$$
\|(f_\varepsilon,g_\varepsilon)\|_{{B}_{2,1}^{\frac32-\frac1r}}\le C\|(\vepsilon,\cepsilon)\|_{L^\infty}\|(\vepsilon,\cepsilon)\|_{{B}_{2,1}^{\frac52-\frac1r}}.
$$
Combining this estimate with Sobolev embedding ${B}_{2,1}^{\frac52-\frac1r}\hookrightarrow L^\infty$ (true  for $r\geq1$), we get
$$
\|(f_\varepsilon,g_\varepsilon)\|_{{B}_{2,1}^{\frac32-\frac1r}}\le C\|(\vepsilon,\cepsilon)\|_{{B}_{2,1}^{\frac52-\frac1r}}^2.
$$
Applying   Proposition \ref{energy1} yields to
$$
\|(\vepsilon,\cepsilon)\|_{L^\infty_T{B}_{2,1}^{\frac52-\frac1r}}\le C \|(\vepsiloninit,\cepsiloninit)\|_{{B}_{2,1}^{\frac52-\frac1r}} e^{CV_\varepsilon(T)}.
$$
Thus we obtain for $r>2$,
\begin{eqnarray*}
\|\Gamma_\varepsilon\|_{L^r_T\dot{B}_{\infty,1}^0}&\le& C \varepsilon^{\frac1r}\Big( \|(\vepsiloninit,\cepsiloninit)\|_{{B}_{2,1}^{\frac32-\frac1r}}+\|(\vepsiloninit,\cepsiloninit)\|_{{B}_{2,1}^{\frac52-\frac1r}}^2 T e^{CV_\varepsilon(T)}\Big)\\
&\le& C_0\varepsilon^{\frac1r}(1+T)e^{CV_\varepsilon(T)}.
\end{eqnarray*}
Since the real part of $\Gamma_\varepsilon$ is the compressible part of $v_\varepsilon$ then
$$
\|\mathcal{Q}\vepsilon\|_{L^r_T\dot{B}_{\infty,1}^0}\le C_0\varepsilon^{\frac1r}(1+T)e^{CV_\varepsilon(T)}.
$$
By the same way we prove a similar result for $\gamma_\varepsilon:$
$$
\|\gamma_\varepsilon\|_{L^r_T\dot{B}_{\infty,1}^0}\le C_0\varepsilon^{\frac1r}(1+T)e^{CV_\varepsilon(T)}.
$$
This gives the desired estimate for $\|\cepsilon\|_{L^r_T\dot{B}_{\infty,1}^0}.$
\end{proof}
\begin{prop}\label{stt1}
Let $(\vepsiloninit,\cepsiloninit)$ be a $ H^s$-bounded family with $s>\frac52$ and $\vepsilon,\cepsilon\in \mathcal{C}([0,T_\varepsilon);H^s)$ be the maximal solution of the system \eqref{eqs:3}. Then for every $r>2,$ $0\le T<T_\varepsilon$ and $\varepsilon\in]0,1]$ we have, 
$$
\|\divergence\vepsilon\|_{L^1_TB_{\infty,1}^0}+\|\nabla\cepsilon\|_{L^1_TB_{\infty,1}^0}\le C_0\varepsilon^{\frac{2s-5}{r(2s-3)}}(1+T^{2}) e^{CV_\varepsilon(T)}.
$$

\end{prop}
\begin{proof}
Applying Lemma \ref{lem2} with $\varphi=\mathcal{Q}\vepsilon$ and using the identity $\divergence \mathcal{Q}\vepsilon=\divergence\vepsilon$, we get
\begin{eqnarray*}
\|\divergence\vepsilon\|_{L^1_TB_{\infty,1}^0}&\le& CT^{1-\frac{2s-5}{r(2s-3)}}\|\mathcal{Q}\vepsilon\|_{L^r_TL^\infty}^{\frac{2s-5}{2s-3}}\|\mathcal{Q}\vepsilon\|_{L^\infty_TH^s}^{\frac{2}{2s-3}}\\
&\le&CT^{1-\frac{2s-5}{r(2s-3)}}\|\mathcal{Q}\vepsilon\|_{L^r_T\dot{B}_{\infty,1}^0}^{\frac{2s-5}{2s-3}}\|\vepsilon\|_{L^\infty_TH^s}^{\frac{2}{2s-3}}.
\end{eqnarray*}
We have used in the last inequality the embedding $\dot{B}_{\infty,1}^0\hookrightarrow L^\infty$. Now,
combining this estimate with Proposition \ref{energy1} and Corollary \ref{cor11},
\begin{eqnarray*}
\|\divergence\vepsilon\|_{L^1_TB_{\infty,1}^0}
&\le&C_0\varepsilon^{\frac{2s-5}{r(2s-3)}}(1+T)^{1+\frac{2s-5}{2s-3}(1-\frac1r)} e^{CV_\varepsilon(t)}\\
&\le&C_0\varepsilon^{\frac{2s-5}{r(2s-3)}}(1+T^{2}) e^{CV_\varepsilon(t)}.
\end{eqnarray*}
Similarly we obtain an analogous  estimate for the acoustic part $\|\nabla\cepsilon\|_{L^1_TB_{\infty,1}^0}$.

\end{proof}

\section{Axisymmetric flows}
In this  section we intend to establish some preliminary results about the axisymmetric geometry for  the compressible flows. First, we prove the persistence in time of this geometry when it is initially prescribed and second we analyze the structure and the dynamics of the vorticity. We end this section with some useful a priori estimates.
\subsection{Persistence of the geometry}
The study of  axisymmetric  flows was  initiated by Ladyzhanskaya \cite{Ladyza}  and Ukhoviskii and Yudovich \cite{Ukhovskii}  for both incompressible Euler and Navier-Stokes equations. This study has been recently  extended for other models of incompressible fluid dynamics like stratified Euler and Navier-Stokes  systems \cite{AHK, HR}. First of all we will show  the compatibility of this geometry with the model \eqref{eqs:3} but we need before to give a precise definition of   axisymmetric vector fields.
\begin{defi}\label{defi1}

{\it
$\bullet$ We say that a  vector field $v:\RR^3\to\RR^3$ is  axisymmetric  if it satisfies
$$
\mathcal{R}_{-\alpha}\{v(\mathcal{R}_\alpha x)\}=v(x),\quad \forall\alpha\in[0,2\pi],\quad \forall x\in\RR^3,
$$
where $\mathcal{R}_\alpha$ denotes the rotation of axis $(Oz)$ and with angle $\theta$. An axisymmetric vector field $v$ is called without swirl if its angular component vanishes, which is equivalent to the fact that $v$ takes the form:
$$
v(x) = v^r(r, z)e_r + v^z (r, z)e_z,\quad x=(x_{1},x_{2},z),\quad r=({x_{1}^2+x_{2}^2})^{\frac12},
$$
where  $\big(e_r, e_{\theta} , e_z\big)$ is the cylindrical basis of
$\mathbb R^3$ and the components $v^r$ and $v^z$ do not depend on
the angular variable.  

$\bullet$ A scalar function $f:\RR^3\to\RR$ is called axisymmetric if the vector field $x\mapsto f(x)e_z$ is axisymmetric, which means that
$$
f(\mathcal{R}_\alpha x)=f(x),\,\forall\, x\in\RR^3, \forall\,\alpha\in[0,2\pi].
$$ 
This is equivalent to say that $f$ depends only on $r$ and $z$.
}
\end{defi} 
Now we will prove the persistence of this geometry for the system \eqref{eqs:3}.
\begin{prop}\label{pers}
Let $(\vepsiloninit,\cepsiloninit)$ be a smooth axisymmetric initial data without swirl. Then the associated maximal solution of \eqref{eqs:3}   remains axisymmetric.

\end{prop}
\begin{proof}
For the sake of the simplicity we will remove the subscript  $\varepsilon$ from our notations. We set 
$$
v_\alpha(t,x)=\mathcal{R}_{-\alpha}\{v(t,\mathcal{R}_\alpha x)\}\quad\hbox{and}\quad c_\alpha(t,x)=c(t,\mathcal{R}_\alpha x).
$$
Our goal is to show that $(v_\alpha,c_\alpha)$ solves the system   \eqref{eqs:3}. First of all, we claim that
\begin{equation}\label{axx}
(v_\alpha\cdot\nabla v_\alpha)(t,x)=\mathcal{R}_{-\alpha}\{(v\cdot\nabla v)(t,\mathcal{R}_\alpha x)\}.
\end{equation}
Indeed, obvious computations yield to
\begin{eqnarray*}
(v_\alpha\cdot\nabla v_\alpha)(t,x)&=&\sum_{i=1}^3\big(\mathcal{R}_{-\alpha}\{v(t,\mathcal{R}_\alpha x)\}\big)^i\partial_i\{\mathcal{R}_{-\alpha}\,v(t,\mathcal{R}_\alpha x)\}\\
&=&\sum_{i=1}^3\big(\mathcal{R}_{-\alpha}\{v(t,\mathcal{R}_\alpha x)\}\big)^i\,\mathcal{R}_{-\alpha}\partial_i\{v(t,\mathcal{R}_\alpha x)\}\\
&=&\mathcal{R}_{-\alpha}\sum_{i=1}^3\big(\mathcal{R}_{-\alpha}\{v(t,\mathcal{R}_\alpha x)\}\big)^i\,\partial_i\{v(t,\mathcal{R}_\alpha x)\}\\
&:=&\mathcal{R}_{-\alpha}w.
\end{eqnarray*}
From the formula 
$$\partial_i\{v(t,\mathcal{R}_\alpha x)\}^j=\big(\mathcal{R}_{-\alpha}\{(\nabla v^j)(t,\mathcal{R}_\alpha x)\}\big)^i
$$
and using the fact that the rotations preserve Euclidian scalar product we obtain
\begin{eqnarray*}
w^j&=&\sum_{i=1}^3\big(\mathcal{R}_{-\alpha}\{v(t,\mathcal{R}_\alpha x)\}\big)^i\,\big(\mathcal{R}_{-\alpha}\{(\nabla v^j)(\mathcal{R}_\alpha x)\}\big)^i\\
&=&(v\cdot\nabla v^j)(t,\mathcal{R}_\alpha x).
\end{eqnarray*}
Therefore we get
\begin{eqnarray*}
(v_\alpha\cdot\nabla v_\alpha)(t,x)
&=&\mathcal{R}_{-\alpha}\{(v\cdot\nabla v)(t,\mathcal{R}_\alpha x)\}.
\end{eqnarray*}
Now if $f$ is a scalar function then
\begin{equation}\label{axxx}
\mathcal{R}_{-\alpha}\{(\nabla f)(\mathcal{R}_\alpha x)\}=\nabla\{ f(\mathcal{R}_\alpha x)\}.
\end{equation}
Using this identity and \eqref{axx} we prove that $(v_\alpha, c_\alpha)$ satisfies the first equation of \eqref{eqs:3}. It remains to prove that this couple of functions satisfies also the second equation of \eqref{eqs:3}. We write according to the identity \eqref{axxx} and from the fact that $\mathcal{R}_\alpha$ is an isometry,
\begin{eqnarray*}
(v_\alpha\cdot\nabla c_\alpha)(t,x)&=&\sum_{i=1}^3\mathcal{R}_{-\alpha}\{v(t,\mathcal{R}_\alpha x)\}\cdot\mathcal{R}_{-\alpha}\{ \nabla c\}(t,\mathcal{R}_\alpha x)\\
&=&\{v\cdot\nabla c\}(t,\mathcal{R}_\alpha x).
\end{eqnarray*}
It remains to check that
$$
\divergence v_\alpha(t,x)=\{\divergence v\}(t,\mathcal{R}_\alpha x).
$$
Indeed, set $\mathcal{R}_{-\alpha}:=(a_{ij})_{1\le i,j\le 3}$, then since $\mathcal{R}_\alpha^\star=\mathcal{R}_{-\alpha}$,  we find
\begin{eqnarray*}
\divergence v_\alpha(t,x)&=&\sum_{i,j=1}^3\partial_i\{a_{ij}v^j(t,\mathcal{R}_\alpha x)\}\\
&=&\sum_{i,k,j=1}^3a_{ij} a_{ik}\{\partial_kv^j\}(t,\mathcal{R}_\alpha x)\\
&=& \sum_{k,j=1}^3\delta_{kj}\{\partial_kv^j\}(t,\mathcal{R}_\alpha x)\\
&=&\{\divergence v\}(t,\mathcal{R}_\alpha x).
\end{eqnarray*}
Finally, we deduce that $(v_\alpha, c_\alpha)$ satisfies the same equations of $(v,c)$ and by the uniqueness of the solutions we get $v_\alpha=v, c_\alpha=c$ for every $\alpha\in[0,2\pi]$ and thus the solution is axisymmetric. To achieve the proof it remains to show that the angular component $v^\theta$ of the velocity $v$ is zero. By direct computations using the axisymmetry of the solution $(v,c)$ we get 
$$
\partial_t v^\theta+v\cdot\nabla v^\theta+\frac{v^r}{r} v^\theta=0.
$$
It follows from  the maximum principle and Gronwall's inequality that
$$
\|v^\theta(t)\|_{L^\infty}\le \|v^\theta_0\|_{L^\infty}e^{\|v^r/r\|_{L^1_tL^\infty}}.
$$
Therefore if $v^\theta_0\equiv0$ then $v^\theta(t)\equiv0$ everywhere the solution is defined.
\end{proof}
\begin{Rema}\label{remaa}
From the above computations we get the following assertions:
\begin{enumerate}
\item If $c$ is a scalar axisymmetric function then its gradient $\nabla c$ is an axisymmetric vector field.

\item If $v$ is an axisymmetric vector field, then its divergence $\divergence v$ is an axisymmetric scalar function.

\end{enumerate}
\end{Rema}
\subsection{Dynamics of the vorticity}
We will start with recalling  some algebraic properties of the axisymmetric vector fields and especially we will discuss  the special structure of the vorticity of the \mbox{system \eqref{eqs:3}}. First, we make some general statements: let $w=w^r(r,z) e_r+w^\theta(r,z)e_\theta+w^z(r,z) e_z$ and $v=v^r(r,z) e_r+v^z(r,z) e_z$ be two smooth  vector fields, then \begin{equation}\label{eqq1}
w\cdot\nabla= w^r\partial_r+\frac{w^\theta}{r}\partial_\theta+w^z\partial_z,\quad \divergence v=\partial_r v^r+\frac{v^r}{r}+\partial_z v^z.
\end{equation}
Easy computations show that the vorticity $\Omega:=\hbox{curl } v$ of the  vector field $v$ takes  the form,
$$
\Omega=(\partial_zv^r-\partial_rv^z) e_{\theta}.
$$
Now let us study  under this special geometry  the dynamics of the vorticity of the \mbox{system \eqref{eqs:3}}. Denote by 
$
\Omegaepsilon=(\partial_zv^r_\varepsilon-\partial_rv^z_\varepsilon) e_{\theta}:=\Omegaepsilon^\theta e_\theta,
$ the vorticity of $\vepsilon.$
Then applying the curl operator to the velocity equation yields
$$
\partial_t\Omegaepsilon+\textnormal{curl}(\vepsilon\cdot\nabla\vepsilon)=0.
$$
By straightforward computations we get the identity
$$
\textnormal{curl}(\vepsilon\cdot\nabla\vepsilon)=\vepsilon\cdot\nabla\Omegaepsilon-\Omegaepsilon\cdot\nabla\vepsilon+\Omegaepsilon\divergence\vepsilon.
$$
Therefore we get
$$
\partial_t\Omegaepsilon+\vepsilon\cdot\nabla\Omegaepsilon+\Omegaepsilon\divergence\vepsilon=\Omegaepsilon\cdot\nabla\vepsilon
$$
Now,  since $\Omegaepsilon=\Omegaepsilon^\theta e_\theta
$ and by \eqref{eqq1} it follows that the stretching term takes the form,
 \begin{eqnarray*}
 \Omegaepsilon\cdot\nabla\vepsilon&=&\Omegaepsilon^\theta\frac1r\partial_\theta(v^r_\epsilon e_r+v^z_\epsilon e_z)\\
 &=&\Omegaepsilon^\theta\,\frac{ v^r_\epsilon}{r}\, e_\theta\\
 &=&
 \frac{v^r_\epsilon}{r}\Omegaepsilon.
\end{eqnarray*} 
Consequently the vorticity equation becomes
\begin{equation}\label{vor0}
\partial_t\Omegaepsilon+\vepsilon\cdot\nabla\Omegaepsilon+\Omegaepsilon\divergence\vepsilon=\frac{v^r_\epsilon}{r}\Omegaepsilon.
\end{equation}
Thus the quantity $ \frac{\Omegaepsilon}{r}$ is governed by the following equation\begin{equation}\label{vor1}
(\partial_t+\vepsilon\cdot\nabla+\divergence\vepsilon)\big(\frac{\Omegaepsilon}{r}\big)=0.
\end{equation}
We observe that this equation is analogous  to the vorticity equation in dimension two,
In the case of incompressible axisymmetric   flows the quantity $ \frac{\Omegaepsilon}{r}$ is just transported by the flow and this gives new conservation laws that lead to global existence of smooth solutions. However in our case we can not get global estimates due to the presence of $\divergence \vepsilon$, yet  we have already seen that this quantity is damped by higher oscillations and thus we can expect that the  time lifespan can  grow when the Mach number becomes small.  This will be clearly discussed later in the next sections.
\subsection{Geometric properties  }
When we deal with the critical regularities, it seems that the only use of  the equation \eqref{vor1} is not sufficient for our study and so  we need more refined properties of the vorticity. We start with the following results.\begin{prop}\label{mim0} Let $v=(v^1,v^2,v^3)$ be a smooth axisymmetric vector field without swirl. 
Then the following assertions hold true.
\begin{enumerate}
\item  The vector $\omega=\nabla\times v=(\omega^1,\omega^2,\omega^3)$ satisfies $\omega(x)\times e_{\theta}(x)=(0,0,0).$ In particular, we have for every $(x_{1},x_{2},z)\in \Bbb R^3$,
$$
\omega^3=0,\quad x_{1}\omega^1(x_{1},x_{2},z)+x_{2}\omega^2(x_{1},x_{2},z)=0\quad\hbox{and}
$$
$$
 \omega^1(x_{1},0,z)=\omega^2(0,x_{2},z)=0.
$$

\item For every $q\geq -1$,  $\Delta_{q}u$ is axisymmetric without swirl and
$$
(\Delta_{q}u^1)(0,x_{2},z)=(\Delta_{q}u^2)(x_{1},0,z)=0.
$$
\item Let $\Omega$ be a free divergence vector field such that $\Omega(x)\times e_\theta(x)=0$. Then for $q\geq-1$ we have
$$
\Delta_q\Omega(x)\times e_\theta(x)=0.
$$

\end{enumerate}
\end{prop}
\begin{proof}
The results $(1)-(2)$ and $(3)$ are proved  for example in \cite{AHK}. It remains to prove the last assertion. First, it is easy to check that 
$(\Delta_q\Omega)^3=\Delta_q(\Omega^3)$ and thus the last component of $\Delta_q\Omega$ is zero and consequently our claim reduces to the following identity,
$$
x_1\Delta_q\Omega^1+x_2\Delta_q\Omega^2=0.
$$
Using Fourier transform this is equivalent to 
$$
\partial_1\big(\varphi(2^{-q}|\xi|)\widehat{\Omega^1}(\xi)\big)+\partial_2\big(\varphi(2^{-q}|\xi|)\widehat{\Omega^2}(\xi)\big)=0.
$$
Straightforward computations and the fact that $\Omega^3\equiv 0$ yield
\begin{eqnarray*}
\partial_1\big(\varphi(2^{-q}|\xi|)\widehat{\Omega^1}(\xi)\big)+\partial_2\big(\varphi(2^{-q}|\xi|)\widehat{\Omega^2}(\xi)\big)&=&2^{-q}|\xi|^{-1}\varphi^\prime(2^{-q}|\xi|)\big(\xi^1\widehat{\Omega^1}(\xi)+\xi^2\widehat{\Omega^2}(\xi)  \big)\\
&+&\varphi(2^{-q}|\xi|)\big(\partial_1\widehat{\Omega^1}(\xi)+\partial_2\widehat{\Omega^2}(\xi)\big)\\
&=&-i2^{-q}|\xi|^{-1}\varphi^\prime(2^{-q}|\xi|)\widehat{\diver \Omega}(\xi)\\&-&i\varphi(2^{-q}|\xi|)\mathcal{F}\big({x_1\Omega^1+x_2{\Omega^2}}\big)(\xi)\\
&=&0.
\end{eqnarray*}
The last identity is an easy consequence from the hypotheses $\diver \Omega=0$ and $\Omega\times e_\theta=0.$

\end{proof}
The next result deals with some properties of the flow $\psi$ associated to a time-dependent axisymmetric flow $(t,x)\mapsto v(t,x)$. It is defined as follow through the integral equation
$$
\psi(t,x)=x+\int_0^t v\big(\tau,\psi(\tau,x)\big)d\tau.
$$
It is well-known from the theory of differential equation that if $I$ is an interval and $v\in L^1(I, \textnormal{Lip})$ then the flow is well-defined on the full interval $I$.
Denote by $\mathcal{G}$ the group of rotations with axis $(Oz)$, that is
$$
\mathcal{G}:=\big\{\mathcal{R}_\theta,\theta\in[0,2\pi]\big\}.
$$Now we will prove the following result.
\begin{prop}\label{mim01}
Let $(t,x)\mapsto v(t,x)$ be a time-dependent smooth axisymmetric vector field without swirl and $(t,x)\mapsto \psi(t,x)$ its flow. Then the following results hold true.
\begin{enumerate}
\item For all $\mathcal{R}_\theta\in\mathcal{G}$ we have
$$
\mathcal{R}_\theta(\psi(t,x))=\psi(t,\mathcal{R}_\theta x), \quad\forall x\in\RR^3.
$$
\item For every $x\in\RR^3, t\geq 0$ we have
$$
\psi(t,x)\cdot e_\theta(x)=0.
$$
\item For every $t$, the vector field  $x\mapsto v(t,\psi(t,x))$ is axisymmetric without swirl.
\item For all $q\in \NN$ we have
$$
S_{q}\big(v(t)\circ\psi(t)\big)(x)\cdot e_\theta(x)=0,\quad\forall x\in\RR^3.
$$

\end{enumerate}

\end{prop}
\begin{proof}
$\bf(1)$ We set $\psi_\theta(t,x):=\mathcal{R}_\theta\psi(t,x). $ Thus differentiating with respect to $t$ we get
$$
\partial_t\psi_\theta(t,x)=\mathcal{R}_\theta\big( v(t,\psi(t,x))\big).
$$
Since $v$ is axisymmetric then
$$
\mathcal{R}_\theta\big( v(t,X)\big)=v(t,\mathcal{R}_\theta X),\quad\forall X\in\RR^3
$$
and consequently
$$
\partial_t\psi_\theta(t,x)= v\big(t,\mathcal{R}_\theta(\psi(t,x))\big)=v(t,\psi_\theta(t,x)),\quad\psi_\theta(0,x)=\mathcal{R}_\theta x.
$$
It is easy to see that $(t,x)\mapsto\psi(t,\mathcal{R}_\theta x)$ satisfies the same differential equation of $\psi_\theta$ and thus by uniqueness we get the desired identity.

$\bf(2)$ This result, which means that the trajectory of a given particle $x$ remains in the vertical plan, was proved in Proposition \cite{AHK1}.

$\bf(3)$ Let $\mathcal{R}_\theta\in \mathcal{G}.$ Since $v$ is axisymmetric then using $(1)$ of Proposition \ref{mim01} yields to
\begin{eqnarray*}
\mathcal{R}_\theta\big(v(t,\psi(t,x))\big)&=&v(t,\mathcal{R}_\theta(\psi(t,x)))\\
&=&v(t,\psi(t,\mathcal{R}_\theta x)).
\end{eqnarray*}
It remains to show that the angular component of the vector field $x\mapsto v(t,\psi(t,x))$ is zero. Since the angular component of  $v$ is zero then
$$
v(t,X)\cdot e_\theta(X)=0,\quad \forall X\in\RR^3
$$
and by $(2)$ of Proposition \ref{mim01} we have $e_\theta(\psi(t,x))=\pm e_\theta(x).$ It follows that the angular component of the vector field $x\mapsto v(t,\psi(t,x))$ is vanishes and consequently it is axisymmetric without swirl.

$\bf(4)$ Combining the preceding result with the part $(2)$ of Proposition \ref{mim0} we obtain that the vector field $x\mapsto S_q \big(v(t)\circ\psi(t))\big)(x)$ is also axisymmetric without swirl and consequently its angular component is zero.
%We write by definition and change of variables 
%\begin{eqnarray*}
%S_{q}(v(t,\psi(t,\cdot))(x)&=&2^{3q}\int_{\RR^3}h(2^{q}(x-y)) v(t,\psi(t,y))dy\\
%&=&2^{3q}\int_{\RR^3}h(2^{q}(x-\psi^{-1}(t,y))) v(t,y)dy
%\end{eqnarray*}
\end{proof}
The end of this section is devoted to the  study of some geometric properties of  a  compressible transport model. It is   a vorticity equation like  in which no relations between the vector field $v$ and the solution $\Omega$ are supposed.
$$
\mbox{(CT)}\left\lbrace
\begin{array}{l}
\partial_t \Omega+(v\cdot\nabla)\Omega+\Omega\,\diver v=\Omega\cdot\nabla v,\\
{\Omega}_{| t=0}=\Omega_0.
\end{array}
\right.
$$
We will assume that $v$ is axisymmetric and the unknown function  $\Omega=(\Omega^1,\Omega^2,\Omega^3)$ is a vector field. The following result describes the persistence    of some  initial geometric conditions of the \mbox{solution $\Omega.$}
\begin{prop}\label{mim}
Let $T>0$ and $v$ be an axisymmetric vector field without swirl and belonging to the space $L^1([0,T], \textnormal{Lip}(\RR^3))$. Denote by   $\Omega $  the unique  solution  of \textnormal{(CT)} corresponding to smooth initial data $\Omega_{0}.$ Then we have the following properties:
\begin{enumerate}
\item If $\diver \Omega_{0}=0,$ then \quad$\diver \Omega(t)=0,\quad \forall\,\,t\in [0,T].$
\item If $\Omega_{0}\times e_{\theta}=0,$ then we have 
$$ \forall t\in [0,T],\qquad \Omega(t)\times e_{\theta}=0.
$$  Consequently, $\Omega^1(t,x_{1},0,z)=\Omega^2(t,0,x_{2},z)=0,$ and 
$$
\partial_t \Omega+(v\cdot\nabla)\Omega+\Omega\,\diver v= \frac{v^r}{r}\Omega.
$$
\end{enumerate}
\end{prop}
\begin{proof}
{\bf $(1)$} We apply the divergence operator to the equation (V) 
$$
\partial_{t}\diver\Omega+\diver (u\cdot\nabla\Omega+\Omega\,\diver u)=\diver(\Omega\cdot\nabla u).
$$
Straightforward computations yield to
\begin{eqnarray*}
\diver (u\cdot\nabla\Omega)&=&u\cdot\nabla\diver \Omega+\sum_{i,j=1}^3\partial_i u^j\,\partial_j\Omega^i\\
\diver (\Omega\,\diver u)&=&\diver\Omega\,\diver u+\Omega\cdot\nabla\diver u\\
\diver(\Omega\cdot\nabla u)&=&\Omega\cdot\nabla\diver u+\sum_{i,j=1}^3\partial_i u^j\,\partial_j\Omega^i.
\end{eqnarray*}
Thus, the quantity $\diver\Omega$  satisfies the equation
$$
\partial_{t}\diver\Omega+u\cdot\nabla\diver \Omega+\diver\Omega\,\diver u=0.
$$
Using maximum principle and Gronwall inequality we get
$$
\|\diver\Omega(t)\|_{L^\infty}\le \\\diver\Omega_0\|_{L^\infty}e^{\|\diver u\|_{L^1_tL^\infty}}.
$$
Therefore if $\diver\Omega_0=0$ then for every time $\Omega(t)$ remains incompressible.

{\bf $(2)$} We denote by $(\Omega^r,\Omega^\theta,\Omega^z)$ the coordinates of $\Omega$ in cylindrical basis. It is obvious that 
$\Omega^r=\Omega\cdot e_r.$ Recall that in cylindrical coordinates the \mbox{operator $u\cdot\nabla$} has the form
$$
u\cdot\nabla=u^r\partial_{r}+\frac1r u^\theta\partial_{\theta}+u^z\partial_{z}=u^r\partial_{r}+u^z\partial_{z}.
$$
We have used in the last equality the fact that for axisymmetric flows the angular component is zero. Hence we get
\begin{eqnarray*}
(u\cdot\nabla\Omega)\cdot e_{r}&=&u^r\partial_{r}\Omega\cdot e_{r}+u^z\partial_{z}\Omega\cdot e_{r}\\
&=&(u^r\partial_{r}+u^z\partial_{z})(\Omega\cdot e_{r})\\
&=&u\cdot\nabla \Omega^r,
\end{eqnarray*}
Where we use $\partial_{r}e_{r}=\partial_{z}e_{r}=0.$
 Now it remains to compute $(\Omega\cdot\nabla u)\cdot e_{r}.$
 By a straightforward computations we get,
  \begin{eqnarray*}
 (\Omega\cdot\nabla u)\cdot e_{r}&=&\Omega^{r}\,\partial_{r}u\cdot e_{r}+\frac1r\Omega^{\theta}\,\partial_{\theta}u\cdot e_{r}+\Omega^3\,\partial_{3}u\cdot e_{r}\\
 &=&\Omega^{r}\partial_{r}u^r+\Omega^3\partial_{3}u^r.
 \end{eqnarray*}
Thus the component $\Omega^r$ obeys to the equation
$$
\partial_{t}\Omega^r+u\cdot \nabla \Omega^r+\diver u\, \Omega^r=\Omega^r\partial_{r}u^r+\Omega^3
\partial_{3}u^r.
$$
From the maximum principle  we deduce
$$
\|\Omega^r(t)\|_{L^\infty}\leq\int_{0}^t\big(\|\Omega^r(\tau)\|_{L^\infty}+\|\Omega^3(\tau)\|_{L^\infty}\big)\|\nabla u(\tau)\|_{L^\infty}d\tau.
$$

On the other hand the component $\Omega^3$ satisfies the equation
\begin{eqnarray*}
\partial_{t}\Omega^3+u\cdot\nabla\Omega^3+\diver u\,\Omega^3&=&\Omega^3\partial_{3}u^3+\Omega^r\partial_{r}u^3.
\end{eqnarray*}
 This leads to
 $$
 \|\Omega^3(t)\|_{L^\infty}\leq \int_{0}^t\big(\|\Omega^3(\tau)\|_{L^\infty}+\|\Omega^r(\tau)\|_{L^\infty}\big)\|\nabla u(\tau)\|_{L^\infty}d\tau.
 $$
 Combining these estimates and using Gronwall's inequality we obtain for every $t\in\RR_{+},$ $\Omega^3(t)=\Omega^r(t)=0,$  which is the desired result.
 
 Under these assumptions the stretching term becomes
 \begin{eqnarray*}
 \Omega\cdot \nabla u&=&\frac 1r\Omega^\theta\partial_{\theta}(u^r e_{r})\\
 &=&\frac1r u^r\Omega^\theta e_{\theta}=\frac1r u^r\Omega.
 \end{eqnarray*}
 \end{proof}

\subsection{Some a priori estimates}
Our goal in this section is to establish some a priori estimates that will be used later for both cases of critical and sub-critical regularities. Our result reads as follows.
\begin{prop}\label{vort12}
Let $(\vepsilon,\cepsilon)$ be a smooth axisymmetric solution of the system \eqref{eqs:3}. Then the following estimates hold true.
\begin{enumerate}
\item Denote by $\Omega_\varepsilon$ the vorticity of $\vepsilon$, then for $t\geq  0$ and  $p\in[1,\infty]$ we have 
$$
\Big\|\frac{\Omega_\varepsilon}{r}(t)\Big\|_{L^{p}}\le \Big\|\frac{\Omega_\varepsilon^0}{r}\Big\|_{L^{p}}e^{(1-\frac1p)\|\diver\vepsilon\|_{L^1_t L^\infty}}
$$
and  for $1<p<\infty, q\in [1,+\infty]$ 
$$
\Big\|\frac{\Omega_\varepsilon}{r}(t)\Big\|_{L^{p,q}}\le C\Big\|\frac{\Omega_\varepsilon^0}{r}\Big\|_{L^{p,q}}e^{(1-\frac1p)\|\diver\vepsilon\|_{L^1_t L^\infty}}.
$$
\item For $t\geq0$, we have
$$
\|\frac{v^r_\varepsilon}{r}\|_{L^\infty}\lesssim \|\vepsilon\|_{L^2}+\|\divergence\vepsilon\|_{B_{\infty,1}^0}+
\Big\|\frac{\Omegaepsilon^0}{r}\Big\|_{L^{3,1}}e^{\|\divergence\vepsilon\|_{L^1_tL^\infty}}.
$$
\item For $t\geq0$, we have
$$
\|\Omegaepsilon(t)\|_{L^\infty}\le\|\Omegaepsilon^0\|_{L^\infty}e^{C\|\vepsilon\|_{L^1_tL^2}}e^{C\Big\{\|\divergence\vepsilon\|_{L^1_tB_{\infty,1}^0}+t\,\|{\Omegaepsilon^0}/{r}\|_{L^{3,1}}\,e^{\|\divergence\vepsilon\|_{L^1_tL^\infty}}\Big\}}.
$$

\end{enumerate}
\end{prop}
\begin{proof}
$\bf{(1)}$ To prove this result we will use the characteristic method. Let $\psi_\varepsilon$ denote the flow associated to the velocity $v_\varepsilon$ and  defined by the integral equation
$$
\psi_\varepsilon(t,x)=x+\int_0^t\vepsilon(\tau,\psi_\varepsilon(\tau,x))d\tau.
$$
Set $f_\varepsilon(t,x)=\frac{\Omega_\varepsilon}{r}(t,\psi(t,x))$, then it is easy to see from the equation \eqref{vor1} that $$
\partial_tf_\varepsilon(t,x)+(\diver \vepsilon)(t,\psi(t,x))f_\varepsilon(t,x)=0.
$$ 
Thus we get easily
\begin{equation}\label{maha2}
f_\varepsilon(t,x)=f_\varepsilon(0,x)\,e^{-\int_0^t(\diver \vepsilon)(\tau,\psi(\tau,x))d\tau}.
\end{equation}
For $p=+\infty$ we get
$$
\Big\|\frac{\Omega_\varepsilon}{r}(t)\Big\|_{L^\infty}=\Big\|\frac{\Omega_\varepsilon^0}{r}\Big\|_{L^{\infty}}.
$$
Now for $p\in[1,+\infty[$ we will use the identity
$$
\textnormal{det} \nabla\psi_\varepsilon(t,x)=e^{\int_0^t(\diver \vepsilon)(\tau,\psi(\tau,x))d\tau}
$$
and thus we get by \eqref{maha2}
\begin{eqnarray*}
\Big\|\frac{\Omega_\varepsilon}{r}(t)\Big\|_{L^p}^p&=&\|f_\varepsilon(t)\circ\psi_\varepsilon^{-1}(t,\cdot)\|_{L^p}^p\\
&=&\int_{\RR^3}|f_\varepsilon(t,x)|^p\textnormal{det} \nabla\psi_\varepsilon(t,x)dx\\
&=&\int_{\RR^3}|f_\varepsilon(0,x)|^pe^{(1-p)\int_0^t(\diver \vepsilon)(\tau,\psi(\tau,x))}d\tau\\
%&\le& \|f_\varepsilon(t)\|_{L^p}\|\textnormal{det} \nabla\psi_\varepsilon(t)\|_{L^\infty}^{1\over p}\\
&\le&\Big\|\frac{\Omega_\varepsilon^0}{r}\Big\|_{L^{p}}^pe^{(p-1)\|\diver \vepsilon\|_{L^1_tL^\infty}}.
\end{eqnarray*}
The estimates in Lorentz spaces are a direct consequence of real interpolation argument.

${\bf(2)}$ The estimate of the quantity $\|\frac{v^r_\epsilon}{r}\|_{L^1_tL^\infty}$ will require the use of some special properties of axisymmetric flows. First, we split the velocity into compressible and incompressible parts:
 $$\vepsilon=\mathcal{P}\vepsilon+\nabla\Delta^{-1}\divergence\vepsilon.
 $$ 
 We point out that in this decomposition both vector fields are also axisymmetric. Indeed, from Remark \ref{remaa}, the scalar function $\divergence\vepsilon$ is an axisymmetric scalar function and   $\Delta^{-1}\divergence\vepsilon$ too. Again from Remark \ref{remaa} the vector field  $\nabla\Delta^{-1}\divergence\vepsilon$ is axisymmetric. Now obvious computations yield to
\begin{eqnarray}\label{hama1}
\nonumber{v^r_\varepsilon}&=&{(\mathcal{P}\vepsilon)^r}+(\nabla\Delta^{-1}\divergence\vepsilon)\cdot e_r\\
&=&{(\mathcal{P}\vepsilon)^r}+{\partial_r}{\Delta^{-1}\divergence\vepsilon}.
\end{eqnarray}
Therefore we deduce
$$
\frac{v^r_\varepsilon}{r}=\frac{(\mathcal{P}\vepsilon)^r}{r}+\frac{\partial_r}{r}{\Delta^{-1}\divergence\vepsilon}\cdot
$$
To estimate the second term of the preceding identity we make  use  of the algebraic identity described in Lemma \ref{prop1} and dealing with the action of the  operator $\frac{\partial_r}{r}\Delta^{-1}u$ over axisymmetric functions. 
\begin{eqnarray*}
\Big\|\frac{\partial_r}{r}\Delta^{-1}\divergence\vepsilon
  \Big\|_{L^\infty}&\le&\sum_{i,j=1}^2\|\mathcal{R}_{ij}\divergence\vepsilon\|_{L^\infty}\\
  &\le&\sum_{i,j=1}^2\|\mathcal{R}_{ij}\divergence\vepsilon\|_{\dot{B}_{\infty,1}^0}\\
 &\le& C \|\divergence\vepsilon
  \Big\|_{\dot{B}_{\infty,1}^0}\\
  &\le&C\|\vepsilon\|_{L^2}+\|\divergence\vepsilon
  \Big\|_{{B}_{\infty,1}^0}.
\end{eqnarray*}
We have used the emebedding $\dot{B}_{\infty,1}^0\hookrightarrow L^\infty$ combined with the continuity of Riesz transforms on the homogeneous Besov spaces. Furthermore, In the last line we use Bernstein inequalities in order to bound low frequencies:
$$
\sum_{j\in\ZZ_{-}}\|\dot\Delta_j\divergence\vepsilon\|_{L^\infty}\le C\|\vepsilon\|_{L^2}\sum_{j\in\ZZ_{-}}2^{\frac52\,j}\le  C\|\vepsilon\|_{L^2}.
$$
Now let us come back to \eqref{hama1} and look at the  incompressible term.  Since $\mathcal{P}\vepsilon$ is axisymmetric and satisfying furthermore $\divergence \mathcal{P}\vepsilon=0$ and $\hbox{curl }\mathcal{P}\vepsilon=\Omegaepsilon$ then we can use the following inequality proven by Shirota and Yanagisawa \cite{Taira}, 
$$
\Big|\frac{(\mathcal{P}\vepsilon)^r}{r}(x)\Big|\le C\int_{\RR^3}\frac{|\frac{\Omegaepsilon^\theta}{r}(y)|}{|x-y|^2}dy.
$$
Now since $\frac{1}{|\cdot|^{2}}$ belongs to Lorentz space $L^{3,\infty}$ then the usual convolution laws yield
$$
\Big\|\frac{(\mathcal{P}\vepsilon)^r}{r}(t)\Big\|_{L^\infty}\lesssim\Big\|\frac{\Omegaepsilon}{r}(t)\Big\|_{L^{3,1}}.
$$
According to the first part of Proposition \ref{vort12} we have
$$
\Big\|\frac{\Omegaepsilon}{r}(t)\Big\|_{L^{3,1}}\le C\Big\|\frac{\Omegaepsilon^0}{r}\Big\|_{L^{3,1}}e^{\|\divergence\vepsilon\|_{L^1_tL^\infty}}.
$$
Combining these estimates we get
\begin{equation}\label{lao7}
\|\frac{v^r_\varepsilon}{r}\|_{L^\infty}\lesssim \|\vepsilon\|_{L^2}+\|\divergence\vepsilon\|_{B_{\infty,1}^0}+
\Big\|\frac{\Omegaepsilon^0}{r}\Big\|_{L^{3,1}}e^{\|\divergence\vepsilon\|_{L^1_tL^\infty}}.
\end{equation}
$\bf{(3)}$  Applying maximum principle to the vorticity equation \eqref{vor0} and using  Gronwall's inequality we get
\begin{equation}\label{vorrt}
\|\Omegaepsilon(t)\|_{L^\infty}\le\|\Omegaepsilon(0)\|_{L^\infty}e^{\|\divergence\vepsilon\|_{L^1_tL^\infty}+\|\frac{v^r_\epsilon}{r}\|_{L^1_tL^\infty}}.
\end{equation}
Plugging the estimate \eqref{lao7}    into \eqref{vorrt} gives
$$
\|\Omegaepsilon(t)\|_{L^\infty}\le\|\Omegaepsilon^0\|_{L^\infty}e^{C\|\vepsilon\|_{L^1_tL^2}}e^{C\Big\{\|\divergence\vepsilon\|_{L^1_tB_{\infty,1}^0}+t\,\|{\Omegaepsilon^0}/{r}\|_{L^{3,1}}\,e^{\|\divergence\vepsilon\|_{L^1_tL^\infty}}\Big\}}.
$$

\end{proof}

%%
%%% Proof Main result
%%
\section{Sub-critical regularities}
The main goal of this section is to prove  Theorem \ref{thm1}, however   we shall limit our analysis    to the estimate of the  lifespan of the solutions  and to the rigorous justification of the low Mach number limit. For example we will not deal with the construction of maximal solution which is cassical and was done for instance in \cite{MR84d:35089}.
 \subsection{Lower bound of $T_\varepsilon$}
 
We will show how the combination  of Strichartz estimates with the special structure of axisymmetric flows  allow to improve  the estimate of the time lifespan $T_\varepsilon$ in the case of ill-prepared initial data. We shall prove the following result.
\begin{prop}\label{lower}
Let $s>\frac52$ and  assume that 
$$
\sup_{\varepsilon\le1}\|(\vepsiloninit,\cepsiloninit)\|_{H^s}<+\infty.
$$
Then the system \eqref{eqs:3} admits a unique solution $(\vepsilon,\cepsilon)\in C([0,T_\varepsilon]; H^s)$ such that
$$
T_\varepsilon\ge C_0 \log\log\log(\varepsilon^{-1}).
$$
Moreover, there exists $\sigma>0$ such that for every $T\in[0,T_\varepsilon]$
$$
\|\divergence\vepsilon\|_{L^1_TB_{\infty,1}^0}+\|\nabla\cepsilon\|_{L^1_TB_{\infty,1}^0}\le C_0\varepsilon^\sigma
$$
and
$$
\|\Omegaepsilon(T)\|_{L^\infty}\le C_0e^{C_0 T},\quad \|\nabla\vepsilon\|_{L^1_TL^\infty}\le C_0 {e^{\exp\{C_0 T\}}},
$$
$$\|(\vepsilon,\cepsilon)(T)\|_{H^s}\le C_0 e^{e^{\exp\{C_0 T\}}}.
$$
\end{prop}
\begin{proof}
According to Proposition \eqref{vort12} we have
$$
\|\Omegaepsilon(t)\|_{L^\infty}\le\|\Omegaepsilon^0\|_{L^\infty}e^{C\|\vepsilon\|_{L^1_tL^2}}e^{C\Big\{\|\divergence\vepsilon\|_{L^1_tB_{\infty,1}^0}+t\,\|{\Omegaepsilon^0}/{r}\|_{L^{3,1}}\,e^{\|\divergence\vepsilon\|_{L^1_tL^\infty}}\Big\}}.
$$
To bound $\|{\Omegaepsilon^0}/{r}\|_{L^{3,1}}$ we  use the fact that  $\Omegaepsilon^0(0,0,z)=0, $ see Propostion \ref{mim0},  combined with Taylor formula  
$$
\Omegaepsilon^0(x_{1},x_{2},z)=\int_{0}^1\Big(x_{1}\partial_{x_{1}}\Omegaepsilon^0(\tau x_{1},\tau x_{2},z)+x_{2}\partial_{x_{2}}\Omegaepsilon^0(\tau x_{1},\tau x_{2},z)\Big)d\tau.  
$$
Applying \eqref{imbed0}  with  a homogeneity argument yields
\begin{eqnarray*}
\|\Omegaepsilon^0/r\|_{L^{3,1}}&\lesssim&\int_{0}^1\|\nabla\Omegaepsilon^0(\tau\cdot,\tau\cdot,\cdot)\|_{L^{3,1}}d\tau\\
&\lesssim& \|\nabla\Omegaepsilon^0\|_{L^{3,1}}\int_{0}^1\tau^{-\frac{2}{3}}d\tau\\
&\lesssim &\|\nabla\Omegaepsilon^0\|_{L^{3,1}}.  \end{eqnarray*}
According to the embedding $H^{s-2}\hookrightarrow L^{3,1} ,\,\hbox{for } s>\frac52$ one has
\begin{eqnarray*}
\|{\Omegaepsilon^0}/{r}\|_{L^{3,1}}&\lesssim& \|{\nabla\Omegaepsilon^0}\|_{H^{s-2}}\\
&\lesssim&\|\Omegaepsilon^0\|_{H^{s-1}}\\
%&\lesssim&\|\mathcal{P}(\vepsiloninit)\|_{H^{s}}\\
&\lesssim&\|\vepsiloninit\|_{H^s}.
\end{eqnarray*}
From Proposition \ref{energy1} we get
\begin{equation}\label{tour1}
\|\Omegaepsilon(t)\|_{L^\infty}\le C_0 e^{C_0 (1+t)\exp{\big\{\|\divergence\vepsilon\|_{L^1_tB_{\infty,1}^0}}\big\}}.
\end{equation}

%\begin{proof}
Using Lemma \ref{lem22} and Proposition \ref{energy1} we get
$$
\|\nabla\vepsilon(t)\|_{L^\infty}\lesssim\|\vepsilon(t)\|_{L^2}+\|\divergence\vepsilon(t)\|_{B_{\infty,1}^0}+C_0\|\Omegaepsilon(t)\|_{L^\infty}\big(1+V_\varepsilon(t)\big).
$$
Integrating in time and using Proposition \ref{energy1} and Proposition \ref{stt1},
\begin{eqnarray*}
\|\nabla\vepsilon\|_{L^1_TL^\infty}&\lesssim& \|\vepsilon\|_{L^1_TL^2}+\|\divergence\vepsilon\|_{L^1_TB_{\infty,1}^0}+C_0\int_0^T\|\Omegaepsilon(t)\|_{L^\infty}\big(1+V_\varepsilon(t)\big)dt\\
&\le&C_0Te^{\|\divergence\vepsilon\|_{L^1_TL^\infty}}+C_0\alpha_\varepsilon(T)  e^{CV_\varepsilon(T)}+C_0\int_0^T\|\Omegaepsilon(t)\|_{L^\infty}\big(1+V_\varepsilon(t)\big)dt,
\end{eqnarray*}
with
$$
\alpha_\varepsilon(T):=C_0\varepsilon^{\frac{2s-5}{r(2s-3)}}(1+T^{2}).
$$
Using again Proposition \ref{stt1} we obtain
$$
V_\varepsilon(T)\le  C_0\alpha_\varepsilon(T)  e^{CV_\varepsilon(T)}+C_0Te^{\|\divergence\vepsilon\|_{L^1_tL^\infty}}+C_0\int_0^T\|\Omegaepsilon(t)\|_{L^\infty}(1+V_\varepsilon(t)dt.
$$
Thus Gronwall's lemma combined with \eqref{tour1} and Proposition \ref{stt1} yield
\begin{eqnarray}\label{tout01}
\nonumber V_\varepsilon(T)&\le& C_0\,e^{C_0\int_0^T\|\Omegaepsilon(t)\|_{L^\infty}dt}\Big(\alpha_\varepsilon(T)  e^{CV_\varepsilon(T)}+Te^{\|\divergence\vepsilon\|_{L^1_tL^\infty}}\Big)\\
\nonumber&\le& e^{e^{C_0 (1+T)\exp{\big\{\|\divergence\vepsilon\|_{L^1_TB_{\infty,1}^0}}\big\}}}
 \Big(\alpha_\varepsilon(T)  e^{CV_\varepsilon(T)}+Te^{\|\divergence\vepsilon\|_{L^1_tL^\infty}}\Big)\\
 \nonumber&\le&  e^{e^{C_0 (1+T)\exp{\big\{\|\divergence\vepsilon\|_{L^1_tB_{\infty,1}^0}}\big\}}}
 \Big(\alpha_\varepsilon(T) \,e^{CV_\varepsilon(T)}+1\Big)\\
 &\le&e^{e^{C_0 (1+T)\exp{\big\{\alpha_\varepsilon(T)e^{CV_\varepsilon(T)}\big\}}}}.
\end{eqnarray}
We choose $T_\varepsilon$ such that
\begin{equation}\label{c1}
\log\big(\frac{1}{\alpha_\varepsilon(T_\varepsilon)}\big)>Ce^{e^{\,C_0\, e\,(1+T_\varepsilon)}}.
\end{equation}
Then we claim that for every $t\in[0,T_\varepsilon]$ 
\begin{equation}\label{tout1}
CV_\varepsilon(t)\le \log\big(\frac{1}{\alpha_\varepsilon(T_\varepsilon)}\big).
\end{equation}
Indeed, we set $I_{T_\varepsilon}:=\big\{t\in[0,T_\varepsilon]; CV_\varepsilon(t)\le \log\big(\frac{1}{\alpha_\varepsilon(T_\varepsilon)}\big)\big\}.$ First this set is nonempty since $0\in I_{T_\varepsilon}$. By the continuity of $t\mapsto V_\varepsilon(t)$, the set  $I_{T_\varepsilon}$ is closed and thus to prove that $I_{T_\varepsilon}$ coincides with $[0,T_\varepsilon]$ it suffices to show that $I_{T_\varepsilon}$ is an open set. Let $t\in I_{T_\varepsilon}$ then using \eqref{tout01} and \eqref{tout1} we get
\begin{eqnarray*}
CV_\varepsilon(t)&\le& Ce^{e^{C_0 (1+t)\exp{\big\{\alpha_\varepsilon(t)e^{CV_\varepsilon(t)}\big\}}}}\\
&\le&Ce^{e^{C_0 (1+t)\exp\{\frac{\alpha_\varepsilon(t)}{\alpha_\varepsilon(T_\varepsilon)}\}}}\\
&\le&Ce^{e^{C_0\,e\, (1+t)}}\\
&<&\log\big(\frac{1}{\alpha_\varepsilon(T_\varepsilon)}\big).
\end{eqnarray*}
This proves that $t$ is in the interior of $I_{T_\varepsilon}$ and thus $I_{T_\varepsilon}$ is an open set of $[0,T_\varepsilon]$. Consequently we conclude that $I_{T_\varepsilon}=[0,T_\varepsilon].$
Now we choose precisely $T_\varepsilon$ such that 
$$
C_0 e (1+T_\varepsilon)=\log\log\log(\varepsilon^{-\theta}),\quad \hbox{with}\quad \theta>0,
$$
then the assumption \eqref{c1} is satisfied whenever
$$
{\varepsilon^{\frac{2s-5}{r(2s-3)}}}(1+T_\varepsilon^{2})<{\varepsilon^{C\theta}}\cdot
$$
This last condition  is satisfied for small values of $\varepsilon$ when  $C\theta<\frac{2s-5}{r(2s-3)}.$ Now inserting \eqref{tout1} into \eqref{tout01} we  get  for $T\in[0,T_\varepsilon]$
\begin{equation}\label{Lipsss}
V_\varepsilon(T)\le e^{e^{C_0 e(1+T)}}. 
\end{equation}
Plugging this into the estimate of Proposition \ref{stt1} we obtain for $T\in[0,T_\varepsilon]$
%$$
%\|\divergence\vepsilon\|_{L^1_TB_{\infty,1}^0}+\|\nabla\cepsilon\|_{L^1_TB_{\infty,1}^0}\le C e^{e^{C_0 (1+T)}}\,\varepsilon^{\frac{2s-5}{8s-12}}.
%$$
\begin{eqnarray}\label{disppp1}
\nonumber\|\divergence\vepsilon\|_{L^1_TB_{\infty,1}^0}+\|\nabla\cepsilon\|_{L^1_TB_{\infty,1}^0}&\le& C_0 e^{Ce^{\exp\{C_0 e (1+T)\}}}\,\varepsilon^{\frac{2s-5}{r(2s-3)}}\\
\nonumber&\le&C\,\varepsilon^{\frac{2s-5}{r(2s-3)}-C\theta}\\
&\le& C\varepsilon^\sigma,\quad \hbox{with}\quad \sigma=\frac{2s-5}{r(2s-3)}-C\theta.
\end{eqnarray}
From Corollary \ref{cor11} and by choosing $\theta$ sufficiently small we obtain for every $r\in]2,+\infty[$
\begin{eqnarray*}
\nonumber\|\mathcal{Q}\vepsilon\|_{L^r_TL^\infty}+\|\cepsilon\|_{L^r_TL^\infty}&\le&C_0 \varepsilon^{\frac1r}(1+T)e^{C V_\varepsilon(t)} \\
\nonumber&\leq&C_0  \varepsilon^{\frac1r-C\theta}\log\log\log(\varepsilon^{-\theta})\\
&\le&C_0\varepsilon^{\sigma'},\quad\sigma^\prime>0.
\end{eqnarray*}
We point out that the use of H\"{o}lder inequality and the slow growth of $T_\varepsilon$ allow us to get the following: for every $r\in[1,+\infty[$
\begin{eqnarray}\label{esppp1}
\|\mathcal{Q}\vepsilon\|_{L^r_TL^\infty}+\|\cepsilon\|_{L^r_TL^\infty}
&\le&C_0\varepsilon^{\sigma'},\quad\sigma^\prime>0.
\end{eqnarray}
Inserting \eqref{disppp1} into \eqref{tour1} leads to
$$\|\Omegaepsilon(T)\|_{L^\infty}\le C_0 e^{C_0T}.
$$
To  estimate the solutions of \eqref{eqs:3} in Sobolev norms we use Proposition \ref{energy1} combined with the Lipschitz bound \eqref{Lipsss}
$$
\|(\vepsilon,\cepsilon)(T)\|_{H^s}\le C_0 e^{e^{\exp\{C_0T\}}}.
$$
\end{proof}
%%
%% Incompressible limit
%%
\subsection{Incompressible limit}
Our task now is to prove the second part of  Theorem \ref{thm1}. More precisely we will establish the following result.
\begin{prop}
Let $(\vepsiloninit,\cepsiloninit)$ be a $H^s$-bounded family of axisymmetric initial data with $s>\frac52$, that is
$$
\sup_{0<\varepsilon\le1}\|(\vepsiloninit,\cepsiloninit)\|_{H^s}<+\infty.
$$
 Assume that there exists $v_0\in H^s$ such that $$
 \lim_{\varepsilon\to 0}\|\mathcal{P}\vepsiloninit-v_0\|_{L^2}=0.
 $$ 
  Then the family
$(\mathcal{P}\vepsilon)_{\varepsilon}$  tends to the 
   solution $v$ of the incompressible Euler  system $(\ref{eq:45})$ associated to the initial data $v_0$: for every $T>0,$  
   \begin{equation*}
    \lim_{\varepsilon\to 0} \|\mathcal{P} \vepsilon-v\|_{L^\infty_TL^2}=0.
   \end{equation*}
However, the compressible and acoustic parts tend to zero: for every $r\in[1,+\infty[,T>0$
 \begin{equation*}
    \lim_{\varepsilon\to 0} \|(\mathcal{Q}\vepsilon, \cepsilon)\|_{L^r_TL^\infty}=0.
   \end{equation*}

\end{prop}
\begin{proof}
The last assertion concerning the compressible and acoustic parts of the fluid is a direct consequence of  \eqref{esppp1}. Now 
we intend to show that for every time $T>0$ the family $(\mathcal{P}\vepsilon)_{\varepsilon}$ is a Cauchy sequence in $L^\infty_TL^2$ for small values of $\varepsilon$ and this will allow us to prove strong convergence, when $\varepsilon$ tends to zero, of the incompressible parts to a vector-valued function $v$ which is a solution of the incompressible Euler equations. For this purpose, let $\varepsilon>\varepsilon^\prime>0$ and set $\zeta_{\varepsilon,\varepsilon'}=v_\varepsilon-v_{\varepsilon^\prime}$ and $w_{\varepsilon,\varepsilon'}=\mathcal{P}{v_{\varepsilon}}-\mathcal{P}{v_{\varepsilon^\prime}}.$ Applying Leray's projector $\mathcal{P}$ to the first equation of \eqref{eqs:3} yields to
$$
\partial_t\mathcal{P}\vepsilon+\mathcal{P}(\vepsilon\cdot\nabla \vepsilon)=0.
$$
It follows that
$$
\partial_t\ww+\mathcal{P}(\vepsilon\cdot\nabla \zeta_{\varepsilon,\varepsilon'})+\mathcal{P}(\zeta_{\varepsilon,\varepsilon'}\cdot\nabla v_{\epsilon^\prime})=0.
$$
Taking the $L^2$ inner product of this equation with $\ww$  and using the fact that Leray's projector is an involutive self-adjoint operator we get
\begin{eqnarray*}
\frac12\frac{d}{dt}\|\ww(t)\|_{L^2}^2&=&-\int_{\RR^3}(\vepsilon\cdot\nabla \zeta_{\varepsilon,\varepsilon'})\mathcal{P}\ww dx-\int_{\RR^3}(\zeta_{\varepsilon,\varepsilon'}\cdot\nabla v_{\epsilon^\prime})\mathcal{P}\ww dx\\
&=&-\int_{\RR^3}(\vepsilon\cdot\nabla(\ww+ \mathcal{Q} \zeta_{\varepsilon,\varepsilon'}))\,\ww dx-\int_{\RR^3}((\ww+\mathcal{Q}\zeta_{\varepsilon,\varepsilon'})\cdot\nabla v_{\epsilon^\prime})\ww dx\\
&=&\frac12\int_{\RR^3}|\ww|^2\,\divergence\,\vepsilon dx-\int_{\RR^3}(\ww\cdot\nabla v_{\epsilon^\prime})\ww dx\\&-&
\int_{\RR^3}\big(\vepsilon\cdot\nabla \mathcal{Q} \zeta_{\varepsilon,\varepsilon'}+\mathcal{Q}\zeta_{\varepsilon,\varepsilon'}\cdot\nabla v_{\epsilon^\prime}\big)\ww dx\\
&\le&\big(\|\divergence \vepsilon(t)\|_{L^\infty}+\|\nabla v_{\varepsilon^\prime}(t)\|_{L^\infty}\big)\|\ww(t)\|_{L^2}^2\\&+&\Big(\|\vepsilon(t)\|_{L^2}\|\nabla \mathcal{Q} \zeta_{\varepsilon,\varepsilon'}(t)\|_{L^\infty}+\|\mathcal{Q}\zeta_{\varepsilon,\varepsilon'}(t)\|_{L^\infty}\|\nabla v_{\varepsilon^\prime}(t)\|_{L^2}  \Big)\|\ww(t)\|_{L^2}.
\end{eqnarray*}
Integrating in time and using Gronwall's lemma 
$$
\|\ww(t)\|_{L^2}\leq\big(\|\ww^0\|_{L^2} +F_{\varepsilon,\varepsilon^\prime}(t)\big)e^{\|\divergence \vepsilon\|_{L^1_tL^\infty}+\|\nabla v_{\varepsilon^\prime}\|_{L^1_tL^\infty}}
$$
with
$$F_{\varepsilon,\varepsilon^\prime}(t)=\|\vepsilon\|_{L^\infty_tL^2}\|\nabla \mathcal{Q} \zeta_{\varepsilon,\varepsilon'}\|_{L^1_tL^\infty}+\|\mathcal{Q}\zeta_{\varepsilon,\varepsilon'}\|_{L^1_tL^\infty}\|\nabla v_{\varepsilon^\prime}\|_{L^\infty_tL^2}   \Big).
$$
Splitting $ \mathcal{Q} \zeta_{\varepsilon,\varepsilon'}$ into low and high frequencies we get
$$
\|\nabla \mathcal{Q} \zeta_{\varepsilon,\varepsilon'}\|_{L^1_tL^\infty}\lesssim \|\textnormal{div }\zeta_{\varepsilon,\varepsilon'}\|_{L^1_tB_{\infty,1}^0}+\| \mathcal{Q} \zeta_{\varepsilon,\varepsilon'}\|_{L^1_tL^\infty}.
$$
Therefore combining  Proposition  \ref{energy1} and Proposition  \ref{lower} with \eqref{esp1} we get for some $\eta>0$
\begin{eqnarray*}
\sup_{t\in[0,T_\varepsilon]}F_{\varepsilon,\varepsilon^\prime}(t)&\le& C_0\varepsilon^\eta.
\end{eqnarray*}
According to Proposition \ref{lower} we obtain
$$
\|\ww(t)\|_{L^2}\le C_0e^{e^{\exp\{C_0 t\}}}\big(\|\ww^0\|_{L^2}+\varepsilon^\eta\big).
$$
Thus we deduce that $(\mathcal{P}\vepsilon)_{\varepsilon}$ is a Caucy sequence in the Banach space $L^\infty_TL^2$ and therefore it converges strongly to some $v$ which belongs to $L^\infty_{\textnormal{loc}}(\RR_+;H^s)$. This last claim is a consequence of the uniform bound in $H^s$. It remains now to prove that $v$ is a solution for the  incompressible Euler system with initial data $v^0:=\lim_{\varepsilon\to 0}\mathcal{P}\vepsiloninit.$ The passage to the limit in the linear part $(\partial_t \vepsilon)_{\varepsilon}$ is easy to do by using the convergence  in the  weak sense. For the nonlinear term we split the velocity into its compressible and incompressible parts:	
$$
\vepsilon\cdot\nabla\vepsilon=\mathcal{P}\vepsilon\cdot\nabla\mathcal{P}\vepsilon+\mathcal{P}\vepsilon\cdot\nabla\mathcal{Q}\vepsilon+\mathcal{Q}\vepsilon\cdot\nabla\vepsilon.
$$
By virtue of  Proposition \ref{energy1}, Proposition \ref{lower} and \eqref{esppp1}  we get
\begin{eqnarray*}
\|\mathcal{P}\vepsilon\cdot\nabla\mathcal{Q}\vepsilon\|_{L^1_TL^2}&\le& \|\vepsilon\|_{L^\infty_TL^2}\|\nabla\mathcal{Q}\vepsilon\|_{L^1_TL^\infty}\\
&\le& C_0\big(\|\mathcal{Q}\vepsilon\|_{L^1_TL^\infty}+\|\divergence\vepsilon\|_{L^1_TB_{\infty,1}^0}  \big)\\
&\le& C_0\varepsilon^\eta. 
\end{eqnarray*}
For the last term we  combine \eqref{esppp1} with Proposition \ref{lower},
\begin{eqnarray*}
\|\mathcal{Q}\vepsilon\cdot\nabla\vepsilon\|_{L^1_TL^2}&\le& \|\vepsilon\|_{L^\infty_TH^1}\|\mathcal{Q}\vepsilon\|_{L^1_TL^\infty}\\
&\le& C_0 \varepsilon^\eta. 
\end{eqnarray*}
It follows that
$$
\lim_{\varepsilon\to 0}\|\vepsilon\cdot\nabla\vepsilon-\mathcal{P}\vepsilon\cdot\nabla\mathcal{P}\vepsilon\|_{L^1_TL^2}=0.
$$
The strong convergence of $\mathcal{P}\vepsilon$ to $v$ in $L^\infty_TL^2$ allows us to claim that 
$$
\lim_{\varepsilon\to 0}\|\mathcal{P}\vepsilon\otimes\mathcal{P}\vepsilon- v\otimes\, v\|_{L^1_TL^1}=0.
$$
Consequently, the family
$\mathcal{P}(\vepsilon\cdot\nabla\vepsilon)$ converges strongly in some Banach space to $\mathcal{P}(v\cdot\nabla v)$  and  thus we deduce that $v$ satisfies the incompressible Euler system.
\end{proof}

\section{Critical regularities} \label{critik}
In this section we will study the system \eqref{eqs:3} with initial data lying in the  critical Besov \mbox{spaces $B_{2,1}^{\frac52}$.} There are some additional difficulties compared to the subcritical case that we can summarize in two points. The first one concerns the Beal-Kato-Majda criterion which is out of use and the second one is related to the extension of the interpolation argument used in Proposition \ref{stt1} to the critical regularity. We will start with a strong version of Theorem \ref{thm11} in which we require the initial data to be in a more regular   
space of type $B_{2,1}^{\frac52,\Psi}$.  We are able to prove the desired result for any slowly growth of $\Psi$. Roughly speaking, to get the incompressible limit and quantify the lower bound of the lifespan,  the function $\Psi$ must  only increase to infinity. The study of the case $\Psi\equiv1$ which is the subject of the  Theorem \ref{thm11} will be deduced from the Corollary \ref{corza1} and the Theorem \ref{thm111}.
\begin{thm}\label{thm111}
Let $\Psi \in \mathcal{U}_\infty$ and $\{(\vepsiloninit,\cepsiloninit)_{0<\varepsilon\le1}\}$ be a $B_{2,1}^{\frac52,\Psi}$-bounded family of axisymmetric initial data, that is
$$
\sup_{0<\varepsilon\le1}\|(\vepsiloninit,\cepsiloninit)\|_{B_{2,1}^{\frac52,\Psi}}<+\infty.
$$
Then the system \eqref{eqs:3} has a unique solution $(\vepsilon,\cepsilon)\in C([0,T_\varepsilon[; B_{2,1}^{\frac52,\Psi}),$ with
$$
T_\varepsilon\geq C_0\log\log\log\Phi(\varepsilon):=\tilde T_\varepsilon, \quad \hbox{ if}\quad 0<\varepsilon<<1,
$$
where 
$$
\Phi(x):= x^{\frac{1}{2r}}+\frac{1}{\Psi\big(\log\big(\frac{1}{x^{\frac{1}{2r}}}\big)\big)},\quad  x\in]0,1]
$$ 
and $r$ is a free parameter belonging to $]2,+\infty[$. Moreover we have for small values of $\varepsilon$,
     \begin{equation}\label{disp23}
   \|(\diver\vepsilon,\nabla \cepsilon)\|_{L^1_{\tilde T_\varepsilon} L^\infty}\le C_0 \Phi^{\frac13}(\varepsilon).
        \end{equation}
Assume in addition that
the incompressible parts $(\mathcal{P}\vepsiloninit)$ converge in $L^2$ to some $v_0.$ Then 
   the incompressible parts of the solutions tend to the Kato
   solution of the system $(\ref{eq:45})$:
   \begin{equation*}
     \mathcal{P} \vepsilon \rightarrow v
     \quad \text{ in } L^\infty_{loc} (\RR ^{ +} ; L ^2 ).
   \end{equation*}

\end{thm}
Before giving the proof we will state some remarks.
\begin{Remas}
\begin{enumerate} {\it 
\item  When $\Psi$ has at most an exponential growth at infinity, that is $\Psi(x) \leq  e^{C x}, $ then according to  Theorem $\ref{thm111}$ we get that 
 $$
 T_\varepsilon\geq C_0\log\log\big|\log{\Psi\big(\log\big(\frac{1}{\varepsilon^{\frac{1}{2r}}}\big)\big)}\big|.
 $$ 
 We observe that in the case where $\Psi$ has exactly an exponential growth $\Psi(x)=2^{ax} $ then the  Besov space $B_{2,1}^{s,\Psi}$ reduces to the Besov space $B_{2,1}^{s+a}.$ With this choice  the regularity of the initial data in Theorem $\ref{thm111}$  becomes subcritical  and what we obtain is in agreement  with Theorem $\ref{thm1}$; we get especially the same result on the lower bound of the lifespan.
 \item When $\Psi$ has a polynomial growth  that is $\Psi(x)=(x+2)^\alpha, \alpha>0$ then it is clear that $\Psi\in \mathcal{U}_\infty$ and moreover
 $$
 T_\varepsilon\geq C_0\log\log\log\frac1\varepsilon.
 $$
 }
 \end{enumerate}
\end{Remas}
 For the proof of Theorem \ref{thm111} we shall  examine only the estimate of the lower bound of $T_\varepsilon$ and the estimate \eqref{disp23}. The proof of the incompressible limit can be done similarly to the proof of \mbox{Theorem \ref{thm1}.} We point out that   the analysis in the critical case is rather difficult compared to the subcritical case and we mainly distinguish two  difficulties that one ought to face.
 
  The first one is related  to the criticality of Strichartz estimate for $\|(\diver \vepsilon,\nabla\cepsilon)\|_{L^1_tL^\infty}$. This quantity has the same scale of  the critical Besov space $B_{2,1}^{\frac52}$ and to circumvent this problem we move to a slightly smooth space $B_{2,1}^{\frac52,\Psi}$ but so close to $B_{2,1}^{\frac52}$.  We  will see later in the proof that the additional slow decreasing of the Fourier modes  are needed in order to interpolate the involved Strichartz  norm between the energy norm $\|(\vepsilon,\cepsilon)\|_{L^\infty_tB_{2,1}^{\frac52,\Psi}}$ and  a  Strichartz estimate for a lower norm of \mbox{type $\|(\nabla\Delta^{-1}\diver \vepsilon,\cepsilon)\|_{L^1_tL^\infty}$.   }
  
  The second difficulty that one should  to deal with has a close relation  with the criterion of   Beale-Kato-Majda which is not known to   work for the critical regularities. In other words, the bound  of $\|\Omegaepsilon(t)\|_{L^\infty}$ is not sufficient  to propagate initial regularities and thus one  should estimate a more stronger norm  $ \|\Omegaepsilon(t)\|_{B_{\infty,1}^0}$ instead.  This is the hard part of the proof and the geometric structure of the vorticity will play a significant role, especially the results discussed before in Proposition \ref{mim01}  and Proposition \ref{mim}.

\subsection{Proof of Theorem \ref{thm11}}
The proof of Theorem \ref{thm11} is a consequence of Theorem \ref{thm111} combined with Corollary \ref{corza1}. Indeed, from the assumption 
$$
\sum_{q}2^{\frac52 q}\sup_{0\le\varepsilon\le1}\|(\Delta_q\vepsiloninit,\Delta_q\cepsiloninit)\|_{L^2}<+\infty 
$$
and using Corollary \ref{corza1} we conclude the existence of  $\Psi\in \mathcal{U}_\infty$ such that the family $(\vepsiloninit,\cepsiloninit)_{0<\varepsilon\le1}$ is uniformly bounded in $B_{2,1}^{\frac52,\Psi}$. Now we can  just use Theorem \ref{thm111}.
\subsection{Logarithmic estimate}
Now let us examine a little further some interesting  properties of  the compressible transport model  given by
$$
\mbox{(CT)}\left\lbrace
\begin{array}{l}
\partial_t \Omega+(u\cdot\nabla)\Omega+\Omega\,\diver u=\Omega\cdot\nabla u\\
{\Omega}_{| t=0}=\Omega_0.
\end{array}
\right.
$$
This is the model governing the vorticity dynamics for the system \eqref{eqs:3} and  
it is worth while to study the linear growth of the norm $B_{\infty,1}^{0}$ which is crucial for the analysis of the critical case as it has been already pointed out. Our main result reads as follows.
\begin{thm}\label{Thm89}
Let $p\in[1,+\infty[$, $u$ be an axisymmetric smooth vector field without swirl and $\Omega$ be a smooth vector-valued solution of the equation \textnormal{(CT)}.
We assume that  $\Omega_0$ satisfies 
$$\diver\Omega_0=0\quad\hbox{and}\quad\Omega_0\times e_\theta=0.
$$ Then  
$$
\|\Omega(t)\|_{B_{\infty,1}^0}\le C\|\Omega_0\|_{B_{\infty,1}^0}e^{\|u^r/r\|_{L^1_tL^\infty}}\Big(1+e^{C\|\nabla u\|_{L^1_tL^\infty}}\|\diver{ u}\|_{ L^1_tB_{p,1}^{\frac3p}}^2\big)\Big(1+\int_0^t\|u(\tau)\|_{\textnormal{Lip}}d\tau\Big).
$$
with the notation $\|u\|_{\textnormal{Lip}}:=\|u\|_{L^\infty}+\|\nabla u\|_{L^\infty}$ and $C$ depends only on $p$.
%There exists a decomposition $(\tilde{\Omega}_{q})_{q\geq-1}$ of the vorticity $\omega$ such that
%\begin{enumerate}
%\item For every $t\in\RR_{+},\,\omega(t,x)=\sum_{q\geq-1}\tilde\omega_{q}(t,x).$
%\item For every $t\in\RR_{+}, \diver \tilde\omega_{q}(t,x)=0.$
%\item We have for $t\in\RR_+,\,p\in ]3,+\infty[$

%\end{enumerate}
\end{thm}
\begin{proof}
First of all, we point out that for the incompressible case, that is $\diver u=0,$ this result was proved in \cite{AHK} by using a tricky splitting of $\Omega$ combined with some geometric aspects  of axisymmetric flows. We shall use here the same approach of  \cite{AHK}, however  the lack of the incompressibility brings more  technical difficulties that we are unable to circumvent without using some refined geometric properties of axisymmetric vector fields. The results previously examined   in Proposition \ref{mim0} and Proposition \ref{mim01} are  very crucial at  different steps of the proof. 
   
Before going further in the details we will first  discuss the main idea of the proof. We localize in frequency the initial data and consider the solution of the same problem (CT). Then by linearity we can rebuild the solution of the initial problem by a superposition argument. The main information that we shall establish in order to get the logarithmic estimate is a  frequency decay property. To be more precise, let $q\geq -1$ and  denote by $\tilde{\Omega}_{q}$  the unique global vector-valued solution of the problem
\begin{equation}\label{td1}
\left\lbrace
\begin{array}{l}
\partial_t \tilde{\Omega}_{q}+(u\cdot\nabla)\tilde{\Omega}_{q}+\tilde{\Omega}_{q}\,\diver u=\tilde\Omega_{q}\cdot\nabla u\\
{\tilde\Omega_q}{_{|t=0}}=\Delta_{q}\Omega_{0}.
\end{array}
\right.
\end{equation} 
Since $\diver \Delta_{q}\Omega_{0}=\Delta_q\diver\Omega_0=0,$ then it follows from Proposition \ref{mim} that 
$\diver\tilde\Omega_{q}(t,x)=0.$ Using the linearity of the problem and the uniqueness of the solutions we get the following decomposition\begin{equation}\label{LU}
\Omega(t,x)=\sum_{q\geq-1}\tilde{\Omega}_{q}(t,x).
\end{equation}
%Remark that the desired estimate is equivalent to 
%$$
%\|\www_{q}(t)\|_{B_{\infty,\infty}^{\pm\frac12}}\lesssim \|\Delta_{q}\omega_{0}\|_{B_{\infty,\infty}^{\pm\frac12}} e^{CU(t)}.
%$$
%First we will show a more precise estimate
%$$
%\|\Delta_{j}\www_{q}(t)\|_{L^\infty}\leq C2^{j-q}e^{CV_p(t)}\|\Delta_{q}\omega_{0}\|_{L^\infty}
%$$
 We wish now to prove some frequency  decay for $\tilde\Omega_q(t).$  More precisely we will establish the following estimate: let $\psi$ be the flow associated to  $u$ and $F$ be the function defined by  $$
 F(t,x)=\int_0^t(\diver u)\big(\tau,\psi(\tau)\circ\psi^{-1}(t,x)\big)d\tau.
 $$
 Then for $p\in[1,6[$ and  for every  $j,q\geq-1$ we have
 \begin{equation}\label{kor}
\|\Delta_j(e^{F(t)}\tilde\Omega_q(t) )\|_{L^\infty}\le C2^{-\frac12|j-q|}\|\Delta_q\Omega_0\|_{L^\infty}\exp\big\{C\big(\|\nabla u\|_{L^1_tL^\infty}+\|\diver u\|_{L^1_tB_{p,1}^{\frac3p}}\big)\big\}.
 \end{equation}
  In the language of Besov spaces, this estimate is equivalent to 
 \begin{equation}\label{korr1}
  \|e^{F(t)}\tilde\Omega_q(t) )\|_{B_{\infty,\infty}^{\pm\frac12}}\le C\|\Delta_q\Omega_0\|_{B_{\infty,\infty}^{\pm\frac12}}\exp\big\{C\big(\|\nabla u\|_{L^1_tL^\infty}+\|\diver u\|_{L^1_tB_{p,1}^{\frac3p}}\big)\big\}.
  \end{equation}
  We remark that in the statement of Theorem \ref{Thm89} we require $p$ to belong to $[1,+\infty[$ which is not the case in the above claim \eqref{korr1}. To do so, we should work with a regularity of size $\eta\in]0,1[$ instead of $\frac12$ and the condition on $p$ will be $-\eta+\frac3p>0.$ 
Before going further into the detail we shall recall some classical properties for the flow $\psi$. It is defined by the  integral equation
$$
\psi(t,x)=x+\int_0^tv(\tau,\psi(\tau,x))d\tau, \, x\in\RR^3, t\geq 0.
$$
The next estimates for the  flow and its inverse $\psi^{-1}$ are well-known
\begin{equation}\label{Eqz33}
\|\nabla\psi^{\pm 1}(t)\|_{L^\infty}\le e^{\int_0^t\|\nabla v(\tau)\|_{L^\infty}d\tau},\quad \|\nabla\big(\psi(\tau)\circ\psi^{-1}(t)\big)\|_{L^\infty}\le e^{|\int_\tau^t\|\nabla v(\tau)\|_{L^\infty}d\tau|}.
\end{equation}

  The propagation of the negative regularity is  much easier than the positive one which needs more elaborated analysis.
 We apply first the
 Proposition \ref{Lems2}, yielding to
\begin{eqnarray}\label{bs10}
\nonumber e^{-CV_p(t)}\|\www_{q}(t)\|_{B_{\infty,\infty}^{-\frac12}}&\lesssim&\|\Delta_{q}\Omega_0\|_{B_{\infty,\infty}^{-\frac12}}+\int_0^te^{-CV_p(\tau)}\|\tilde{\Omega}_{q}^1\,\diver u\|_{ B_{\infty,\infty}^{-\frac12}}d\tau\\
&+&\int_{0}^te^{-CV_p(\tau)}\|\www_{q}\cdot\nabla u(\tau)\|_{B_{\infty,\infty}^{-\frac12}}d\tau.
\end{eqnarray}
with the following notations,
$$V_p(t):=\|u\|_{L^1_t\textnormal{Lip}}+\|\diver u\|_{L^1_tB_{p,1}^{\frac3p}},\,\, p\in[1,6[;\quad \|u\|_{\textnormal{Lip}}:=\|u\|_{L^\infty}+\|\nabla u\|_{L^\infty}.
$$
The first integral term of the right-hand side can be estimated as follows: for $p\in[1,6[$
$$
\int_0^te^{-CV_p(\tau)}\|\tilde{\Omega}_{q}^1\,\diver u\|_{ B_{\infty,\infty}^{-\frac12}}d\tau\le C\int_0^te^{-CV_p(\tau)}\|\diver u\|_{B_{p,1}^{\frac3p}}\|\tilde{\Omega}_{q}^1\|_{ B_{\infty,\infty}^{-\frac12}}d\tau
$$
We have used in the preceding estimate the following law product: for $s<0,\, p\in[1,+\infty[$ such \mbox{that $s+\frac3p>0$} we have
\begin{equation}\label{prod1}
\|fg\|_{B_{\infty,\infty}^{s}}\le C\|f\|_{B_{p,1}^{\frac3p}}\|g\|_{B_{\infty,\infty}^{s}}.
\end{equation}
The proof of this estimate can be done for example through the use of Bony's decomposition. Let us now come back to \eqref{bs10} and examine the last integral term. Using Bony's decomposition, we get
\begin{eqnarray*}
\|\www_{q}\cdot\nabla u\|_{B_{\infty,\infty}^{-\frac12}}&\leq&\|T_{\www_{q}}\cdot\nabla u\|_{B_{\infty,\infty}^{-\frac12}}+\|T_{\nabla u}\cdot\www_{q}\|_{B_{\infty,\infty}^{-\frac12}}\\
&+&\|\mathcal{R}\big({\www_{q}}\cdot\nabla, u\big)\|_{B_{\infty,\infty}^{-\frac12}}\\
&\lesssim&\|\nabla u\|_{L^\infty}\|\www_{q}\|_{B_{\infty,\infty}^{-\frac12}}+\|\mathcal{R}\big({\www_{q}}\cdot\nabla, u\big)\|_{B_{\infty,\infty}^{-\frac12}}.
\end{eqnarray*}
Since $\diver\www_{q}(t)=0,$ then the remainder term can be treated as follows
\begin{eqnarray*}
\|\mathcal{R}\big({\www_{q}}\cdot\nabla, u\big)\|_{B_{\infty,\infty}^{-\frac12}}&=&\|\diver\mathcal{R}\big({\www_{q}}\otimes, u\big)\|_{B_{\infty,\infty}^{-\frac12}}\\
&\lesssim&\sup_{k}2^{k\over2}\sum_{j\geq k-3
}\|\Delta_{j}\www_{q}\|_{L^\infty}\|\widetilde\Delta_{j}u\|_{L^\infty}\\
&\lesssim& \|\www_{q}\|_{B_{\infty,\infty}^{-\frac12}}\|u\|_{B_{\infty,\infty}^1}\\
&\lesssim&\|\www_{q}\|_{B_{\infty,\infty}^{-\frac12}}\|u\|_{\textnormal{Lip}}
\end{eqnarray*}
where we have used the embedding $\textnormal{Lip}\hookrightarrow B_{\infty,\infty}^1$. Combining these estimates yields to
\begin{equation}\label{prs11}
\|\www_{q}\cdot\nabla u\|_{B_{\infty,\infty}^{-\frac12}}\lesssim\|\www_{q}\|_{B_{\infty,\infty}^{-\frac12}}\|u\|_{\textnormal{Lip}}.
\end{equation}
Now by inserting this estimate into (\ref{bs10}) we obtain
\begin{eqnarray*}
e^{-CV_p(t)}\|\www_{q}(t)\|_{B_{\infty,\infty}^{-\frac12}}&\lesssim&\|\Delta_{q}\Omega_0\|_{B_{\infty,\infty}^{-\frac12}}\\
&+&\int_{0}^t\big(\|u(\tau)\|_{\textnormal{Lip}}+\|\diver u\|_{B_{p,1}^{\frac3p}}\big)e^{-CV_p(\tau)}\|\www_{q}(\tau)\|_{B_{\infty,\infty}^{-\frac12}}d\tau.
\end{eqnarray*}

Hence we obtain by Gronwall's inequality
\begin{eqnarray}\label{tao3}
\|\www_{q}(t)\|_{B_{\infty,\infty}^{-\frac12}}&\leq& C\|\Delta_{q}\Omega_0\|_{B_{\infty,\infty}^{-\frac12}}e^{CV_p(t)}.
%&\leq& C2^{-q}\|\Delta_{q}\omega_0\|_{L^\infty}e^{CU(t)}.
\end{eqnarray}
%\end{proof}
%\end{document}
%This gives by definition
%$$
%\|\Delta_{j}\www_{q}(t)\|_{L^\infty}\leq C2^{\frac{j-q}{2}}\|\Delta_{q}\omega_0\|_{L^\infty}e^{CU(t)}.$$
%which  achieves the proof of the theorem.
Now we claim that for any $p\in[1,6[$ we have
\begin{eqnarray}\label{ts1}
\|e^{F(t)}\tilde\Omega_q(t)\|_{B_{\infty,\infty}^{-\frac12}}&\le&C \|\Delta_q\Omega_0\|_{B_{\infty,\infty}^{-\frac12}}e^{CV_p(t)}.%\quad V_p(t):=\|\nabla u\|_{L^1_tL^\infty}+\|\diver u\|_{L^1_tB_{p,\infty}^{\frac3p}}.
%&\le& \|\Delta_q\Omega_0^1\|_{B_{\infty,\infty}^{-\frac12}}\big(1+e^{CV(t)}\big)
\end{eqnarray}
 Indeed, by \eqref{prod1} we get   for $p\in[1,6[$
% from the definition of $\eta_q$ we get
% $$
% \|\eta_q^1(t,\psi^{-1}(t,\cdot))\|_{B_{\infty,\infty}^{-\frac12}}=\|e^{W(t,\psi^{-1}(t,\cdot))} \www_q(t)\|_{B_{\infty,\infty}^{-\frac12}}.
% $$
% Next, we recall  the classical law product that one could obtain for example by using Bony's decomposition: Let $s<0, p\in[1,\infty[$ such that $s+\frac3p>0,$ then 
% $$
% \|uv\|_{B_{\infty,\infty}^s}\le C\|u\|_{B_{\infty,\infty}^s}\|v\|_{B_{p,1}^{\frac3p}}.
% $$
% Consequently we obtain for $p\in[1,6[$
 \begin{eqnarray}\label{tao2}
 \nonumber\|e^{F(t)} \www_q(t)\|_{B_{\infty,\infty}^{-\frac12}}&\le&\|\www_q(t)\|_{B_{\infty,\infty}^{-\frac12}}+\big\|\big(e^{F(t)}-1\big) \www_q(t)\big\|_{B_{\infty,\infty}^{-\frac12}}\\
 &\lesssim&\|\www_q(t)\|_{B_{\infty,\infty}^{-\frac12}}+\big\|e^{F(t)}-1\big\|_{B_{p,1}^{\frac3p}} \|\www_q(t)\big\|_{B_{\infty,\infty}^{-\frac12}}.
 \end{eqnarray}
 We shall use  the following estimate 
 \begin{eqnarray}\label{tao1}
 \nonumber\big\|e^{F(t)}-1\big\|_{B_{p,1}^{\frac3p}}&\lesssim&  \big\|F(t)\big\|_{B_{p,1}^{\frac3p}} e^{C\|F(t)\|_{L^\infty}}\\
 &\lesssim&  \big\|F(t)\big\|_{B_{p,1}^{\frac3p}} e^{C\|\diver{u}\|_{L^1_tL^\infty}}
\end{eqnarray}
 which can be deduced for instance from the result: There exists $C>0$ such that for $u\in B_{p,1}^{\frac3p}$ and $n\in\NN$ we have
 $$
 \|u^n\|_{ B_{p,1}^{\frac3p}}\le C^{n-1}\|u\|_{L^\infty}^{n-1}\|u\|_{ B_{p,1}^{\frac3p}}.
 $$
 To estimate $F(t)$ we will use 
 the  following composition law, see for instance \cite{MR86j:46026} : for $s\in]0,1[,p,r\in[1,\infty],  f\in B_{p,r}^{s}$ and $\psi$ a diffeomorphism then  we have
 $$
 \|f\circ\psi\|_{B_{p,r}^{s}}\le C\|f\|_{B_{p,r}^{s}}\big(1+\|\nabla\psi\|_{L^\infty}^s\big).
 $$
 Applying this result combined with \eqref{Eqz33} we get for $p\in]3,\infty[$
\begin{eqnarray*}
 \|F(t)\|_{ B_{p,1}^{\frac3p}}&\le&  \int_0^t\|\diver{ u(\tau)}\|_{ B_{p,1}^{\frac3p}}\big(1+\|\nabla\big( \psi(\tau,\psi^{-1}(t,\cdot))\big)\|_{L^\infty}^{\frac3p}\big)d\tau\\
 &\lesssim&  \int_0^t\|\diver{ u(\tau)}\|_{ B_{p,1}^{\frac3p}} e^{\int_{\tau}^t\|\nabla u(s)\|_{L^\infty}ds}d\tau.
 \end{eqnarray*}
Combining this estimate with \eqref{tao1} yields to
\begin{equation}\label{kao45}
\|e^{F(t)}-1\big\|_{B_{p,1}^{\frac3p}}\le C  e^{C\|\nabla u\|_{L^1_tL^\infty}}\int_0^t\|\diver{ u(\tau)}\|_{ B_{p,1}^{\frac3p}}d\tau.
\end{equation}
Putting together this estimate,  \eqref{tao2} and \eqref{tao3} 
\begin{eqnarray*}
 \nonumber\|e^{F(t)}\tilde\Omega_q(t)\|_{B_{\infty,\infty}^{-\frac12}}&\lesssim&\|\www_q(t)\|_{B_{\infty,\infty}^{-\frac12}}\Big(1+  e^{C\|\nabla u\|_{L^1_tL^\infty}}\|\diver{ u}\|_{ L^1_tB_{p,1}^{\frac3p}}\Big)\\
 \nonumber&\lesssim&  \|\Delta_q\Omega_0\|_{B_{\infty,\infty}^{-\frac12}}e^{CV_p(t)}\Big(1+  e^{C\|\nabla u\|_{L^1_tL^\infty}}\|\diver{ u}\|_{L^1_t B_{p,1}^{\frac3p}}\Big)\\
&\le& C \|\Delta_q\Omega_0\|_{B_{\infty,\infty}^{-\frac12}}e^{CV_p(t)}.
 \end{eqnarray*}

This completes the proof of \eqref{korr1} in the case of the negative regularity. It remains to show the estimate for $B_{\infty,\infty}^{\frac12}$ which is the hard part of the proof. But before  giving precise discussions about  the difficulties, we shall  rewrite under a suitable way   the stretching term of  the equation (\ref{td1}). For this purpose we use 
 Proposition \ref{mim0}-$(3)$ leading to $(\Delta_{q}\Omega_{0})\times e_{\theta}=0.$ Now according to the   Proposition \ref{mim}-$(2)$ this geometric property is conserved through the time, that is $\www_{q}(t)\times e_{\theta}=0,$ and furthermore the equation \eqref{td1} becomes
\begin{equation}\label{td2}
\left\lbrace
\begin{array}{l}
\partial_t \tilde{\Omega}_{q}+(u\cdot\nabla)\tilde{\Omega}_{q}+\tilde{\Omega}_{q}\,\diver u=\frac{u^r}{r}\tilde\Omega_{q}\\
{\tilde\Omega_q}{_{|t=0}}=\Delta_{q}\Omega_{0}.
\end{array}
\right.
\end{equation}
Applying the maximum principle and using Gronwall's inequality we obtain
\begin{eqnarray}\label{maxx1}
\|\tilde \Omega_{q}(t)\|_{L^\infty}
&\leq& \|\Delta_{q}\Omega_{0}\|_{L^\infty}e^{ 
\|(\diver u,{u^r/r})\|_{L^1_tL^\infty}}.
%\\
%\label{ben0}
%&\leq&\|\Delta_{q}\omega_{0}\|_{L^\infty}e^{Ct\|\omega_{0}\|_{L^{3,1}}}.
\end{eqnarray}
From the geometric constraint  $\www_{q}(t)\times e_{\theta}=0$ we see that the solution $\www_{q}$ has two components in the Cartesian basis, $\tilde{\Omega}_{q}=(\tilde{\Omega}_{q}^1,\tilde{\Omega}_{q}^2,0),$ and the mathematical analysis will be  the same for both components. Thus we will just focus   on the treatment of  the first component $\tilde{\Omega}_{q}^1$.

From the identity ${{u^r}\over{ r}}={{u^1}\over{x_{1}}}={{u^2}\over{ x_{2}}},$ which is an easy consequence of $u^\theta=0$, it is clear that 
the  \mbox{function $\tilde{\Omega}_{q}^1$} satisfies the equation
$$
\left\lbrace
\begin{array}{l}
\partial_t \tilde{\Omega}_{q}^1+(u\cdot\nabla)\tilde{\Omega}_{q}^1+\tilde{\Omega}_{q}^1\,\diver u=u^2{{\tilde{\Omega}_q^1}\over x_{2}},\\
{\tilde\Omega_q}^1{_{|t=0}}=\Delta_{q}\Omega_{0}^1.
\end{array}
\right.
$$
 Now we shall discuss the most difficulties encountered when we would like to propagate positive regularities and explain how to circumvent them.
In the work of \cite{AHK}, the velocity is divergence-free and thus the persistence of the the regularity $B_{\infty,\infty}^s, 0<s\leq1$ is guaranteed by the special structure of the stretching term $\frac{v^r}{r}\tilde{\Omega}_q$ where the axisymmetry of the velocity plays a central role. In our case the velocity is not incompressible and the new term $\tilde{\Omega}_q\,\diver u$ does not in general belong to  any space of type $B_{\infty,\infty}^s, s>0$ because in the context of critical regularities the quantity $\diver u$ is only allowed to be in $L^\infty$ or in any Banach space with the same scaling. Our idea is to filtrate the compressible part leading to  a new function governed by an equation where all the terms have positive regularity. But we need to be careful with the filtration procedure, its commutation with the transport part leads to  a bad term and that's why we have to extract the compressible part after using the Lagrangian coordinates. For the simplicity, we will use the following notation: we denote \mbox{by $t\mapsto \big(x_1(t),x_2(t),x_3(t)\big)=\psi(t,x),$} the path of the individual particle located initially at the position $x$. Now, we introduce the new functions:
$$\eta_q(t,x)=e^{G(t,x)} \www_q(t,\psi(t,x)),
$$
and
$$ G(t,x)= \int_0^t(\diver u)(\tau,\psi(\tau,x))d\tau,\quad U(t,x)=u(t,\psi(t,x)).
$$ 
As before, we shall restrict attention to the treatment of the first component of $ \tilde{\Omega}_q$, the second one can be teated in a similar way. It is not hard to see that the function $\eta_q^1$ satisfies  the equation
\begin{equation}\label{kao1}
\partial_t\eta_q^1(t,x)=U^2(t,x)\frac{\eta_q^1(t,x)}{x_2(t)},\,\eta_q(0,x)=\Delta_q\Omega_0^1(x).
\end{equation}
Accordingly, we deduce from this equation that the quantity $t\mapsto\frac{\eta_q^1(t,x)}{x_2(t)}$ is conserved. Indeed, differentiate this function with respect to $t$, then use the equation of the flow we find
\begin{eqnarray*}
\partial_t\big(\frac{\eta_q^1(t,x)}{x_2(t)}\big)&=&\frac{x_2(t)\partial_t\eta_q^1(t,x)-\eta_q^1(t,x)x_2^\prime(t)}{x_2^2(t)}\\
&=&\frac{U^2(t,x)\eta_q^1(t,x)-\eta_q^1(t,x)u^2(t,\psi(t,x))}{x_2^2(t)}\\
&=&0.
\end{eqnarray*}
It follows that
\begin{equation}\label{eqid}
\frac{\eta_{q}^1(t,x)}{x_{2}(t)}=\frac{\Delta_q\Omega_0^1(x)}{x_{2}}\cdot
\end{equation}
Now, we are going to estimate $\eta_q^1$ in the Besov space $B_{\infty,\infty}^{\frac12}$ and thus we prove the remaining case of  \eqref{korr1} concerning the positive regularity. Applying Taylor's formula to the equation \eqref{kao1} and using \eqref{eqid} we obtain
\begin{eqnarray}\label{Eqid12}
\nonumber\|\eta_q^1(t)\|_{B_{\infty,\infty}^{\frac12}}&\le& \|\eta_q^1(0)\|_{B_{\infty,\infty}^{\frac12}}+\int_0^t \Big\|U^2(\tau)\frac{\eta_q^1(\tau)}{x_2(\tau)}\Big\|_{B_{\infty,\infty}^{\frac12}}d\tau\\
&\le&\|\Delta_q\Omega_0^1\|_{B_{\infty,\infty}^{\frac12}}+\int_0^t \Big\|U^2(\tau)\frac{\eta_q^1(\tau)}{x_2(\tau)}\Big\|_{B_{\infty,\infty}^{\frac12}}d\tau.
\end{eqnarray}
%We have used the following result generalized in the following lemma and whose proof is postponed until  the end of this section.
%\begin{lem}
%Let $s\in]-1,0[\cup]0,+\infty[$ then 
%$$
%\Big\|\frac{\Delta_q\Omega_0^1}{x_{2}}  \Big\|_{B_{\infty,\infty}^s}\le C_s \|{\Delta_q\Omega_0^1}\|_{B_{\infty,\infty}^{s+1}}
%$$
%\end{lem}
To estimate the integrand in the right-hand side in the preceding  estimate, we  shall make use of the paradifferential calculus throughout  Bony's decomposition. So, we get by definition and from the triangular inequality,
\begin{eqnarray*}
\Big\|U^2\frac{\eta_{q}^1}{x_{2}(t)}\Big\|_{B_{\infty,\infty}^{\frac12}}\leq\Big\|T_{\frac{\eta_{q}^1}{x_{2}(t)}}U^2\Big\|_{B_{\infty,\infty}^{\frac12}}+
\Big\|T_{U^2}\frac{\eta_{q}^1}{x_{2}(t)}\Big\|_{B_{\infty,\infty}^{\frac12}}+
\Big\|\mathcal{R}\big(U^2,\frac{\eta_{q}^1}{x_{2}(t)}\big)\Big\|_{B_{\infty,\infty}^{\frac12}}.
\end{eqnarray*}
The estimate of the first paraproduct and the remainder term are easy compared to the second paraproduct for which we will use some sophisticated  analysis. Now for the first paraproduct we write by definition and by the continuity of the Fourier cut-offs,
\begin{eqnarray}\label{bs1}
\nonumber\Big\|T_{\frac{\eta_{q}^1}{x_{2}(t)}}U^2\Big\|_{B_{\infty,\infty}^{\frac12}}&=& \sup_{k\geq-1}2^{\frac{k}{2}}\sum_{j\in\NN}\|\Delta_k(S_{j-1}(\eta_{q}^1/x_{2}(t)) \Delta_j U^2)\|_{L^\infty}\\
\nonumber&=& \sup_{k\geq-1}2^{\frac{k}{2}}\sum_{|j-k|\le 4}\|\Delta_k(S_{j-1}(\eta_{q}^1/x_{2}(t)) \Delta_j U^2)\|_{L^\infty}\\
\nonumber&\lesssim& \sup_{j\in \NN} 2^{\frac{j}{2}}\|\Delta_j U^2\|_{L^\infty}\|S_{j-1}(\eta_{q}^1/x_{2}(t)\|_{L^\infty} \\
\nonumber&\lesssim& \sup_{j\in \NN} 2^{j} \|\Delta_j U^2\|_{L^\infty} \sup_{j\in \NN} 2^{-\frac{j}{2}} \Big\|S_{j-1}\big(\frac{\eta_{q}^1}{x_{2}(t)}\big)\Big\|_{L^\infty}\\
&\lesssim&\| \nabla U\|_{L^\infty}\Big\|\frac{\eta_{q}^1}{x_{2}(t)}\Big\|_{B_{\infty,\infty}^{-\frac12}},
\end{eqnarray}
where we have used  Bernstein inequality in the last line. Now, using the identity \eqref{eqid}  we get
\begin{eqnarray*}
\Big\|\frac{\eta_{q}^1}{x_{2}(t)}\Big\|_{B_{\infty,\infty}^{-\frac12}}&=&\Big\|\frac{\Delta_q\Omega_0^1}{x_{2}}\Big\|_{B_{\infty,\infty}^{-\frac12}}\\
&\lesssim& 2^{\frac12q}\|\Delta_q\Omega_0\|_{L^\infty}\approx \|\Delta_q\Omega_0\|_{B_{\infty,\infty}^{\frac12}}
\end{eqnarray*}
The proof of  last inequality can be done as follows: applying Proposition \ref{mim0}-$(3)$ we get the  identity $\Delta_q\Omega_0^1(x_{1},0,x_3)=0$. This yields in view of  Taylor's expansion
$$
\Delta_q\Omega_0^1(x_{1},x_{2},x_3)=x_{2}\int^1_{0}(\partial_{x_{2}} \Delta_q\Omega_0^1)(x_{1},\tau x_{2},x_3)d\tau
$$
and hence we obtain by Lemma \ref{dilatation} and Bernstein inequality,
\begin{eqnarray}\label{log1}
\nonumber\Big\|\frac{\Delta_q\Omega_0^1}{x_{2}}\Big\|_{B_{\infty,\infty}^{-\frac12}}&\le&\int_0^1\|(\partial_{2} \Delta_q\Omega_0^1)(\cdot,\tau \cdot,\cdot)\|_{B_{\infty,\infty}^{-\frac12}}d\tau\\
\nonumber&\le& C\|\partial_{2} \Delta_q\Omega_0^1\|_{B_{\infty,\infty}^{-\frac12}}\int_0^1\tau^{-\frac12}d\tau\\
 &\le&C\| \Delta_q\Omega_0^1\|_{B_{\infty,\infty}^{\frac12}}.
%&\le&\|\Delta_q\omega_0\|_{B_{\infty,\infty}^{\frac12}} .
\end{eqnarray}
Therefore we get by \eqref{bs1}
$$
\Big\|T_{\frac{\eta_{q}^1}{x_{2}(t)}}U^2\Big\|_{B_{\infty,\infty}^{\frac12}}\lesssim\|\nabla U\|_{L^\infty} \|\Delta_q\Omega_0\|_{B_{\infty,\infty}^{\frac12}}.
$$ 
 Now according to Leibniz formula and \eqref{Eqz33}  we find
 %\end{proof}
 %\end{document}
\begin{eqnarray}\label{lz1}
\nonumber\|\nabla U(t)\|_{L^\infty}&\le& \|\nabla u(t)\|_{L^\infty}\|\nabla\psi(t)\|_{L^\infty}\\
&\le&  \|\nabla u(t)\|_{L^\infty}e^{ \|\nabla u\|_{L^1_tL^\infty}}.
\end{eqnarray}
Hence it follows that
\begin{eqnarray}\label{para45}
\nonumber\Big\|T_{\frac{\eta_{q}^1}{x_{2}(\cdot)}}U^2\Big\|_{L^1_tB_{\infty,\infty}^{\frac12}}&\le& C \|\nabla u\|_{L^1_tL^\infty}e^{ \|\nabla u\|_{L^1_tL^\infty}}\|\Delta_q\Omega_0\|_{B_{\infty,\infty}^{\frac12}}\\
&\le& Ce^{C \|\nabla u\|_{L^1_tL^\infty}}\|\Delta_q\Omega_0\|_{B_{\infty,\infty}^{1\over2}}.
\end{eqnarray}
To estimate the  remainder term we use its definition combined with \eqref{log1},
\begin{eqnarray}\label{bs2}
\nonumber
\Big\|\mathcal{R}\big(U^2,\frac{\eta_{q}^1}{x_{2}(t)}\big)\Big\|_{B_{\infty,\infty}^{1\over2}}&\lesssim &\sup_{k\geq-1}\sum_{j\geq k-3}2^{\frac{k}{2}}\|\Delta_{j}U^2\|_{L^\infty}\Big\|\widetilde\Delta_{j}\big(\frac{\eta_{q}^1}{x_{2}}\big)\Big\|_{L^\infty}
\\
\nonumber&\lesssim &\|U\|_{B_{\infty,\infty}^1}\Big\|\frac{\eta_{q}^1}{x_{2}(t)}\Big\|_{B_{\infty,\infty}^{- \frac{1}{2}}}\\
&\lesssim& \|U\|_{B_{\infty,\infty}^1} \|\Delta_q\Omega_0\|_{B_{\infty,\infty}^{\frac12}}.%\quad V(t):=\|\nabla v\|_{L^1_tL^\infty}.
\end{eqnarray}
Straightforward computations  combined with the embedding $ 
\hbox{Lip}(\RR^3)\hookrightarrow B_{\infty,\infty}^1$  imply that
\begin{eqnarray*}
\|U(t)\|_{B_{\infty,\infty}^{1}}&\lesssim& \|U(t)\|_{L^\infty}+\|\nabla U(t)\|_{L^\infty}\\
&\lesssim& \|u(t)\|_{L^\infty}+\|\nabla u(t)\|_{L^\infty}\|\nabla\psi(t)\|_{^\infty}\\
&\lesssim& \|u(t)\|_{L^\infty}+\|\nabla u(t)\|_{L^\infty}e^{\|\nabla u\|_{L^1_tL^\infty}}\\
&\lesssim&\|u(t)\|_{\textnormal{Lip}} e^{\|u\|_{L^1_t\textnormal{Lip}}}.
\end{eqnarray*}
Plugging this estimate into \eqref{bs2} yields to
\begin{eqnarray}\label{para46}
\Big\|\mathcal{R}\big(U^2,\frac{\eta_{q}^1}{x_{2}(\cdot)}\big)\Big\|_{L^1_tB_{\infty,\infty}^{1\over2}}&\le& Ce^{C\|u\|_{L^1_t\textnormal{Lip}}}\|\Delta_q\Omega_0\|_{B_{\infty,\infty}^{1\over2}}.
\end{eqnarray}
Now let us examine the second paraproduct which  is more subtle and requires some special properties of  the axisymmetric vector field  $u.$ First we write by definition
$$
\Big\|T_{U^2}\frac{\eta_{q}^1}{x_{2}}\Big\|_{B_{\infty,\infty}^{1\over2}}\lesssim\sup_{j\in\NN}2^{j\over2}\Big\|S_{j-1}U^2\Delta_{j}\big(\frac{\eta_{q}^1}{x_{2}(t)}\big)\Big\|_{L^\infty}.
$$
At first sight the term $ \frac{\eta_{q}^1}{x_{2}(t)}$ consumes a derivative more than what we require and the best configuration is to carry this derivative to the velocity $U^2.$ To do this we will use the commutation structure between the frequency cut-off operator $\Delta_j$ and the singular multiplication by $\frac{1}{x_2}$.
According to the identity \eqref{eqid}  we write
\begin{eqnarray*}
S_{j-1}U^2(x)\Delta_{j}\big(\frac{\eta_{q}^1(t,x)}{x_{2}(t)}\big)&=&S_{j-1}U^2(x)\Delta_{j}\big(\frac{\Delta_q\Omega_0^1}{x_2}\big)\\
&=&\frac{S_{j-1}U^2(x)}{x_2}\, {\Delta_{j}\big({\Delta_q\Omega_0^1}}\big)+{S_{j-1}U^2(x)}{}\Big[\Delta_{j}, \frac{1}{x_2}\Big](\Delta_q\Omega_0^1)\\
%&=&\frac{S_{j-1}U^2(x)}{x_2}\, {\Delta_{j}\Delta_q\omega_0^1}+\frac{S_{j-1}U^2(x)}{x_2(t)}\Big[\Delta_{j}, x_{2}(t)\Big](\eta_{q}^1(x)/{x_{2}(t)})\\
%&=&S_{j-1}\big(U^2(x)/x_2(t)\big){\Delta_{j}\eta_{q}^1(x)}-{\Delta_{j}\eta_{q}^1(x)}[S_{j-1},\frac{1}{x_2(t)}]U^2\\&+&S_{j-1}U^2(x)\Big[\Delta_{j}, \frac1x_{2}(t)\Big]\eta_{q}^1\\
&:=&{\hbox{I}}_{j}(t,x)+{\hbox{II}}_{j}(t,x).
\end{eqnarray*}
The estimate of the first term can be done as follows
\begin{eqnarray*}
\|{\hbox{I}}_{j}(t)\|_{L^\infty}%&\lesssim&%\Big\|\frac{S_{j-1}U^2(t,x)}{x_2(t)}\Big\|_{L^\infty}\|\Delta_{j}\eta_{q}^1\|_{L^\infty}\\
&\lesssim&\Big\|\frac{S_{j-1}U^2(t,x)}{x_2}\Big\|_{L^\infty}\|\Delta_{j}\Delta_q\Omega_0^1\|_{L^\infty}.
%&\lesssim&\|u^2(t,x)/x_2\|_{L^\infty}\|\Delta_{j}\eta_{q}^1\|_{L^\infty}\\
%&\lesssim&\|(u^r/r)(t)\|_{L^\infty}\|\Delta_{j}\eta_{q}^1(t)\|_{L^\infty}.
\end{eqnarray*}
Using Proposition \ref{mim01}-$(4)$ we get $S_{j-1}U^2(x_1,0,x_3)=0$ and thus by Taylor's formula and \eqref{lz1} we obtain
\begin{eqnarray}\label{rz1}
\nonumber\Big\|\frac{S_{j-1}U^2(t,x)}{x_2}\Big\|_{L^\infty}&\lesssim&\|\nabla S_{j-1}U^2\|_{L^\infty}\\
\nonumber&\lesssim&\|\nabla U^2\|_{L^\infty}\\
&\lesssim&\|\nabla u\|_{L^\infty}e^{C\|\nabla u\|_{L^1_tL^\infty}}.
%&\leq& C\|u(t)\|_{B_{\infty,1}^0}e^{C\|u\|_{L^1_tB_{\infty,1}^0}}.
\end{eqnarray}
Since $\Delta_{j}\Delta_q\omega_0=0$ for $|j-q|\geq2$ then we deduce 
\begin{equation}\label{bs3}
\sup_{j}2^{j\over2}\|\hbox{I}_{j}(t)\|_{L^\infty}\le C\|\nabla u(t)\|_{L^\infty}e^{C\|\nabla u\|_{L^1_tL^\infty}}\|\Delta_q\Omega_0\|_{B_{\infty,\infty}^{1\over2}}.
\end{equation}
For the commutator term $\hbox{II}_{j}$ we write by definition
\begin{eqnarray*}
\hbox{II}_{j}(t,x)&=&\frac{S_{j-1} U^2(t,x)}{x_{2}}\,\, 2^{3j}\int_{\RR^3}h(2^j(x-y))(x_{2}-y_{2})\frac{\Delta_q\Omega_0^1(y)}{y_{2}}dy\\
% &=&-\frac{\Delta_{j}\eta_q^1(x)}{x_{2}(t)}\,\, 2^{3j}\int_{\RR^3}(x_2-y_2)h(2^j(x-y))\frac{x_{2}(t)-y_{2}(t)}{x_2-y_2}\frac{u^2(t,\psi(t,y))}{y_{2}(t)}dy ,
&=&2^{-j}\frac{S_{j-1}U^2(x)}{x_{2}}\, 2^{3j}\tilde{h}(2^j\cdot)\star(\frac{\Delta_q\Omega_0^1}{y_{2}})(x)
\end{eqnarray*}
with $\tilde{h}(x)=x_{2}h(x).$ Now we claim that for every $f\in\mathcal{S}'$ we have
%\end{proof}

%\end{document}
$$
2^{3j}\tilde{h}(2^j\cdot)\star f=\sum_{|j-k|\leq 1}2^{3j}\tilde{h}(2^j\cdot)\star\Delta_{k}f.
$$
Indeed, we have  $\widehat{\tilde h}(\xi)=i\partial_{\xi_{2}}\widehat{h}(\xi)=i\partial_{\xi_{2}}\varphi(\xi).$ It follows that $supp\, \widehat{\tilde h}\subset supp\, \varphi.$ So we get $2^{3j}\tilde{h}(2^j\cdot)\star\Delta_{k}f=0,$ for $|j-k|\geq 2.$
This leads in view of Young inequality, \eqref{rz1}  and \eqref{log1}  to
\begin{eqnarray}\label{bs4}
\nonumber\sup_{j\in\NN}2^{j\over2}\|\hbox{II}_{j}\|_{L^\infty}&\lesssim&\sup_{j\in\NN}\Big\|\frac{S_{j-1}U^2}{x_{2}}\Big\|_{L^\infty}\sum_{|j-k|\leq 1}2^{-\frac{k}{2}}\|\Delta_{k}(\frac{\Delta_q\Omega_0^1}{x_{2}})\|_{L^\infty}\\
\nonumber&\lesssim& \|\nabla U\|_{L^\infty}\Big\|\frac{\Delta_q\Omega_0^1}{x_{2}}\Big\|_{B_{\infty,\infty}^{-\frac12}}\\
&\lesssim&\|\nabla u(t)\|_{L^\infty}e^{\|\nabla u\|_{L^1_tL^\infty}}\|\Delta_q\Omega_0\|_{B_{\infty,\infty}^{\frac12}}.
\end{eqnarray}
%It follows from Young inequality and \eqref{rz1}  
%\begin{eqnarray*}
%\|\hbox{II}_{j}(t)\|_{L^\infty}&\lesssim&\Big\|\frac{S_{j-1}U^2(t,x)}{x_2}\Big\|_{L^\infty}2^{-j}\Big\|\frac{\Delta_q\Omega_0^1(t,x)}{x_{2}} \Big\|_{L^\infty}\\
%&\lesssim&\|\nabla U^2\|_{L^\infty}2^{-j}\|\nabla\Delta_q\Omega_0^1\|_{L^\infty}\\
%&\lesssim&2^{q-j}\|u(t)\|_{B_{\infty,1}^0}e^{C\|u\|_{L^1_tB_{\infty,1}^0}}\|\Delta_q\Omega_0^1\|_{L^\infty}.
%\end{eqnarray*}
%Therefore
%\begin{equation}\label{bs4}
%\sup_{j\in\NN}2^j\|\hbox{II}_{j}(t)\|_{L^\infty}\le C\|u(t)\|_{B_{\infty,1}^0}e^{C\|u\|_{L^1_tB_{\infty,1}^0}}\|\Delta_q\Omega_0^1\|_{B_{\infty,\infty}^1}.
%\end{equation}
Putting together \eqref{bs3} and \eqref{bs4} yields to
\begin{equation}\label{para47}
\Big\|T_{U^2}\frac{\eta_{q}^1}{x_{2}(\cdot)}\Big\|_{L^1_tB_{\infty,\infty}^{1\over2}}\le Ce^{C\|\nabla u\|_{L^1_tL^\infty}}\|\Delta_q\Omega_0^1\|_{B_{\infty,\infty}^{\frac12}}.
\end{equation}
Finally, combining  \eqref{para45},\eqref{para46} and \eqref{para47} we get 
$$
\Big\|U^2\frac{\eta_{q}^1}{x_{2}(\cdot)}\Big\|_{L^1_tB_{\infty,\infty}^{1\over2}}\lesssim Ce^{C\|u\|_{L^1_t\textnormal{Lip}}}\|\Delta_q\Omega_0^1\|_{B_{\infty,\infty}^{1\over2}}.
$$
Inserting this estimate into \eqref{Eqid12} we obtain
\begin{equation}\label{para50}
\|\eta_q^1(t)\|_{B_{\infty,\infty}^{1\over2}}\le Ce^{C\|u\|_{L^1_t\textnormal{Lip}}}\|\Delta_q\Omega_0^1\|_{B_{\infty,\infty}^{1\over2}}.
\end{equation}
Next, we shall  show how to derive the estimate \eqref{korr1} from the preceding one. First, we observe that
$$
e^{F(t,x)}\tilde\Omega_q(t,x)=\eta_q(t,\psi^{-1}(t,x)).
$$
This will be combined with the following  classical composition law 
$$
\big\|f\circ\psi^{-1}\big\|_{B_{\infty,\infty}^{1\over2}}\le C\|f\|_{B_{\infty,\infty}^{1\over2}}\big(1+ \|\nabla\psi^{-1}\|_{L^\infty}^{\frac12}\big)
$$ 
leading  
acording to \eqref{Eqz33} and \eqref{para50} to
\begin{eqnarray*}%\label{tss1}
\nonumber\|\eta_q^1(t,\psi^{-1}(t,\cdot))\|_{B_{\infty,\infty}^{1\over2}}&\le& C\|\eta_q^1\|_{B_{\infty,\infty}^{1\over2}}e^{C\|\nabla u\|_{L^1_tL^\infty}}\\
&\le&C\|\Delta_q\Omega_0^1\|_{B_{\infty,\infty}^{1\over2}} e^{C\|u\|_{L^1_t\textnormal{Lip}}}.
\end{eqnarray*}
This achieves the proof of \eqref{korr1}. According to \eqref{kor}, its local version reads as follows
 %Combining the estimates \eqref{ts1} and \eqref{tss1} and using the definition of Besov spaces we get: for  $p\in ]3,6[$ and  $j,q\geq-1$ 
\begin{equation}\label{div1}
\|\Delta_j(e^{F(t)}\tilde\Omega_q(t))\|_{L^\infty}\le C2^{-\frac12|j-q|} \|\Delta_q\Omega_0\|_{L^\infty}e^{C V_p(t)}.
\end{equation}
On the other hand, coming back to the equation \eqref{kao1}, taking the $L^\infty$-estimate
 and using Gronwall's inequality  we get 
  \begin{eqnarray*}
\nonumber\|e^{F(t)}\tilde\Omega_q(t)\|_{L^\infty}&=&\|\eta_q(t)\|_{L^\infty}\\
&\le&  \|\Delta_q\Omega_0\|_{L^\infty}e^{\int_0^t\|(U^2(\tau)/x_2(\tau)\|_{L^\infty}d\tau}.
%&\le&e^{2\|(\diver u,u^r/r)\|_{L^1_tL^\infty}}.
\end{eqnarray*}
Now, since $\psi$ is an homeomorphism and $u^2/x_2=u^r/r$ then
$$
\|(U^2(\tau)/x_2(\tau)\|_{L^\infty}=\|(u^r/r)(\tau)\|_{L^\infty}
$$
and thus we obtain
 \begin{eqnarray}\label{div3}
\|e^{F(t)}\tilde\Omega_q(t)\|_{L^\infty}
&\le&  \|\Delta_q\Omega_0\|_{L^\infty}e^{\|\frac{u^r}{r}\|_{L^1_tL^\infty}}.
%&\le&e^{2\|(\diver u,u^r/r)\|_{L^1_tL^\infty}}.
\end{eqnarray}
Set $\zeta(t,x)=e^{F(t,x)}\Omega(t,x),$ then by the  linearity of the  problem \eqref{td1} we get the decomposition
$$
\zeta(t,x)=\sum_{q\geq-1}e^{F(t,x)}\tilde\Omega_q(t,x).
$$
Let $N\in \NN$ be an integer that will be fixed later, then using \eqref{div1} and \eqref{div3} we find
\begin{eqnarray*}
\|\zeta(t)\|_{B_{\infty,1}^0}&\leq&\sum_{j,q\geq-1}\|\Delta_j\big(e^{F(t)}\tilde\Omega_q(t)\big)\|_{L^\infty}\\
&\le&
\sum_{|j-q|\le N}\|\Delta_j\big(e^{F(t)}\tilde\Omega_q(t)\big)\|_{L^\infty}+\sum_{|j-q|> N}\|\Delta_j\big(e^{F(t)}\tilde\Omega_q\big)(t)\|_{L^\infty}\\
&\lesssim&N\|\Omega_0\|_{B_{\infty,1}^0}e^{\|u^r/r\|_{L^1_tL^\infty}}+2^{-\frac12N}\|\Omega_0\|_{B_{\infty,1}^0}e^{CV_p(t)}.
\end{eqnarray*}
Now we choose $N\approx V_p(t)$ then
\begin{equation}\label{div4}
\|\zeta(t)\|_{B_{\infty,1}^0}\le C\|\Omega_0\|_{B_{\infty,1}^0}e^{\|u^r/r\|_{L^1_tL^\infty}}\Big(1+\|\diver{ u}\|_{ L^1_tB_{p,1}^{\frac3p}}+\int_0^t\|u(\tau)\|_{\textnormal{Lip}}d\tau\Big).
\end{equation}
Our next object is  to obtain  an estimate for $\|\Omega(t)\|_{B_{\infty,1}^0}$ form $\zeta$. For this purpose  we develop similar calculous to \eqref{tao2} and \eqref{tao1}. We start with the obvious identity 
$$
\Omega(t,x)=\zeta(t,x)+(e^{-F(t,x)}-1)\zeta(t,x),
$$
Using the law products: for $p\in[1,\infty[$
$$
\|uv\|_{B_{\infty,1}^0}\le C\|u\|_{B_{\infty,1}^0}\|v\|_{B_{p,1}^{\frac3p}}, 
$$
we obtain
$$
\|\Omega(t)\|_{B_{\infty,1}^0}\lesssim\|\zeta(t)\|_{B_{\infty,1}^0}\Big( 1+\|e^{-F(t)}-1\|_{B_{p,1}^{\frac3p}} \Big).
$$
Similarly to \eqref{kao45} we get  for $p\in]3,+\infty[$
\begin{eqnarray*}
\|e^{-F(t)}-1\|_{B_{p,1}^{\frac3p}}&\le&C  e^{C\|\nabla u\|_{L^1_tL^\infty}}\int_0^t\|\diver{ u(\tau)}\|_{ B_{p,1}^{\frac3p}}d\tau.
\end{eqnarray*}
Combining these estimates with \eqref{div4}  we find
\begin{eqnarray*}
\|\Omega(t)\|_{B_{\infty,1}^0}&\lesssim&\|\zeta(t)\|_{B_{\infty,1}^0}\Big(1+ e^{C\|\nabla u\|_{L^1_tL^\infty}}\int_0^t\|\diver{ u(\tau)}\|_{ B_{p,1}^{\frac3p}}d\tau\Big)\\
&\le& C\|\Omega_0\|_{B_{\infty,1}^0}e^{\|u^r/r\|_{L^1_tL^\infty}}\Big(1+e^{C\|\nabla u\|_{L^1_tL^\infty}}\|\diver{ u}\|_{ L^1_tB_{p,1}^{\frac3p}}^2\Big)\Big(1+\int_0^t\|u(\tau)\|_{\textnormal{Lip}}d\tau\Big).
\end{eqnarray*}
This achieves the proof of Theorem \ref{Thm89}.

\end{proof}
\subsection{Lower bound of $T_\varepsilon$}
The object of this section is to give   the proof of Theorem \ref{thm111} but for the sake of the conciseness    we will restrict the attention only to   the lower bound of $T_\varepsilon$. At the same time we will derive the involved   Strichartz estimates which form the cornerstone for the proof of  the low Mach number limit and to avoid the redundancy we will omit the proof of this latter point which  can be done analogously to the subcritical case.
\begin{proof} 
Using Lemma \ref{lem22}  we get
$$
\|\vepsilon(t)\|_{\textnormal{Lip}}\lesssim\|\vepsilon(t)\|_{L^2}+\|\divergence\vepsilon(t)\|_{B_{\infty,1}^0}+\|\Omegaepsilon(t)\|_{B_{\infty,1}^0}.
$$
Integrating in time, using Proposition \ref{energy1} and Proposition \ref{stt1} we find
\begin{eqnarray*}
\|\vepsilon\|_{L^1_T\textnormal{Lip}}&\lesssim& \|\vepsilon\|_{L^1_TL^2}+\|\divergence\vepsilon\|_{L^1_TB_{\infty,1}^0}+\|\Omegaepsilon\|_{L^1_TB_{\infty,1}^0}\\
&\le&C_0(1+T)\,e^{\|\divergence\vepsilon\|_{L^1_tB_{\infty,1}^0}}+\|\Omegaepsilon\|_{L^1_TB_{\infty,1}^0}.
%&\le&C_0T\,e^{\|\divergence\vepsilon\|_{L^1_tL^\infty}}+C_0(1+T)e^{CV_\varepsilon(T)}F\big(C_0(1+T)\varepsilon^{\frac1r}e^{CV_\varepsilon(T)}\big)+\|\Omegaepsilon\|_{L^1_TB_{\infty,1}^0}\\
%&\le& C_0(1+T)e^{G_\varepsilon(T)}+\|\Omegaepsilon\|_{L^1_TB_{\infty,1}^0},
\end{eqnarray*}
Let
$$
V_\varepsilon(T):=\|\vepsilon\|_{L^1_T\textnormal{Lip}}+\|\nabla\cepsilon\|_{L^1_TL^\infty}$$
%and
%$$ G_\varepsilon(T):=C_0(1+T)e^{CV_\varepsilon(T)}F\big(C_0(1+T)\varepsilon^{\frac1r}e^{CV_\varepsilon(T)}\big)
%$$
then using again Proposition \ref{stt1} we obtain
\begin{eqnarray}\label{hmid1}
 V_\varepsilon(T)%&=&% \|\vepsilon\|_{L^1_T\textnormal{Lip}}+\|\nabla\cepsilon\|_{L^1_TL^\infty}\\
&\le&C_0(1+T)e^{\|(\divergence\vepsilon,\nabla\cepsilon)\|_{L^1_tB_{\infty,1}^0}} +\|\Omegaepsilon\|_{L^1_TB_{\infty,1}^0}.
\end{eqnarray}
Now, according to  Theorem \ref{Thm89} we have
\begin{equation}\label{tita13}
\|\Omegaepsilon(t)\|_{B_{\infty,1}^0}\le C_0e^{\|\vepsilon^r/r\|_{L^1_tL^\infty}}\Big(1+e^{CV_\varepsilon(T)}\|\diver{ \vepsilon}\|_{ L^1_tB_{p,1}^{\frac3p}}^2\big)\Big(1+\int_0^t\|\vepsilon(\tau)\|_{\textnormal{Lip}}d\tau\Big).
\end{equation}
By virtue of Proposition \ref{vort12} and Proposition \ref{energy1},   we get
\begin{eqnarray}\label{tita12}
\nonumber\|\vepsilon^r/r\|_{L^1_tL^\infty}&\le&C\|\vepsilon\|_{L^1_TL^2}+\|\divergence\vepsilon\|_{L^1_TB_{\infty,1}^0}+C_0 Te^{\|\divergence\vepsilon\|_{L^1_TB_{\infty,1}^0}}\\
&\le&C_0(1+T)\,e^{C\|\divergence\vepsilon\|_{L^1_TB_{\infty,1}^0}}.
%nonumber&\le& C_0(1+T)e^{G_\varepsilon(T)}.
\end{eqnarray}
In the inequality \eqref{tita13} we need to estimate $\|\diver{ \vepsilon}\|_{ L^1_tB_{p,1}^{\frac3p}}$. For this purpose we will interpolate this norm between the energy estimate and the dispersive one. More precisely, for $p\in]2,+\infty[$ we get by interpolation
$$
\|\diver{ \vepsilon}\|_{ B_{p,1}^{\frac3p}}\le C\|\diver{ \vepsilon}\|_{ B_{2,1}^{\frac32}}^{\frac2p}\|\diver{ \vepsilon}\|_{ B_{\infty,1}^{0}}^{1-\frac2p}.
$$
Integrating in time  and combining H\"{o}lder inequality with Proposition \ref{energy1} we find
\begin{eqnarray}\label{tita114}
\nonumber\|\diver{ \vepsilon}\|_{ L^1_TB_{p,1}^{\frac3p}}&\le& CT^{\frac2p}\|\diver{ \vepsilon}\|_{ L^\infty_TB_{2,1}^{\frac32}}^{\frac2p}\|\diver{ \vepsilon}\|_{ L^1_TB_{\infty,1}^{0}}^{1-\frac2p}\\
&\le& C_0T^{\frac2p}\,e^{CV_\varepsilon(T)} \|\diver{ \vepsilon}\|_{ L^1_TB_{\infty,1}^{0}}^{1-\frac2p}. \end{eqnarray}
Putting together \eqref{tita13}, \eqref{tita12} and \eqref{tita114} and taking $p=4$ yield to
\begin{eqnarray*}
\|\Omegaepsilon(t)\|_{B_{\infty,1}^0}&\le& C_0 \exp\big\{C_0(1+T)\,e^{C\|\divergence\vepsilon\|_{L^1_TB_{\infty,1}^0}}\big\}\Big(1+  e^{CV_\varepsilon(T)}\|\divergence\vepsilon\|_{L^1_TB_{\infty,1}^0} \Big)\big(1+V_\varepsilon(T)\big)\\
&\le& C_0 \exp\big\{C_0(1+T)\,e^{Ce^{CV_\varepsilon(T)}\,\|\divergence\vepsilon\|_{L^1_TB_{\infty,1}^0}}\big\}\big(1+V_\varepsilon(T)\big).
\end{eqnarray*}
Combining this estimate with \eqref{hmid1} and applying Gronwall's inequality we obtain
\begin{equation}\label{cris}
V_\varepsilon(T)\le C_0 e^{\exp\big\{C_0(1+T)\,e^{Ce^{CV_\varepsilon(T)}\|(\divergence\vepsilon,\nabla\cepsilon)\|_{L^1_TB_{\infty,1}^0}}\big\}}.
\end{equation}
Let $N\in\NN^*$ be an integer that will be fixed later.  Using the dyadic decomposition combined with Bernstein inequality and the property $(1)$ in the Definition \ref{heter}  we get
\begin{eqnarray*}
\|\divergence\vepsilon\|_{B_{\infty,1}^0}&=&\|\divergence\mathcal{Q}\vepsilon\|_{B_{\infty,1}^0}\\
&=&\sum_{q<N}\|\Delta_q\divergence\mathcal{Q}\vepsilon\|_{L^\infty}+\sum_{q\geq N}\|\Delta_q\divergence\mathcal{Q}\vepsilon\|_{L^\infty}\\
&\lesssim&\sum_{q<N}2^q\|\Delta_q\mathcal{Q}\vepsilon\|_{L^\infty}+\frac{1}{\Psi(N)}\sum_{q\geq N}2^{\frac52 q}\Psi(q)\|\Delta_q\mathcal{Q}\vepsilon\|_{L^2}\\
&\lesssim& 2^N\|\mathcal{Q}\vepsilon\|_{L^\infty}+\frac{1}{\Psi(N)}\|\vepsilon\|_{B_{2,1}^{\frac52,\Psi}}.
\end{eqnarray*}
Integrating in time yields to
$$
\|\divergence\vepsilon\|_{L^1_TB_{\infty,1}^0}\lesssim 2^N\|\mathcal{Q}\vepsilon\|_{L^1_TL^\infty}+\frac{1}{\Psi(N)}\|\vepsilon\|_{L^1_TB_{2,1}^{\frac52,\Psi}}.
$$
Using Proposition \ref{energy1}, Corollary \ref{cor11} and  H\"{o}lder inequalities  we get for $r>2,$ 
\begin{eqnarray*}
\|\divergence\vepsilon\|_{L^1_TB_{\infty,1}^0}&\le& C_0(1+T^{2-{1\over r}})e^{CV_\varepsilon(T)}\Big(2^N\varepsilon^{\frac1r}+\frac{1}{\Psi(N)}\Big)\\
&\le& C_0(1+T^{2})e^{CV_\varepsilon(T)}\Big(2^N\varepsilon^{\frac1r}+\frac{1}{\Psi(N)}\Big).
\end{eqnarray*}
Similarly we get
$$
\|\nabla\cepsilon\|_{L^1_TB_{\infty,1}^0}\le C_0(1+T^2)e^{CV_\varepsilon(T)}\Big(2^N\varepsilon^{\frac1r}+\frac{1}{\Psi(N)}\Big).
$$
Thus 
$$
\|(\divergence\vepsilon,\nabla\cepsilon)\|_{L^1_TB_{\infty,1}^0}\le C_0(1+T^2)e^{CV_\varepsilon(T)}\Big(2^N\varepsilon^{\frac1r}+\frac{1}{\Psi(N)}\Big).
$$
We take $N$ such that
$$
e^{N}\approx \varepsilon^{-\frac{1}{2r}}
$$
which leads for small $\varepsilon$  to 
\begin{equation*}
\|(\divergence\vepsilon,\nabla\cepsilon)\|_{L^1_TB_{\infty,1}^0}\le C_0(1+T^2)e^{CV_\varepsilon(T)}\Big(\varepsilon^{\frac{1}{2r}}+\frac{1}{\Psi\big(\log\big(\frac{1}{\varepsilon^{\frac{1}{2r}}}\big)}\Big).
\end{equation*}
To make simple notations we introduce the function 
$$
{\Phi}(x)=x^{\frac{1}{2r}}+\frac{1}{\Psi\big(\log\big(\frac{1}{x^{\frac{1}{2r}}}\big)},\quad  x\in]0,1[.
$$
We observe that since  $\Psi \in \mathcal{U}_\infty$ and it  satisfies the property $(ii)$ of the Definition \ref{heter} then 
\begin{equation}\label{trasq}
\lim_{\varepsilon\to 0}\Phi(\varepsilon)=0.
\end{equation}
Thus we can rewrite the preceding estimate under the form
\begin{equation}\label{zz1}
\|(\divergence\vepsilon,\nabla\cepsilon)\|_{L^1_TB_{\infty,1}^0}\le C_0(1+T^2)e^{CV_\varepsilon(T)}\Phi(\varepsilon).
\end{equation}
 This gives 
\begin{equation*}
e^{CV_\varepsilon(T)}\|(\divergence\vepsilon,\nabla\cepsilon)\|_{L^1_TB_{\infty,1}^0}\le C_0(1+T^2)e^{CV_\varepsilon(T)}\Phi(\varepsilon).\end{equation*}
Plugging  this estimate into \eqref{cris} yields to
\begin{equation}\label{zizq}
CV_\varepsilon(T)\le  e^{e^{C_0(1+T)\exp\{C_0(1+T^2)  \Phi(\varepsilon)\,e^{CV_\varepsilon(T)}
\}}}.
\end{equation}

We choose $T_\varepsilon$ such that
\begin{equation}\label{c1z}
{e^{e^{\exp\{2C_0(1+T_\varepsilon)\}}}}=\Phi^{-\frac12}(\varepsilon).
\end{equation}
Then we claim that for small $\varepsilon$ and for every $t\in[0,T_\varepsilon]$ 
\begin{equation}\label{toutz1}
e^{CV_\varepsilon(t)}\le \Phi^{-\frac12}(\varepsilon).
\end{equation}
Indeed, we set 
$$
I_{T_\varepsilon}:=\Big\{t\in[0,T_\varepsilon]; e^{CV_\varepsilon(t)}\le  \Phi^{-\frac12}(\varepsilon\Big\}.
$$ 
First this set is nonempty since $0\in I_{T_\varepsilon}$. By the continuity of $t\mapsto V_\varepsilon(t)$, the set  $I_{T_\varepsilon}$ is closed and thus to prove that $I_{T_\varepsilon}$ coincides with $[0,T_\varepsilon]$ it suffices to show that $I_{T_\varepsilon}$ is an open set. Let $t\in I_{T_\varepsilon}$ then using \eqref{zizq}  we get for small $\varepsilon$
\begin{eqnarray}\label{hmd11}
e^{CV_\varepsilon(t)}&\le&e^{e^{\exp\{C_0(1+T)\exp\{C_0(1+T^2)  \Phi^{\frac12}(\varepsilon)\}\}}}.
%&\le&  e^{e^{e^{\exp\{C_0(1+T)\}}}}\\
%&\le&.
\end{eqnarray}
From \eqref{c1z} we get for small values of $\varepsilon$, 
 \begin{eqnarray}\label{lowb}
T_\varepsilon&\approx& c \log\log\log \Phi^{-\frac12}(\varepsilon)
%&\approx& c \log\log\log \Phi^{-1}(\varepsilon)
\end{eqnarray}
for some constant $c$ and thus  
$$
\lim_{\varepsilon\to 0}(1+T_\varepsilon^2)  \Phi^{\frac12}(\varepsilon)=\lim_{\varepsilon\to 0} \Phi^{\frac12}(\varepsilon)\log^2\log\log(\Phi^{-\frac12}(\varepsilon))=0.
$$
Therefore for small $\varepsilon$ and for every $t\in[0,T_\varepsilon]$ we have $e^{C_0(1+t^2)  \Phi^{\frac12}(\varepsilon)}<2$ and consequently from \eqref{hmd11} and \eqref{c1z}  we find
\begin{eqnarray*}
e^{CV_\varepsilon(t)}&<&e^{e^{\exp\{2C_0(1+T)\}}}\\
&<& \Phi^{-\frac12}(\varepsilon).
%&\le&  e^{e^{e^{\exp\{C_0(1+T)\}}}}\\
%&\le&.
\end{eqnarray*}

This proves that $t$ is in the interior of $I_{T_\varepsilon}$ and thus $I_{T_\varepsilon}$ is an open set of $[0,T_\varepsilon]$. Consequently we conclude that $I_{T_\varepsilon}=[0,T_\varepsilon].$

Now inserting \eqref{toutz1} into  \eqref{zizq} we  get  for $T\in[0,T_\varepsilon]$
\begin{equation}\label{Lips}
V_\varepsilon(T)\le C_0e^{{\exp{C_0 T}}}. 
\end{equation}
Plugging \eqref{toutz1} into the estimate  \eqref{zz1} we obtain for $T\in[0,T_\varepsilon]$ and for small $\varepsilon$
%$$
%\|\divergence\vepsilon\|_{L^1_TB_{\infty,1}^0}+\|\nabla\cepsilon\|_{L^1_TB_{\infty,1}^0}\le C e^{e^{C_0 (1+T)}}\,\varepsilon^{\frac{2s-5}{8s-12}}.
%$$
\begin{eqnarray}\label{disp1}
\nonumber\|\divergence\vepsilon\|_{L^1_TB_{\infty,1}^0}+\|\nabla\cepsilon\|_{L^1_TB_{\infty,1}^0}&\le& C_0 (1+T^2)\Phi^{\frac12}(\varepsilon)\\
\nonumber&\le&C_0\Phi^{\frac13}(\varepsilon).
\end{eqnarray}
From Corollary \ref{cor11} and the above estimates  we obtain for every $r\in]2,+\infty[$ and for small values of $\varepsilon$
\begin{eqnarray*}
\nonumber\|\mathcal{Q}\vepsilon\|_{L^r_TL^\infty}+\|\cepsilon\|_{L^r_TL^\infty}&\le&C_0 \varepsilon^{\frac1r}(1+T)e^{C V_\varepsilon(t)} \\
\nonumber&\le& C_0 \varepsilon^{\frac1r}\Phi^{-1}(\varepsilon)\\
&\le& C_0\varepsilon^{\frac{1}{2r}}.
%\nonumber&\leq&C_0  \varepsilon^{\frac1r-C\theta}\log\log\log(\varepsilon^{-\theta})\\
\end{eqnarray*}
We point out that the use of H\"{o}lder inequality and the slow growth of $T$ allow us to get the following: for every $r\in[1,+\infty[$
\begin{eqnarray}\label{esp1}
\|\mathcal{Q}\vepsilon\|_{L^r_TL^\infty}+\|\cepsilon\|_{L^r_TL^\infty}
&\le&C_0\varepsilon^{\sigma'},\quad\sigma^\prime>0.
\end{eqnarray}
Inserting \eqref{disp1} into \eqref{tour1} leads to
$$\|\Omegaepsilon(T)\|_{L^\infty}\le C_0 e^{C_0T}.
$$
To  estimate the solutions of \eqref{eqs:3} in Sobolev norms we use Proposition \ref{energy1} combined with the Lipschitz bound \eqref{Lips}
$$
\|(\vepsilon,\cepsilon)(T)\|_{B_{2,1}^{\frac52,\Psi}}\le C_0 e^{e^{\exp\{C_0T\}}}.
$$

\end{proof}

\section{Appendix}
This last section is devoted to the proof of some lemmata used in the preceding sections.
\begin{lem}\label{lem2}
Let $s>\frac52$ and $\varphi\in L^\infty([0,T]; H^s(\RR^3))$, then we have
$$
\|\nabla\varphi\|_{L^1_TB_{\infty,1}^0}\le CT^{1-\frac{2s-5}{r(2s-3)}}\|\varphi\|_{L^r_TL^\infty}^{\frac{2s-5}{2s-3}}\|\varphi\|_{L^\infty_TH^s}^{\frac{2}{2s-3}}.
$$

\end{lem}
\begin{proof}
To prove this inequality we will use a frequency  interpolation argument. We split the function $\varphi$ into its dyadic blocks: $\varphi=\sum_{q\geq-1}\Delta_q\varphi$ and we take  an integer  $N$ that will be be fixed later. Then by Bernstein inequality we obtain
\begin{eqnarray*}
\|\nabla\varphi\|_{B_{\infty,1}^0}&=& \sum_{q=-1}^{N}\|\Delta_q\nabla\varphi\|_{L^\infty}+\sum_{q>N}\|\Delta_q\nabla\varphi\|_{L^\infty}\\
&\le& C 2^N\|\varphi\|_{L^\infty}+C\sum_{q>N}2^{\frac52 q}\|\Delta_q\varphi\|_{L^2}\\
&\le& C 2^N\|\varphi\|_{L^\infty}+C 2^{-N(s-\frac52)}\|\varphi\|_{H^s}.
\end{eqnarray*}
Now we choose $N$ such that
$$
2^{N(s-\frac32)}\approx\frac{\|\varphi\|_{H^s}}{\|\varphi\|_{L^\infty}},
$$
and this gives 
$$
\|\nabla\varphi\|_{L^\infty}\le C\|\varphi\|_{H^s}^{\frac{2}{2s-3}}\|\varphi\|_{L^\infty}^{\frac{2s-5}{2s-3}}.
$$
Now to get the desired estimate, it suffices to use H\"{o}lder inequality in time variable.
\end{proof}
The next lemma deals with a logarithmic estimate in the compressible context.
\begin{lem}\label{lem22}
The following assertions hold true
\begin{enumerate}
\item Sub-critical case:
Let $v$ be a vector field belonging to $H^s(\RR^3)$ with $s>\frac52$ and denote by $\Omega=\nabla\wedge v$ its vorticity,  then we have 
$$
\|\nabla v\|_{L^\infty}\lesssim \|v\|_{L^2}+\|\diver v\|_{B_{\infty,1}^0}+\|\Omega\|_{L^\infty}\log(e+\|v\|_{H^s}).$$
\item Critical case: Let $v$ be a  vector field belonging to $B_{\infty,1}^{1}\cap L^2$, then
$$
\|\nabla v\|_{L^\infty}\lesssim \|v\|_{L^2}++\|\diver v\|_{B_{\infty,1}^0}+\|\Omega\|_{B_{\infty,1}^0}.
$$
\end{enumerate}
\end{lem}
\begin{proof}
We split $v$ into compressible and incompressible parts: $v=\mathcal{Q}v+\mathcal{P}v.$ Then $$\hbox{curl} \,v=\hbox{curl}\,\mathcal{P}v.
$$
Thus the incompressible part $\mathcal{P}v$ has the same vorticity as the total velocity and thus we can use Brezis-Gallou\"{e}t logarithmic estimate
\begin{eqnarray*}
\|\nabla\mathcal{P}v\|_{L^\infty}&\lesssim&\|\mathcal{P}v\|_{L^2}+\|\Omega\|_{L^\infty}\log(e+\|\mathcal{P}v\|_{H^s})\\
&\lesssim&\|v\|_{L^2}+\|\Omega\|_{L^\infty}\log(e+\|v\|_{H^s}).
\end{eqnarray*}
It remains to estimate the  Lipschitz norm of the compressible part $\mathcal{Q}v$. Using Bernstein inequality for low frequency
\begin{eqnarray*}
\|\nabla\mathcal{Q}v\|_{L^\infty}&=&\|\nabla^2\Delta^{-1}\divergence v\|_{L^\infty}\\
&\le&\|\Delta_{-1}\nabla^2\Delta^{-1}\divergence v\|_{L^\infty}+\sum_{q\in\NN}\|\Delta_{q}\nabla^2\Delta^{-1}\divergence v\|_{L^\infty}\\
&\lesssim&\|v\|_{L^2}+\sum_{q\in\NN}\|\Delta_{q}\divergence v\|_{L^\infty}\\
&\lesssim&\|v\|_{L^2}+\|\divergence v\|_{B_{\infty,1}^0}.
\end{eqnarray*}
This concludes the proof of the logarithmic estimate.
\end{proof}
We will establish the  following
\begin{lem}\label{proz1} Let $s>0$ and $\Psi\in \mathcal{U}$, see the Definition \ref{heter}.   Then we  have the following commutator estimate: 
$$\sum_{q\geq-1}2^{qs}\Psi(q)\Vert[\Delta_{q},v\cdot\nabla]u\Vert_{L^2}\lesssim\Vert\nabla v\Vert_{L^\infty}\Vert u\Vert_{B^{s,\Psi}_{2,1}}+\Vert\nabla u\Vert_{L^\infty}\Vert v\Vert_{B^{s,\Psi}_{2,1}}.$$
\end{lem}
\begin{proof}
By using Bony's decomposition, we can split the commutator into three parts,
\begin{eqnarray*}
[\Delta_{q},v\cdot\nabla]u&=&[\Delta_{q},T_{v}\cdot\nabla]u+[\Delta_{q},T_{\nabla\cdot}\cdot v]u+[\Delta_{q},\mathcal{R}(v,\nabla)]u\\
&=&\textnormal{I}_q+\textnormal{II}_q+\textnormal{III}_q.
\end{eqnarray*}
We start with the estimate of the first term $\textnormal{I}.$ Then by definition
$$\Vert[\Delta_{q},T_{v}\cdot\nabla]u\Vert_{L^2}\le\sum_{\vert j-q\vert\le 4}\Vert[\Delta_{q},S_{j-1}v\cdot\nabla]\Delta_{j}u\Vert_{L^2}.$$
We can now use the following commutator inequality, for a proof see for example \cite{MR2000a:76030},
$$
\|[\Delta_q,a]b\|_{L^p}\le C2^{-q}\|\nabla a\|_{L^\infty}\|a\|_{L^p}
$$
which yields in view of Bernstein inequality  to 
\begin{eqnarray*}
\Vert[\Delta_{q},T_{v}\cdot\nabla]u\Vert_{L^2}&\le&C\sum_{\vert j-q\vert\le 4}2^{-q}\Vert\nabla S_{j-1}v\Vert_{L^\infty}\Vert\nabla\Delta_{j}u\Vert_{L^2}\\
&\lesssim& C\Vert\nabla v\Vert_{L^\infty}\sum_{\vert j-q\vert\le 4}2^{j-q}\Vert\Delta_{j}u\Vert_{L^2}\\
&\lesssim& \Vert\nabla v\Vert_{L^\infty}\sum_{\vert j-q\vert\le 4}\Vert\Delta_{j}u\Vert_{L^2}.
\end{eqnarray*}
Multiplying this last inequality by $2^{qs}\Psi(q)$, summing up over $q$, and using the property $(2)$ in the Definition \ref{heter} we get 
\begin{eqnarray*}
\sum_{q\geq -1}2^{qs}\Psi(q)\|\textnormal{I}_q\|_{L^2}&\lesssim&\Vert\nabla v\Vert_{L^\infty}\sum_{\vert j-q\vert\le 4}\frac{\Psi(q)}{\Psi(j)} \big(2^{js}\Psi(j)\Vert\Delta_{j}u\Vert_{L^2}\big)\\
&\lesssim&
\Vert\nabla v\Vert_{L^\infty}\Vert u\Vert_{B^{s,\Psi}_{2,1}}.
\end{eqnarray*}
Similarly, the second term $\textnormal{II}_q$ is estimated by
\begin{equation*}
\sum_{q\geq-1}2^{qs}\Psi(q)\|\textnormal{II}_q\|_{L^2}\lesssim\Vert\nabla u\Vert_{L^\infty}\Vert v\Vert_{B^{s,\Psi}_{2,1}}.
\end{equation*}
What is left  is to estimate the remainder term. From the definition and since the Fourier transform of $[\Delta_{q},\Delta_{j}v]\nabla$ is supported in a ball of radius $2^q$ then we obtain
$$\Vert[\Delta_{q},\mathcal{R}(v,\nabla)]u\Vert_{L^2}\le\sum_{j\ge q-4}\Vert[\Delta_{q},\Delta_{j}v]\nabla\widetilde{\Delta}_{j} u\Vert_{L^2}.$$
To estimate the term inside the sum we do not need to use the structure of the commutator. Applying  H\"older and Bernstein inequalities yields to
\begin{eqnarray*}
\Vert[\Delta_{q},\mathcal{R}(v,\nabla)]u\Vert_{L^2}&\lesssim& \sum_{j\ge q-4} \Vert\Delta_{j}v\Vert_{L^2}\Vert\nabla\widetilde{\Delta}_{j} u\Vert_{L^\infty}\\
&\lesssim& \Vert\nabla u\Vert_{L^\infty}\sum_{j\ge q-4}\Vert\Delta_{j}v\Vert_{L^2}.
\end{eqnarray*}
Multiplying this last inequality by $2^{qs}\Psi(q)$, summing up over $q$, using Fubini identity and the property $(1)$ of the Definition \ref{heter} we get
\begin{eqnarray*}
\sum_{q\geq-1}2^{qs}\Psi(q)\Vert[\Delta_{q},\mathcal{R}(v,\nabla)]u\Vert_{L^2}&\lesssim& \Vert\nabla u\Vert_{L^\infty}\sum_{q\geq-1}\sum_{j\ge q-4}2^{qs}\Psi(q)\|\Delta_j u\|_{L^2}\\
&\lesssim& \Vert\nabla u\Vert_{L^\infty}\sum_{j\geq-1}\|\Delta_j u\|_{L^2}\sum_{q\leq j+4}2^{qs}\Psi(q)\\
&\lesssim&\Vert\nabla u\Vert_{L^\infty}\|u\|_{B_{2,1}^{s,\Psi}}.
\end{eqnarray*}
This concludes the proof of Lemma \ref{proz1}.
\end{proof}

The following  proposition describes the propagation of Besov regularity for transport equation.   
 \begin{prop}\label{Lems2}
Let  $u$ be a smooth  vector field, not necessary of zero divergence. Let $f$ be a  smooth solution of the transport equation 
 $$
\partial_{t}f+u\cdot\nabla f=g,\, f_{|t=0}=f_0,
$$
such that $f_0\in B_{p,r}^s(\mathbb R^3)$ and 
$g\in{L^1_{\textnormal{loc}}}(\mathbb R_{+};B_{p,r}^{s}).  $     
{Then the following assertions hold true}
\begin{enumerate}
\item Let $p,r\in[1,\infty]$and $s\in]0,1[$, then
\begin{equation*}\label{df}
\|f(t)\|_{B_{p,r}^s}     
\leq 
Ce^{CV(t)}
\Big(\|f_0\|_{B_{p,r}^s}+\int_{0}^te^{-CV(\tau)}\|g(\tau)\|_{B_{p,r}^{s}}d\tau\Big),
\end{equation*}
where $ V(t)=\int_{0}^t\|\nabla u(\tau)\|_{L^\infty}d\tau$ and $C$ is a constant depending  on $s.$
 \item Let $s\in]-1,0], r\in[1,+\infty]$ and $ p\in[2,+\infty]$ with $s+\frac3p>0$, then 
 \begin{equation*}\label{df}
\|f(t)\|_{B_{p,r}^s}     
\leq 
Ce^{CV_p(t)}
\Big(\|f_0\|_{B_{p,r}^s}+\int_{0}^te^{-CV_p(\tau)}\|g(\tau)\|_{B_{p,r}^{s}}d\tau\Big),
\end{equation*}
with  $ V_p(t)=\|\nabla u\|_{L^1_tL^\infty}+\|\diver u\|_{L^1_tB_{p,\infty}^{\frac3p}}.$ 
\end{enumerate}
\end{prop}
\begin{proof}
${\bf{(1)}}$ This estimate is classical, see for example \cite{MR2000a:76030} in the H\"{o}lderian case and the similar proof works as well for Besov spaces.

${\bf{(2)}}$ The proof for a general case can be found for example in \cite{Miao1} and for the convenience of the reader we will give here the complete proof in our special case.
We start with localizing in frequency the equation leading to,
$$
\begin{aligned}
\partial_{t}\Delta_q f+(u\cdot\nabla)\Delta_qf
&=\Delta_qg+(u\cdot\nabla)\Delta_qf-\Delta_q\big(u\cdot\nabla f\big)
\\&
=\Delta_qg-[\Delta_q,u\cdot\nabla]f.  
\end{aligned}
 $$
Taking the $L^p$ norm, then the zero divergence of the flow gives  
\begin{equation}\label{ter1}
\|\Delta_qf(t)\|_{L^p}
\leq
\|\Delta_qf_0\|_{L^p}
+\int_0^t\|\Delta_qg\|_{L^p}d\tau
+\int_0^t\Big\|[\Delta_q,u\cdot\nabla]f\Big\|_{L^p}d\tau.  
\end{equation}
From Bony's decomposition, the commutator may be decomposed as follows 
$$
[\Delta_q,u\cdot\nabla]f=[\Delta_{q},T_{u}\cdot\nabla]f+[\Delta_{q},T_{\nabla\cdot}\cdot u]f+[\Delta_{q},\mathcal{R}(u,\nabla)]f
$$
For the paraproducts we do not need any structure for the velocity, so reproducing the same  the proof of Lemma \ref{proz1} we find for $s<1$
$$
\|[\Delta_{q},T_{u}\cdot\nabla]f+[\Delta_{q},T_{\nabla\cdot}\cdot u]f\|_{L^p}\le C\|\nabla u\|_{L^\infty} \|f\|_{B_{p,r}^s} c_q2^{-qs},$$
with $(c_q)_{\ell^r}=1.$ Concerning the remainder term we write
\begin{eqnarray*}
[\Delta_{q},\mathcal{R}(u,\nabla)]f
&=&\sum_{j}[\Delta_{q},\Delta_j u^i\partial_i]\tilde\Delta_jf\\
&=&\partial_i\sum_{j\geq q-4}[\Delta_{q},\Delta_j u^i]\tilde\Delta_jf-\sum_{j\geq q-4}[\Delta_{q},\Delta_j \diver u]\tilde\Delta_jf\\
&=&\hbox{I}_q+\hbox{II}_q
\end{eqnarray*}
Similarly to the proof of Lemma \ref{proz1} we get for $s>-1$
$$
\|\hbox{I}_q\|_{L^p}\le C\|\nabla u\|_{L^\infty} \|f\|_{B_{p,r}^s} c_q2^{-qs},\quad (c_q)_{\ell^r}=1
$$
For the second term,  since the Fourier transform of  $[\Delta_{q},\Delta_j \diver u]\tilde\Delta_jf$ is supported in a bull of size $2^q$ then using Bernstein and H\"{o}lder inequalities we obtain for $p\geq2$
\begin{eqnarray*}
\|[\Delta_{q},\Delta_j \diver u]\tilde\Delta_jf\|_{L^p}&\le& C2^{\frac3p q} \|[\Delta_{q},\Delta_j \diver u]\tilde\Delta_jf\|_{L^{\frac{p}{2}}}\\
&\le&C2^{\frac3p q} \|\Delta_j \diver u\|_{L^p}\|\tilde\Delta_jf\|_{L^p}
\end{eqnarray*}
Therefore
\begin{eqnarray*}
2^{qs}\|\hbox{II}_q\|_{L^p}&\leq& C2^{\frac3p q}\sum_{j\geq q-4}\|\tilde\Delta_jf\|_{L^p} \|\Delta_j \diver u\|_{L^p}\\
&\le&  C\|\diver v\|_{B_{p,\infty}^{\frac3p}}\sum_{j\geq q-4}2^{(q-j)(s+\frac3p)}\big(2^{js}\|\tilde\Delta_jf\|_{L^p}\big)
\end{eqnarray*}
Now since $s+\frac3p>0$ then we can use Young inequality to deduce that
$$
\Big(2^{qs}\|\hbox{II}_q\|_{L^p}\Big)_{\ell^r}\le C\|\diver v\|_{B_{p,\infty}^{\frac3p}}\|f\|_{B_{p,r}^s}
$$
Combining these estimates yields for $ -\min(1,\frac3p)<s\leq0$ to
$$
\Big(2^{qs}\|[\Delta_q, u\cdot\nabla]u\|_{L^p}\Big)_{\ell^r}\le C\|\diver v\|_{B_{p,\infty}^{\frac3p}}\|f\|_{B_{p,r}^s}
$$
Plugging this estimate into \eqref{ter1} we get
$$
\|f(t)\|_{B_{p,r}^s}\lesssim \|f(0)\|_{B_{p,r}^s} +\int_0^t V_p(t)\|f(\tau)\|_{B_{p,r}^s}d\tau+\int_0^t\|g(\tau)\|_{B_{p,r}^s}d\tau, 
$$
with 
$$
V_p(t):=\|\nabla u\|_{L^\infty}+\|\diver u\|_{B_{p,\infty}^{\frac3p}}.
$$
To get the desired estimate it suffices to use Gronwall's lemma.

\end{proof}

%\begin{lem}\label{lem2}
%Let $s>\frac52$ and $\varphi\in L^\infty([0,T]; H^s(\RR^3))$, then we have
%$$
%\|\nabla\varphi\|_{L^1_TB_{\infty,1}^0}\le CT^{1-\frac{2s-5}{r(2s-3)}}\|\varphi\|_{L^r_TL^\infty}^{\frac{2s-5}{2s-3}}\|\varphi\|_{L^\infty_TH^s}^{\frac{2}{2s-3}}.
%$$

%\end{lem}
%\begin{proof}
%To prove this inequality we will use a frequency  interpolation argument. We split the function $\varphi$ into its dyadic blocks: $\varphi=\sum_{q\geq-1}\Delta_q\varphi$ and we take  an integer  $N$ that will be be fixed later. Then by Bernstein inequality we obtain
%\begin{eqnarray*}
%\|\nabla\varphi\|_{B_{\infty,1}^0}&=& \sum_{q=-1}^{N}\|\Delta_q\nabla\varphi\|_{L^\infty}+\sum_{q>N}\|\Delta_q\nabla\varphi\|_{L^\infty}\\
%&\le& C 2^N\|\varphi\|_{L^\infty}+C\sum_{q>N}2^{\frac52 q}\|\Delta_q\varphi\|_{L^2}\\
%&\le& C 2^N\|\varphi\|_{L^\infty}+C 2^{-N(s-\frac52)}\|\varphi\|_{H^s}.
%\end{eqnarray*}
%Now we choose $N$ such that
%$$
%2^{N(s-\frac32)}\approx\frac{\|\varphi\|_{H^s}}{\|\varphi\|_{L^\infty}},
%$$
%and this gives 
%$$
%\|\nabla\varphi\|_{L^\infty}\le C\|\varphi\|_{H^s}^{\frac{2}{2s-3}}\|\varphi\|_{L^\infty}^{\frac{2s-5}{2s-3}}.
%$$
%Now to get the desired estimate, it suffices to use H\"{o}lder inequality in time variable.
%\end{proof}

The last point that we will discuss concerns the action of the operator $({\partial_r}/{r})\Delta^{-1}$ over axisymmetric functions. we will show that its restriction over this class of functions behaves like Riesz transforms. This study was done before in \cite{HR} and for the convenient of the reader  we will give the complete proof here.
 \begin{lem}\label{prop1}
We have for  every axisymmetric smooth scalar function $u$
\begin{equation}
\label{prop1-1}
({\partial_r}/{r})\Delta^{-1}u(x)=\frac{x_2^2}{r^2}\mathcal{R}_{11}u(x)+\frac{x_1^2}{r^2}\mathcal{R}_{22}u(x)-2\frac{x_1x_2}{r^2}\mathcal{R}_{12}u(x),
\end{equation}
with 
$
\mathcal{R}_{ij}=\partial_{ij}\Delta^{-1}.$  
\end{lem}
\begin{proof}
We set $f=\Delta^{-1}u,$ then we can show from Biot-Savart law that $f$ is also axisymmetric. Hence we get
 by using  polar coordinates that 
\begin{equation}
\label{laplace1}
\partial_{11}f+\partial_{22}f  =  (\partial_r/r)f+\partial_{rr}f.
\end{equation}
 where 
$$
\partial_r=\frac{x_1}{r}\partial_1+\frac{x_2}{r}\partial_2.
$$
By using this expression of $\partial_{r}$, we obtain
\begin{eqnarray*}
\partial_{rr}&=&\big(\frac{x_1}{r}\partial_1+\frac{x_2}{r}\partial_2\big)^2
=\partial_r(\frac{x_1}{r})\partial_1+\partial_r(\frac{x_2}{r})\partial_2+\frac{x_1^2}{r^2}\partial_{11}+\frac{x_2^2}{r^2}\partial_{22}
+\frac{2x_1 x_2}{r^2}\partial_{12}.\\
&=&\frac{x_1^2}{r^2}\partial_{11}+\frac{x_2^2}{r^2}\partial_{22}+\frac{2x_1 x_2}{r^2}\partial_{12}
\end{eqnarray*}
 since
 $$
\partial_r(\frac{x_i}{r})=0,\quad\forall i\in\{1,2\}.
$$
This yields by using \eqref{laplace1} that 
\begin{eqnarray*}
{\partial_r \over r} f &=&(1-\frac{x_1^2}{r^2})\partial_{11} f +(1-\frac{x_2^2}{r^2})\partial_{22} f -\frac{2x_1 x_2}{r^2}\partial_{12}f \\
&=&\frac{x_2^2}{r^2}\partial_{11} f +\frac{x_1^2}{r^2}\partial_{22}f -\frac{2x_1 x_2}{r^2}\partial_{12}f.
\end{eqnarray*}
 To  get \eqref{prop1-1},   it suffices to replace $f$ by  $\Delta^{-1} u $.

\end{proof}

The following result describes  the anisotropic dilatation in Besov spaces.
\begin{lem}\label{dilatation}
Let $s\in]-1,+\infty[\backslash\{0\}$, $f:\RR^3\to \RR$ be a function belonging to $B_{\infty,\infty}^{s}, $ and denote by $f_{\lambda}(x_{1},x_{2},x_{3})=f(\lambda x_{1},x_{2},x_{3}).$ Then, there exists an absolute constant $C>0$ such that for all $\lambda\in]0,1[$ 
$$
\|f_{\lambda}\|_{B_{\infty,\infty}^{s}}\leq C\lambda^s\|f\|_{B_{\infty,\infty}^s}.
$$
\end{lem}
\begin{proof}
Let $q\geq-1,$ we denote by $f_{q,\lambda}=(\Delta_{q}f)_{\lambda}.$ From the definition we have
\begin{eqnarray*}
\|f_{\lambda}\|_{B_{\infty,\infty}^s}&=&\|\Delta_{-1}f_{\lambda}\|_{L^\infty}+\sup_{j\in\NN}2^{js}\|\Delta_{j}f_{\lambda}\|_{L^\infty}.
%&\leq&C\|f\|_{L^\infty}+\sup_{j\in\NN}2^{\atop q\geq-1}\|\Delta_{j}f_{q,\lambda}\|_{L^\infty}.
\end{eqnarray*}
For $j,q\in\NN,$ the Fourier transform of $\Delta_{j}f_{q,\lambda}$ is supported in the set
$$
\Big\{|\xi_{1}|+|\xi'|\approx 2^j\quad\hbox{and}\quad \lambda^{-1} |\xi_{1}|+|\xi'|\approx 2^q \Big\},
$$
where $\xi'=(\xi_{2},\xi_{3}).$
A direct consideration  shows that this set is empty \mbox{if $2^q\lesssim 2^j$ or $2^{j-q}\lesssim \lambda$}. Thus we get for an integer $n_{1}$
\begin{eqnarray*}
\|f_{\lambda}\|_{B_{\infty,\infty}^s}&\lesssim&\|\Delta_{-1}f_\lambda\|_{L^\infty}+\sum_{q-n_{1}+\log_2\lambda\leq j\atop
j\leq q+n_{1}}2^{js}\|\Delta_{j}f_{q,\lambda}\|_{L^\infty}\\
&\lesssim&\|\Delta_{-1}f_\lambda\|_{L^\infty}+\sum_{q-n_{1}+\log_2\lambda\leq j\atop
j\leq q+n_{1}}2^{(j-q)s} 2^{qs}\|f_{q}\|_{L^\infty}\\
&\lesssim&\|\Delta_{-1}f_\lambda\|_{L^\infty}+\|f\|_{B_{\infty,\infty}^s}\sum_{-n_{1}+\log_2\lambda}^{n_1}
2^{js} \\
&\lesssim&\|\Delta_{-1}f_\lambda\|_{L^\infty}+\|f\|_{B_{\infty,\infty}^s}\lambda^s.
%&\lesssim&\|f\|_{L^\infty}+(n_{1}-\log\lambda)\sum_{q}\|f_q\|_{L^\infty}\\
%&\lesssim&(1-\log\lambda)\|f\|_{B_{\infty,1}^0}.
\end{eqnarray*}
Let us now turn to the estimate of the low frequency $\|\Delta_{-1}f_\lambda\|_{L^\infty}.$ We observe that the Fourier transform of  $\Delta_{-1}f_{q,\lambda}$ vanishes when $q\geq n_0$ where $n_0$ is an absolute integer. Consequently we get
\begin{eqnarray*}
\|\Delta_{-1}f_{\lambda}\|_{L^\infty}&=&\sum_{q=-1}^{n_0}\|\Delta_{-1}f_{q,\lambda}\|_{L^\infty}\\
&\le&C\sum_{q=-1}^{n_0}\|\Delta_{q}f\|_{L^\infty}\\
&\le& C\|f\|_{B_{\infty,\infty}^s}
\end{eqnarray*}
This completes the proof of the lemma.
\end{proof}

{\bf Acknowledgments}

The author would like to thank J.-Y. Chemin, I. Gallagher and D. Li for the fruitful discussions about the critical case.

\end{document}